\documentclass[12pt,oneside]{amsart}

      \usepackage{amssymb}
      \usepackage{graphicx}
      \usepackage{dcolumn,indentfirst,cite,color,hyperref}
      \usepackage{fancyhdr}
      \theoremstyle{plain}
     \newtheorem{thm}{Theorem}[section]
\newtheorem{lem}[thm]{Lemma}
\newtheorem{p}[thm]{Problem}

\newtheorem{rem}[thm]{Remark}
\newtheorem{ex}[thm]{Example}
\newtheorem{defi}[thm]{Definition}
\newcommand{\bess}{\begin{eqnarray*}}
\newcommand{\eess}{\end{eqnarray*}}
\input xy
\xyoption{all}

      \makeatletter
      \def\@setcopyright{}
      \def\serieslogo@{}
      \makeatother

\begin{document}

% First we specify the top matter (author, title, etc).
%
% Note: All of the top matter items are optional and can be omitted.
% But you will probably want to specify at least the author and title% and perhaps an abstract.

   % author information

   % first author

  \author[X. Wang]{Xiaoguang Wang}
   
  \address{ICERM, Brown University, Box 1995, Providence, RI 02912}
  \email{wxg688@163.com}

   % second author

  % \author{Stephen G.~Simpson}

   % the address where the research was carried out
   %\address{University of G\"ottingen, G\"ottingen, Germany}

   % current address, usually not needed because it is the same as the
   % regular address
   %\curraddr{Department of Mathematics, Pennsylvania State University,
    % University Park, State College PA 16802}

   %\email{simpson@math.psu.edu}

   % title

   \title[A decomposition theorem]{A decomposition theorem for Herman maps}

   % Note that the short title for running heads goes in square
   % brackets.  This is optional.  The long title goes in curly
   % braces.  In the long title, line breaks are indicated by \\.

   % abstract (optional)
  
   \begin{abstract}
   In 1980s, Thurston
 established a topological characterization theorem for postcritically finite rational maps.
 In this paper, a
   decomposition theorem for a class of postcritically infinite branched covering termed `Herman map' is developed. It's shown that
   every  Herman map can be decomposed along a
stable multicurve into finitely many Siegel maps and Thurston maps,
such that the combinations and rational realizations of these
resulting maps essentially dominate the original one.
This result gives an answer to a problem of McMullen in a sense and
 %possibility of understanding the combinatorics and dynamics of
 %rational maps with Herman rings and
  enables us to prove a  Thurston-type theorem for rational maps with Herman rings.
   \end{abstract}

   % AMS subject classifications (used in AMS journals)
   \subjclass[2000]{Primary 37F30; Secondary 37F50, 37F10, 37F20}

   % AMS keywords (used in AMS journals)
   \keywords{decomposition theorem, Herman map, branched covering, Thurston obstruction, rational-like map, renormalization}

   % acknowledge support, etc
  % \thanks{This research was partially supported by NSF grant
   %  DOA-123456789.}
   %\thanks{We would like to thank our colleagues for their helpful
    % criticism.}

   % dedication
  % \dedicatory{Dedicated to Professor Donald Knuth on the occasion
   %  of his $100$th birthday}

   % today's date, or fill in whatever date you prefer
   \date{\today}

% This ends the top matter information.
% We can now tell LaTeX to display the top matter.

   \maketitle

% Having displayed the top matter, we now proceed to the body of the
% article.

% The body of the article is divided into sections.
% Each section begins with a \section command.

\section{Introduction}

In 1980s,  Douady and Hubbard \cite{DH1} revealed the complexity of the family of quadratic polynomials. Contemporaneously, Thurston's 3-dimensional insights
revolutionized the theory of Kleinian group \cite{Th1}.
After then, Sullivan \cite{Su} discovered  a dictionary between these two objects. Applying quasiconformal method to rational maps, he translated the Ahlfors' finiteness theorem into a solution of a long-outstanding problem of wandering domains.

Based on Sullivan's dictionary, McMullen asked a question: Is there a 3-dimensional geometric  object naturally associated to a rational map?
For example, it's known that Haken manifolds have a hierarchy,  where they can be split up into 3-balls along incompressible surfaces.
  McMullen suggested to translate the Haken theory on cutting along general incompressible subsurfaces into a theory for rational maps with disconnected Julia sets. He posed the following problem (\cite{McM1}, Problem 5.4):

\begin{p}[McMullen] Develop decomposition and combination theorems for rational maps.\end{p}
In this article, we aim to answer this problem in a sense. We will develop a decomposition theorem
   for  rational maps with  Herman rings, or more generally for `Herman maps'. Roughly speaking, a Herman map is a postcritically infinite branched covering with `Herman rings' and postcritically finite outside the closure of all rotation domains. We will show that a Herman map can always be decomposed along a stable multicurve into two kinds of `simpler' maps --Siegel maps and Thurston maps, such that the  combinations and rational realizations of these
 `simpler' maps  essentially dominate the original one.  Here, roughly, a Siegel map is a postcritically infinite branched covering with `Siegel disks' and postcritically finite outside the closure of all `Siegel disks', a Thurston map is a postcritically finite branched covering.
  The precise formulation of the decomposition theorem requires a fair number of definitions and is put in the next section.

 % See next section for formal definitions.

 The significance of the decomposition theorem is that it gives a way to  extend  Thurston's theory beyond postcritically finite setting. The theory deals with the following problem: Given a branched covering, when it is equivalent (in a proper sense) to a rational map?   Thurston \cite{Th2} gave a complete answer to this problem for postcritically finite  cases
  in 1980s by showing that such a  map either is equivalent to an essentially unique rational map    or contains a `Thurston obstruction'. Here, an obstruction is a collection of Jordan curves such that a certain associated matrix has leading eigenvalue greater than 1.
    The  detailed proof of Thurston's theorem is given  by
  Douady and Hubbard \cite{DH} in 1993.  The insights  produce many new, sometimes unexpected applications in complex dynamics \cite{BFH,Br,C,CT1,E,Ge,Go,HSS,P,R1,R2,R3,Se,ST,T1,T2}...
    Since then, many people have tried to extend Thurston's theorem beyond postcritically finite rational maps. Recently, progress has made for several families of holomorphic maps. For  example, Hubbard, Schleicher and Shishikura \cite{HSS} extended Thurston's theorem to
     postsingularly  finite  exponential maps $\lambda e^z$; Cui and Tan \cite{CT1}, Zhang and Jiang  \cite{ZJ}, independently, proved a
     Thurston-type theorem for
     hyperbolic rational maps; other extensions include geometrically finite rational maps and rational maps with Siegel disks \cite{CT2,Z2},...

 The  decomposition theorem enables us to extend Thurston's theorem  to rational maps with Herman rings. More generally,  we have:

 {\it Thurston-type theorems for  rational maps with Herman rings can
be reduced to Thurston-type theorems for  rational maps with Siegel
disks.}

Besides, the decomposition theorem reveals an analogue  between Haken manifolds and Herman maps (compare \cite{McM1,P}):

\vskip 0.2cm

{\small \vskip 2pt \tabcolsep 0.02in
{\center \begin{tabular}{cc}
\hline
 Haken manifold & Herman map \\
[4pt]
\hline

 incompressible surfaces (I.S.)&  stable multicurve  \\

%(non-peripheral) incompressible torus &  \\

cutting along I.S. & decomposition along stable multicurve \\

%split along incompressible surfaces & decompose along stable multicurve \\

 %\hline
  finite procedure to find an I.S.
&finite iterate to get a stable multicurve \\

%\hline
  resulting pieces are 3-balls & resulting maps are Siegel/Thurston maps\\

  reducing theorems to 3-balls& reducing theorems to Siegel/Thurston maps\\

combination theorem & build up rational map from  renormalizations\\

% \hline
rigidity theorem & rigidity theorem\\

 hyperbolic structure & rational realization
\\[5pt]
 \hline
\end{tabular}}
    
% Sections can be labeled for cross referencing.

\section{Definitions and main theorems}\label{c11}

 Let $S^2$ be the two-sphere and
$f:S^2\rightarrow S^2$ be an orientation preserving branched
covering of degree at leat two. We denote by $\deg(f,x)$  the local
degree of $f$ at $x\in S^2$. The critical set $\Omega_f$ of $f$ is
defined by
$$\Omega_f=\{x\in S^2; \deg(f,x)>1\},$$ and the  postcritical set
$P_f$ of $f$ is defined by
$$P_f=\overline{\bigcup_{n\geq1}f^n(\Omega_f)}.$$

We say that $f$ is {\it postcritically finite } if $P_f$ is a finite set. Such a map is also called a
{\it Thurston map}. For a Thurston map, we define a function $\nu_f:
S^2\rightarrow \mathbb{N}\cup\{\infty\}$ in the following way: For
each $x\in S^2$, define $\nu_f(x)$ (may be $\infty$) as the least
common multiple of the local degrees $\deg(f^n,y)$ for all $n>0$ and
all $y\in S^2$ such that $f^n(y)=x$. Note that $\nu_f(x)=1$ if
$x\notin P_f$. We call $\mathcal{O}_f=(S^2, \nu_f)$ the {\it orbifold } of
$f$.

In the following, we will define two classes of postcritically infinite branched covering--Herman map and Siegel map--step by step. These maps are branched coverings
with `rotation domains' and postcritically finite elsewhere. Since the branched covering $f$ will be required to have certain
 regularity (e.g. `holomorphic' or `quasi-regular') in some parts of the two-sphere $S^2$,
 it is reasonable to equip $S^2$ with a complex structure.
For this,  we will identify $S^2$ with $\mathbb{\overline{C}}=\mathbb{C}\cup\{\infty\}$ in our discussion.

\begin{defi}[Rotation domain]\label{def3-1}
 We say $\langle U_0,\cdots, U_{p-1}\rangle$
  is a cycle of rotation disks (resp. annuli) of $f$ if

1. All  $U_i$ are disks (resp. annuli), with disjoint closures. Each
boundary component of  $U_i$ is a Jordan curve.

 2. $f$ should induce conformal isomorphisms
$U_0\xrightarrow[]{\cong} U_1 \xrightarrow[]{\cong} \cdots\xrightarrow[]{\cong}
U_{p-1}\xrightarrow[]{\cong} U_p=U_0$ and  the return map
$f^p:U_0\rightarrow U_0$ is conformally conjugate to an irrational
rotation $z\mapsto e^{2\pi i \theta}z$.

%That is, there is a conformal map $\varphi: U_0\rightarrow V
%(=\mathbb{D} \text{ or } \{z\in\mathbb{C}; 1<|z|<R\})$ such that
%$\varphi\circ f^p \circ\varphi^{-1}(z)=e^{2\pi i \theta}z$ for some
%$\theta\in \mathbb{R}\setminus \mathbb{Q}$.

3. Each boundary cycle of $U_i$ contains at least one
critical point of $f$.
\end{defi}

By definition, two different cycles of rotation domains have disjoint closures.
Moreover, in the case that all   $U_i$ are annuli, the union $\cup_{0\leq j<p}\partial U_j$ consists of  two cycles of boundary curves,
 thus there are at least
two critical points on $\cup_{0\leq j<p}\partial U_j$.

One may compare this definition with the definitions of Siegel disks
and Herman rings  for rational maps in \cite{M}. In general,
 for rational maps, whether the boundary of a rotation domain  %is a quasi-circle and whether it
  contains a critical point   depends on the rotation number. It's known from Graczyk and Swiatek
   \cite{GS} that for any rational map, the boundary of a Siegel disk (or Herman ring) with bounded type rotation number always contains a critical
    point. On the other hand, there exist quadratic polynomials with Siegel disks
    whose boundaries do not contain any critical point (see \cite{H} or \cite{ABC}).
   We remark that in Definition \ref{def3-1}, whether the boundary of a rotation domain contains a critical point is not essential,
     we need such an assumption simply because we want to  concentrate on the  combination of the
   branched covering rather than
     the complexity caused by the rotation number. Our method can be easily generalized.
     
     \begin{defi} [Herman  map and Siegel map]\label{def3-3}  %\textbf{Rotation domain}
We  say that $f$ is a Herman map if $f$  has at least one cycle of rotation annuli
 and postcritically finite outside the union of all rotation domains; a Siegel map if $f$ has
  at least one cycle of rotation disks
% but no rotation annulus,
  and   postcritically finite outside the union of all rotation disks.
\end{defi}

Note that   a Herman map may have rotation disks and a Siegel map has no rotation annuli.

The category of branched covering consisting of  Herman maps, Siegel maps and Thurston maps are called {\it HST maps}.  Namely,
a HST map is an orientation preserving branched cover
such that each critical orbit either  is finite   or meets  the closure of
some rotation domain (if any).  Given a HST map $f$, let $n_{RD}(f)$ be the number of rotation
disk cycles, $n_{RA}(f)$ be the number of rotation annulus cycles, they satisfy $n_{RD}(f)+2n_{RA}(f)\leq 2 \mathrm{deg}(f)-2.$
Obviously, a Thurston map $f$ is a HST map with $n_{RD}(f)=n_{RA}(f)=0$. The union of all rotation domains
of $f$ is denoted by $R_f$ (probably empty).
     
\begin{defi} [Marked set]\label{def3-4} %\textbf{Rotation domain}
Let $f$ be a HST map. A marked
set $P$ is a compact set that satisfies the following:

1. $f(P)\subset P$.

2. $P\supset P_f\cup \overline{R_f}$ and $P-(P_f\cup
\overline{R_f})$ is a finite set.
\end{defi}

In this paper, we always use a pair $(f,P)$, a branched covering
together with a marked set, to denote a HST map.

\begin{defi}[C-equivalence and q.c-equivalence]\label{def3-4-a}
Two HST maps $(f,P)$  and $(g,Q)$ are called combinatorially
equivalent or `c-equivalent' for short (resp. q.c-equivalent),  if
there is a pair $(\phi,\psi)$ of homeomorphisms (resp.
quasi-conformal maps)  of $\mathbb{\overline{C}}$ such that

1. $\phi\circ f=g\circ \psi$ and $\phi(P)=Q$.

2. $\phi$ and $\psi$ are  holomorphic in   $R_f$.

3. $\phi$ and $\psi$ are isotopic  rel $P$. That is, there is a
continuous map $H:[0,1]\times \mathbb{\overline{C}}\rightarrow
\mathbb{\overline{C}}$ such that for any $t\in [0,1 ]$,
$H(t,\cdot):\mathbb{\overline{C}}\rightarrow \mathbb{\overline{C}}$
is a homeomorphism, $H(0,\cdot)=\phi,
H(1,\cdot)=\psi$ and $H(t,z)=\phi(z)$ for any $t\in [0,1]$ and any
$z\in P$.
\end{defi}
 In this case,
we say that $(f,P)$ is {\it c-equivalent} (resp. {\it q.c-equivalent}) to $(g,Q)$
via $(\phi,\psi)$. If $(g,Q)$ is a rational map, we call $(g,Q)$ a {\it rational realization}  (resp. {\it q.c-rational realization})
of $(f,P)$. Note that if $(f,P)$ has  a q.c-rational realization, then  $(f,P)$ is
necessarily a quasi-regular map\footnote{A {\it quasi-regular map}
 is locally a composition of a holomorphic map and a quasi-conformal map.}.

 A Jordan curve $\gamma\subset\mathbb{\overline{C}}\setminus P$ is called {\it null-homotopic}, (resp. {\it  peripheral}) if a component of
$\mathbb{\overline{C}}\setminus \gamma$ contains no (resp. one) point of $P$; {\it non-peripheral} (or {\it essential}) if each component of
 $\mathbb{\overline{C}}\setminus \gamma$  contains at least two points
of $P$.

\begin{defi}[Multicurve and Thurston obstruction]\label{def3-4-b}
A multicurve $\Gamma=\{\gamma_1, \cdots, \gamma_n\}$ is a
 collection of finite non-peripheral, disjoint, and  no two homotopic Jordan curves in $\mathbb{\overline{C}}\setminus P$.
 Its $(f,P)$-transition matrix
$W_\Gamma=(a_{ij})$ is defined by
$$a_{ij}=\sum_{\alpha\sim
\gamma_i}\frac{1}{\mathrm{deg}(f:\alpha\rightarrow\gamma_j)},$$
where the sum is taken over all the components $\alpha$ of
$f^{-1}(\gamma_j)$ which are homotopic to $\gamma_i$ in
$\mathbb{\overline{C}}\setminus P$.

A multicurve $\Gamma$ in $\mathbb{\overline{C}}\setminus P$ is
called  $(f,P)$-{\it stable} if each non-peripheral component of $f^{-1}(\gamma)$
for $\gamma\in \Gamma$ is
homotopic in $\mathbb{\overline{C}}\setminus P$  to a curve
$\delta\in\Gamma$.

We say that a multicurve $\Gamma$ is a {\it Thurston obstruction} if $\Gamma$ is $(f,P)$-stable and the leading eigenvalue\footnote{The
 {\it leading  eigenvalue} of  a square matrix $A$ is the eigenvalue with largest modulus. It's known that if
 $A$ is non-negative (i.e. each entry is non-negative), then  its   leading  eigenvalue is real and non-negative.}
$\lambda(\Gamma,f)$ of its transition matrix $W_\Gamma$ satisfies
$\lambda(\Gamma,f)\geq1$.
\end{defi}

For convention, an empty set $\Gamma=\emptyset$ is always considered
as a  $(f,P)$-stable multicurve with $\lambda(\Gamma,f)=0$.
 A map without Thurston obstructions is  called an unobstructed map. Else, it is called an obstructed map.

The following characterization theorem, due to Thurston, is fundamental in complex dynamical systems:

\begin{thm}[Thurston, \cite{DH,Th2}]\label{b3-3-a}  Let $(f,P)$ be a Thurston map. Suppose that
$\mathcal{O}_f$ does not have signature $(2,2,2,2)$. Then $(f,P)$ is
c-equivalent to a rational function $(R,Q)$ if and only if
$(f,P)$ has no Thurston obstructions.
The rational function $(R,Q)$ is unique up to M\"obius conjugation.
\end{thm}

The original version ($P=P_f$) of Thurston's theorem is proven by Douady and Hubbard  \cite{DH}, while the `marked' version is
proven in \cite{BCT}.

To extend Thurston's theorem to rational maps with Herman rings, we establish the following main theorem of the  paper:

\begin{thm}[Decomposition  theorem]\label{b3-1}  Let $(f,P)$ be a Herman map, then there exist a
 $(f,P)$-stable multicurve $\Gamma$  and finitely many
Siegel maps and Thurston maps, say $(h_k,P_k)_{k\in \Lambda}$,
where $\Lambda$ is a finite index set, such that

1. (Combination part.) $(f,P)$ has no Thurston obstructions if and
only if $\lambda(\Gamma,f)<1$  and for each $k\in\Lambda$,
$(h_k,P_k)$ has no Thurston obstructions.

2. (Realization part.) $(f,P)$ is c-equivalent to a rational
map  if and only if $\lambda(\Gamma,f)<1$
and for each $k\in\Lambda$, $(h_k,P_k)$ is c-equivalent to a
rational map.
\end{thm}

Theorem \ref{b3-1} actually answers a problem of McMullen (\cite{McM1}, Problem 5.4) in a sense. It gives a way to understand Herman maps (obstructed or not), in particular rational maps with Herman rings,
in terms of the simpler ones
$(h_k,P_k)$'s. The theorem develops a theory for rational maps,  parallel to the Haken theory for manifolds.
 It is known that Haken manifolds can be split up into 3-balls along incompressible surfaces. On the other hand, combination theorems of Klein and Maskit allows one to build up a Klein group with disconnected limit set from a number of subgroups with connected limit sets (see \cite{Ma,McM1}). The decomposition theorem translates this theory to Herman maps in the following way:
one can decompose a Herman map along  a stable multicurve  into  several Siegel maps and Thurston maps whose combinations and rational realizations dominate the original one. Conversely,  one can  rebuild a rational map with  disconnected Julia set from a number of renormalizations with connected Julia sets.

Let's briefly sketch how to get the  maps $(h_k,P_k)_{k\in\Lambda}$ in Theorem \ref{b3-1}.
Given a Herman map $(f,P)$, we first choose a collection of $f$-periodic
analytic curves $\Gamma_0$ in the rotation annuli  and their suitably chosen preimages $\Gamma$. The curves in
$\Gamma_0\cup\Gamma$ decompose the complex sphere into finitely many multi-connected pieces, say
$S^1$,$\cdots$, $S^\ell$. The action of $(f,P)$ on the sphere  induces a well-defined map
$$f_*:\{S^1,\cdots, S^\ell\}\rightarrow\{S^1,\cdots, S^\ell\}$$
 from these pieces to themselves. Under the map $f_*$, each piece is pre-periodic. % The periodic  pieces  are more important than the strictly pre-periodic ones.
  Each cycle of these pieces corresponds to a
 renormalization of $(f,P)$, which takes the form $f^{p_k}:E_k\rightarrow S_k$. Here, $p_k$ is a positive integer, $E_k,S_k$ are
 multi-connected domains with $E_k\subset S_k\in \{S^1,\cdots, S^\ell\}$. These renormalizations  have canonical extensions to the branched coverings of
 the sphere, which are in fact the resulting maps $(h_k,P_k)_{k\in\Lambda}$. Details are put in Section \ref{3-1}.

 \begin{rem} Here are some facts on  Theorem \ref{b3-1}:

1. The multicurve  $\Gamma$ can be an empty set. See Example \ref{b2a1}.

2. The number of the resulting Siegel maps is at least two and at most $n_{RD}(f)+2n_{RA}(f)$.
 The number of the resulting Thurston  maps can be zero (see Example \ref{b2a1}).
%while the number of Thurston maps is bounded by the irreducible rank
%of the $(f,P)$-transition matrix of $\Gamma$.

3. If $\lambda(\Gamma,f)<1$, then for any resulting Thurston map $(h_k,P_k)$, the signature of its orbifold is not
$(2,2,2,2)$ (see Lemma \ref{b3aa}). By Thurston's theorem,   $(h_k,P_k)$ has no Thurston obstructions if and only if $(h_k,P_k)$ is c-equivalent to a
rational map.
\end{rem}

 Given a HST rational map $(f,P)$, let $R_{top}(f,P)$ (resp. $R_{qc}(f,P)$) be the set of all rational maps
 c-equivalent (resp. q.c-equivalent) to
 $(f,P)$. We define the space  $M_{\omega}(f,P)$ to be $R_{\omega}(f,P)$ modulo M\"obius conjugation, where $\omega\in\{top,qc\}$.
 If $(f,P)$ is postcritically finite, one may verify that $M_{top}(f,P)=M_{qc}(f,P)$; if we further require  $(f,P)$  is not a Latt\`es map, it follows from
  Thurston's theorem (rigidity part) that $M_{top}(f,P)$ is a single point. 
In general, it's not clear whether $M_{top}(f,P)=M_{qc}(f,P)$ or whether $M_{qc}(f,P)$ consists of a single point (this problem is related to the `No Invariant Line Field Conjecture'), but we have the following:

\begin{thm}[Rigidity theorem]\label{b3-11}  Suppose that $(f,P)$ is a Herman rational map.
Then there are finitely many Siegel rational  maps $(h_k,P_k)_{1\leq k\leq m}$ such that
$$M_{qc}(f,P)\cong M_{qc}(h_1,P_1)\times\cdots\times M_{qc}(h_m,P_m).$$
\end{thm}
Theorem \ref{b3-11} is in fact the `rigidity part' of Theorem \ref{b3-1}. In particular, it implies $M_{qc}(f,P)$
is a single point if and only if all $M_{qc}(h_k,P_k)$ are single points.

For obstructed Herman maps and Siegel maps, it follows from a theorem of McMullen (see \cite{McM2} or Theorem \ref{b4b}) that they have no rational realizations.
In particular, Theorem \ref{b3-1} implies that any Thurston obstruction of  $(f,P)$   either is contained in $\Gamma$ or comes from one of the resulting maps
$(h_k,P_k)$'s.

For unobstructed Herman maps or Siegel maps, whether  they have rational realizations is a little  bit involved.
%In general, an unobstructed Herman map or  Siegel map .
For example,
we consider   the formal mating (see \cite{YZ} for the definition) of
two quadratic Siegel polynomials $f_{\theta_1}$ and
$f_{\theta_2}$ with bounded type rotation numbers, where $f_\alpha(z)=z^2+\frac{e^{2\pi
i\alpha }}{2}(1-\frac{e^{2\pi i\alpha}}{2})$. The resulting map is an unobstructed Siegel map.
If  $\theta_1+\theta_2=1 { \rm mod \ }\mathbb{Z}$, then the Siegel map has no rational realization.
On the other hand, if $\theta_1+\theta_2\neq 1 { \rm mod \ }\mathbb{Z}$,  Yampolski and Zakeri \cite{YZ} showed that the Siegel map has a unique  rational
realization.
Based on this example and  following the idea of Shishikura \cite{S1}, many unobstructed Herman maps which have no  rational realizations are constructed in \cite{W}.

 %We will see that if an un unobstructed Herman map has no rational realziations
 The following result reveals an  `equivalence' between one unobstructed Herman map and several unobstructed Siegel maps:
%It is in fact a consequence of  Theorems \ref{b3-3-a} and \ref{b3-1}:

\begin{thm}[Equivalence of rational realizations]\label{b3-111}
Given  an unobstructed  Herman map $(f,P)$, there are at most $n_{RD}(f)+2n_{RA}(f)$ unobstructed  Siegel maps $(h_k,P_k)_{1\leq k\leq m}$, such that  the
following two statements are equivalent:

1.  $(f,P)$ has a rational realization.

2.   Each map of $(h_k,P_k)_{1\leq k\leq m}$ has a  rational realization.
\end{thm}

Theorem \ref{b3-111} implies Thurston-type  theorem for Herman rational maps can be reduced to
 Thurston-type  theorem for Siegel rational maps.
 
 At last, we give a significant  application of Theorem \ref{b3-1}. It's a Thurston-type theorem for a class of rational maps with Herman rings:

 \begin{thm}[Characterization of Herman ring]\label{b3-kk}
 Let $(f,P)$ be a Herman map.
    Suppose that  $(f,P)$ has only one fixed annulus $A$ of bounded type rotation number and $P\setminus \overline{A}$ is a finite set.
  Then $(f,P)$ is
c-equivalent to a rational function $(R,Q)$ if and only if
$(f,P)$ has no Thurston obstructions.
The rational function $(R,Q)$ is unique up to M\"obius conjugation.
\end{thm}

An irrational number $\theta\in (0,1)$ is of {\it bounded type} if its continued fraction  $[a_1,a_2,\cdots]$ satisfies $\sup\{a_n\}<+\infty$.
 % Theorem \ref{b3-kk} generalizes Thurston's theorem to rational maps with Herman rings.
  The proof of Theorem \ref{b3-kk}  is based on Theorems \ref{b3-3-a} and \ref{b3-111}, and a theorem of Zhang \cite{Z2}
on characterization of a class of Siegel rational maps.

\vskip0.2cm
\noindent\textbf{Strategy of the proof and organization of the paper.}
The idea `decomposition along a stable multicurve' %give rise to a pant decomposition of marked sphere.
 was initially implicated in Shishikura's paper on Herman-Siegel  surgery \cite{S1}.
%and tree structure of Herman rings \cite{S2}.
 Cui  sketched this idea to prove a Thurston-type theorem for hyperbolic maps in his manuscript \cite{C}. Pilgrim \cite{P,P2} used this idea to develop a decomposition theorem for obstructed Thurston maps.
  In the rewritten work  \cite{CT1} of \cite{C},  Cui and Tan successfully developed this idea to prove a  characterization theorem for hyperbolic rational maps.
% where they found that a hyperbolic map can be decomposition along a stable multicurve  into several `renormalized sub-systems'.
 %comes from  Cui-Tan's paper \cite{CT1}.

Our proof more or less follows the same line as in \cite{CT1}. In both settings, we first choose a specific multicurve and use it
to decompose the complex sphere into several pieces.  The essential difference is:  in their case \cite{CT1},
 these pieces have disjoint closures   and their preimages are compactly contained in themselves; in our case,
 each of these pieces will touch several other pieces and the preimages of them may be
 not compactly contained in themselves. This  leads to several  differences in the proof, especially the
technical difference in  Section \ref{3-3-3}.%

The organization of the paper is as follows:

In  Section \ref{3-1}, we will decompose a Herman map $(f,P)$ into  a number of
Siegel maps, Thurston maps and sphere homeomorphisms $(h_k,P_k)_{1\leq k\leq n}$ along a collection of $f$-periodic
analytic curves $\Gamma_0$ (contained in the rotation annuli)  and their suitably chosen preimages $\Gamma$.

In  Section \ref{3-2}, we study the decomposition of stable multicurves. We show that each stable multicurve $\mathcal{C}$ of $(f,P)$ will
induce a submulticurve $\mathcal{C}_\Gamma$ of $\Gamma$ and a stable multicurve $\Sigma_k$ of $(h_k,P_k)$ for all $1\leq k\leq n$, such that the following identity holds:
$$\lambda(\mathcal{C},f)=\max\Big\{\lambda(\mathcal{C}_\Gamma,f),
\sqrt[p_1]{\lambda(\Sigma_1,h_1)},\cdots,\sqrt[p_n]{\lambda(\Sigma_n,h_n)}\Big\}.$$
Conversely, each stable multicurve $\Sigma_k$ of $(h_k,P_k)$ will generate a  $(f,P)$-stable multicurve $\mathcal{C}$ and a submulticurve $\mathcal{C}_\Gamma$ of $\Gamma$ satisfying the above reduction identity.
This  enables us to prove the `combination part' of Theorem \ref{b3-1}

The `realization part' will be discussed in Sections \ref{3-3} and \ref{3-3-3}. In Section \ref{3-3}, we prove the necessity and a special case
 of sufficiency of the `realization  part' of Theorem \ref{b3-1}.
  In Section \ref{3-3-3}, we prove the sufficiency of the `realization part'
 in the general case $\Gamma\neq\emptyset$.
  The crucial and technical part is to endow the algebraic condition $\lambda(\Gamma,f)<1$ with a geometric meaning. This will be done from
  Section \ref{a1} to Section \ref{a5}.
   We will show that this condition is equivalent to the Gr\"otzsch inequality (Lemma \ref{b4g}).
   Thus it allows us to reconstruct  the rational realization of $(f,P)$ by gluing the holomorphic models of $(h_k,P_k)_{1\leq k\leq n}$ along the multicurve $\Gamma$
without encountering any `gluing obstruction'.

In Section \ref{3-6}, we discuss the renormalizations of rational maps and prove  Theorem  \ref{b3-11}. A straightening theorem
  for rational-like maps  is developed in Section \ref{3-4-1}.
  In Section \ref{3-4-2-1},  we discuss  the  renormaliztions of Herman rational maps. In Section \ref{3-4-3}, we prove Theorem  \ref{b3-11} .

  Theorems \ref{b3-111} and \ref{b3-kk} are consequences of Thurston's theorem and the decomposition theorem, we put the proofs in Section \ref{3-x}.

\vskip0.2cm
\noindent\textbf{Notations and terminologies.} The following are used frequently:  

1. Given  a collection of Jordan curves $\mathcal{C}$ (not
necessarily a multicurve) in $\mathbb{\overline{C}}-P$.  For
 any integer
$k\geq0$, we denote by $f^{-k}(\mathcal{C})$ the collection of all
  components $\delta$  of $f^{-k}(\gamma)$ for $\gamma\in \mathcal{C}$.
  %Set $\cup \mathcal{C}:=\cup_{\gamma\in \mathcal{C}}\gamma$.

2. Let $\mathcal{M}$ be a collection of
subsets of $\mathbb{\overline{C}}$. We
use $\cup \mathcal{M}$ to denote $\cup_{M\in \mathcal{M}}M$.

3. Let $A=(a_{ij})$ be a square real matrix. The Banach norm
$\|A\|$ of $A$ is  defined to be  $(\sum
|a_{ij}|^2)^{1/2}$. The {\it
spectral radius} $\mathrm{sp}(A)$ of $A$ is defined by
$\mathrm{sp}(A):=\lim\sqrt[n]{\|A^n\|}$. It's known from Perron-Frobenius theorem that
 if $A$ is non-negative, then its leading eigenvalue
  is equal to $\mathrm{sp}(A)$.

4. Given two multicurves $\Gamma_1$ and $\Gamma_2$ in
$\mathbb{\overline{C}}-P$. We say that $\Gamma_1$ is {\it
homotopically contained } in $\Gamma_2$, denoted by
$\Gamma_1\prec\Gamma_2$, if each curve $\alpha\in \Gamma_1$  is
homotopic in $\mathbb{\overline{C}}-P$ to some curve
$\beta\in\Gamma_2$. We say that $\Gamma_1$ is identical to
$\Gamma_2$ up to homotopy,  if $\Gamma_1\prec\Gamma_2$ and
$\Gamma_2\prec\Gamma_1$.

5. Let $D$ and $\Omega$ be two planar  domains and  $f: D\rightarrow
\Omega$ be a quasi-regular map, the Beltrami coefficient $\mu_f$ of
$f$ is defined by $\mu_f =\frac{\partial f}{\partial
\overline{z}}/\frac{\partial f}{\partial z}.$

6. For a subset  $E$ of $\mathbb{\overline{C}}$, the
characteristic function $\chi_E:\mathbb{\overline{C}}\rightarrow\{0,1\}$ is defined by
$\chi_E(z)=1$ if $z\in E$ and  $\chi_E(z)=0$ if $z\notin E$.
%\bess
%\chi_E(z)=\begin{cases}
%1,   & \text{ if }  z\in E,\\
%0,   & \text{ if }  z\notin E.
% \end{cases}
% \eess

 7. Let $\Omega$ be a connected and  multi-connected domain in
 $\mathbb{\overline{C}}$, bounded by finitely many Jordan curves.
 We denote by  $\partial\Omega$ the boundary of $\Omega$, and
 $\partial(\Omega)$  the collection of all boundary curves
 of $\Omega$. Obviously, $\partial\Omega=\cup\partial(\Omega)$.

8. The closure and cardinality of the set $E$ are denoted by $\overline{E}$ and $\# E$ respectively.

\vskip0.2cm
\noindent {\bf Acknowledgement.} This work is a part of my thesis \cite{W}.
I would like to thank   Tan Lei for patient guidance, helpful discussions and careful reading the manuscript.
 Thanks go to  Guizhen Cui, Casten Petersen, Kevin Pilgrim,
  Weiyuan Qiu, Mary Rees,
 Mitsuhiro Shishikura and Yongcheng Yin  for  discussions or comments.
 This work was partially supported  by Chinese Scholarship  Council  and CODY network.

\section{Decompositions of Herman maps}\label{3-1}

In this section, we will decompose a Herman map into  finitely many
Siegel maps and Thurston maps along a collection of $f$-periodic
analytic curves and their suitably chosen preimages.
The idea we adopt here is  inspired  by
 Cui-Tan's work on characterizations of hyperbolic
rational maps \cite{CT1} and Shishikura's  `Herman-Siegel'
surgery \cite{S1}.

\subsection{Decomposition along a stable multicurve}\label{3-110}

Let $(f,P)$ be a Herman map, $\mathcal{A}$ be the collection of all
rotation annuli of $f$. For each $A\in \mathcal{A}$, we choose an
 analytic curve $\gamma_A\subset A$ such that $\gamma_A\cap f(P-\cup \mathcal{A})=\emptyset$ (this implies that $\gamma_A$ avoids the postcritical set and
 the images of other marked points) and
$f(\gamma_A)=\gamma_{f(A)}$. It's obvious that if $f^p(A)=A$, then
$f^p(\gamma_A)=\gamma_A$.

Let $\Gamma_0=\{\gamma_A; A\in \mathcal{A}\}$, we first show that
$\Gamma_0$ can generate a unique $(f,P)$-stable multicurve up to
homotopy.

\begin{lem} \label{b2a0} Given a choice of $\Gamma_0$, there is a $(f,P)$-stable multicurve $\Gamma$  such that:

$\bullet$ (Invariant)  For any $\gamma\in \Gamma$, we have
$f(\gamma)\in \Gamma\cup \Gamma_0$.

$\bullet$ (Maximal) $\Gamma$ represents all homotopy classes of
non-peripheral curves of $\cup_{k\geq1}f^{-k}(\Gamma_0)-\Gamma_0$ in
$\mathbb{\overline{C}}-P$.

 Moreover, the  multicurve $\Gamma$ is unique up to homotopy.
\end{lem}
\begin{proof}
First, there is a multicurve  $\Gamma_1$ in
$\mathbb{\overline{C}}-P$ such that $\Gamma_1\subset
f^{-1}(\Gamma_0)-\Gamma_0$ and  $\Gamma_1$ represents all homotopy
classes of non-peripheral curves of $f^{-1}(\Gamma_0)-\Gamma_0$.

Such $\Gamma_1$ is not uniquely chosen. But any two such multicurves
 are identical up to homotopy, thus they have the same number of curves.

For $n\geq2$, we define $\Gamma_n$ inductively in the following way:

$\bullet$   $\Gamma_n\subset f^{-1}(\Gamma_{n-1})$.

$\bullet$   $\Gamma_1\cup\cdots\cup \Gamma_n$ is a multicurve in
$\mathbb{\overline{C}}-P$.

$\bullet$   $\Gamma_1\cup\cdots\cup \Gamma_n$ represents all
homotopy classes of non-peripheral curves of
$f^{-n}(\Gamma_0)-\Gamma_0$.

Since any two distinct curves in
$\cup_{k\geq1}f^{-k}(\Gamma_0)-\Gamma_0$ are disjoint
% or homotopic to each other in $\mathbb{\overline{C}}-P$
and $P$ has finitely many components, we conclude that
$\cup_{k\geq1}f^{-k}(\Gamma_0)-\Gamma_0$ has finitely many homotopy
classes of non-peripheral curves in $\mathbb{\overline{C}}-P$. It
turns out that
% the number of homotopy classes of
$\#(\Gamma_1\cup\cdots\cup \Gamma_n)$ is uniformly bounded above by
some constant $C(P)$. Thus there is an integer $N\geq0$ such that
$\Gamma_{N}\neq\emptyset$ and
$\Gamma_{N+1}=\Gamma_{N+2}=\cdots=\emptyset$. (It can happen that
$N=0$, see Example \ref{b2a1}.)

We set $\Gamma=\emptyset$ if $N=0$ and  $\Gamma=\cup_{1\leq j\leq
N}\Gamma_j$ if $N\geq1$.
 By the choice of $N$,   $\Gamma$
is a $(f,P)$-stable multicurve. By construction, for any $\gamma\in
\Gamma$, we have $f(\gamma)\in \Gamma\cup \Gamma_0$. The homotopy
classes of $\Gamma$ is uniquely determined by those of
non-peripheral curves in $\cup_{k\geq1}f^{-k}(\Gamma_0)-\Gamma_0$.
So $\Gamma$ is unique up to homotopy.
\end{proof}

Here we give an example to show that $\Gamma$ can be an empty set.

\begin{ex} \label{b2a1} {\bf ($\Gamma=\emptyset$)}    The example is from
Shishikura's paper \cite{S1}. Let
$$f(z)=\frac{e^{i\alpha}}{z}\Big(\frac{z-r}{1-rz}\Big)^2,$$
where $\alpha\in\mathbb{R}$  and $0<r<1/5$. We may assume that
$\alpha$ is properly chosen such that $f$ has a fixed Herman ring
$H$  containing  the unit circle $\mathbb{S}$, with bounded type
rotation number (Remark: in this case, each boundary component of
$H$ is a quasi-circle containing a critical point of $f$). There are
two other critical points: $r$ and $1/r$, both of which are
eventually mapped to a repelling cycle of  period two, and
$f(r)=f^3(r)=0, f^2(r)=f(1/r)=\infty$. We choose
$\Gamma_0=\{\mathbb{S}\}$.
 Let $P=\overline{H}\cup P_f=\overline{H}\cup \{0,\infty\}$. Since each component of
$\mathbb{\overline{C}}-\overline{H}$ is a disk containing exactly
one marked point in $P$, the set $\Gamma$ is necessarily empty.
\end{ex}

Let $\Sigma=\Gamma_0\cup\Gamma$. In the following, we will use
$\Sigma$ to decompose the complex sphere $\mathbb{\overline{C}}$
into finitely many pieces. We define \bess &
\mathcal{S}=\{\overline{U};\  U \text{ is a
connected component of } \mathbb{\overline{C}}-\cup\Sigma\},&\\
&\mathcal{E}=\{\overline{V};\  V \text{ is a connected component of
} \mathbb{\overline{C}}-\cup f^{-1}(\Sigma)\}.& \eess

Each element of $\mathcal{S}$ (resp. $\mathcal{E}$) is called an
$\mathcal{S}$-piece (resp. $\mathcal{E}$-piece). The following facts
are easy to verify:

$\bullet$ {\it Every $\mathcal{E}$-piece $E$ is contained in a
unique $\mathcal{S}$-piece and $f(E)\in \mathcal{S}$.}

$\bullet$ {\it For  every $\mathcal{S}$-piece $S$, we have $\#(S\cap
P)+\#\partial(S) \geq3$. Moreover, the $\mathcal{E}$-pieces
contained in $S$ form a partition of $S$, that is, $S=\cup
\{E\in\mathcal{E};E\subset S\}$.}

$\bullet$ {\it  For each curve $\gamma\in \Sigma$, there exist
exactly two $\mathcal{S}$-pieces, say $S_\gamma^+$ and $S_\gamma^-$,
that share $\gamma$ as a common boundary component.}

\begin{defi} Let $T$ be a connected and closed
subset of some  $\mathcal{S}$-piece $S$.  We say that $T$ is parallel to $S$ if $\partial T\cap
P=\emptyset$ and each component of $S\setminus T$ is  either
 an annulus contained in  $S-P$, or  a disk containing at most one point in $P$.
\end{defi}

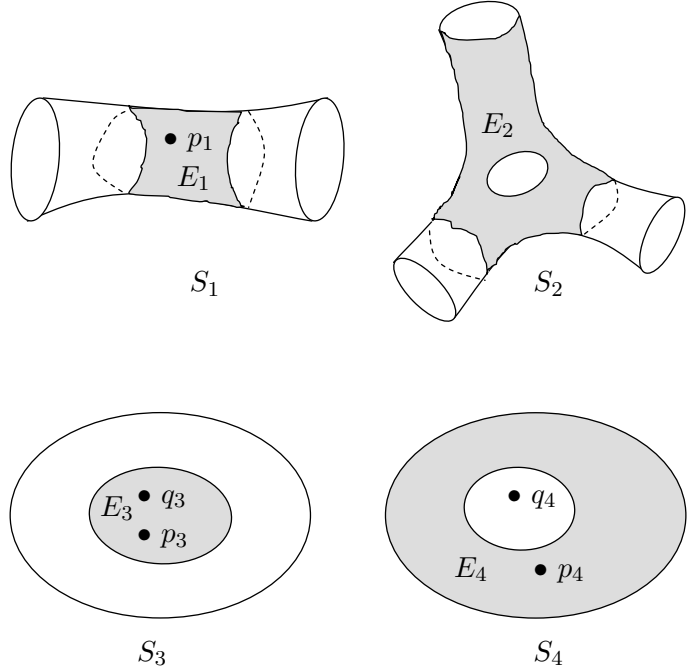
\begin{figure}[h]
\centering{
\includegraphics[height=9cm]{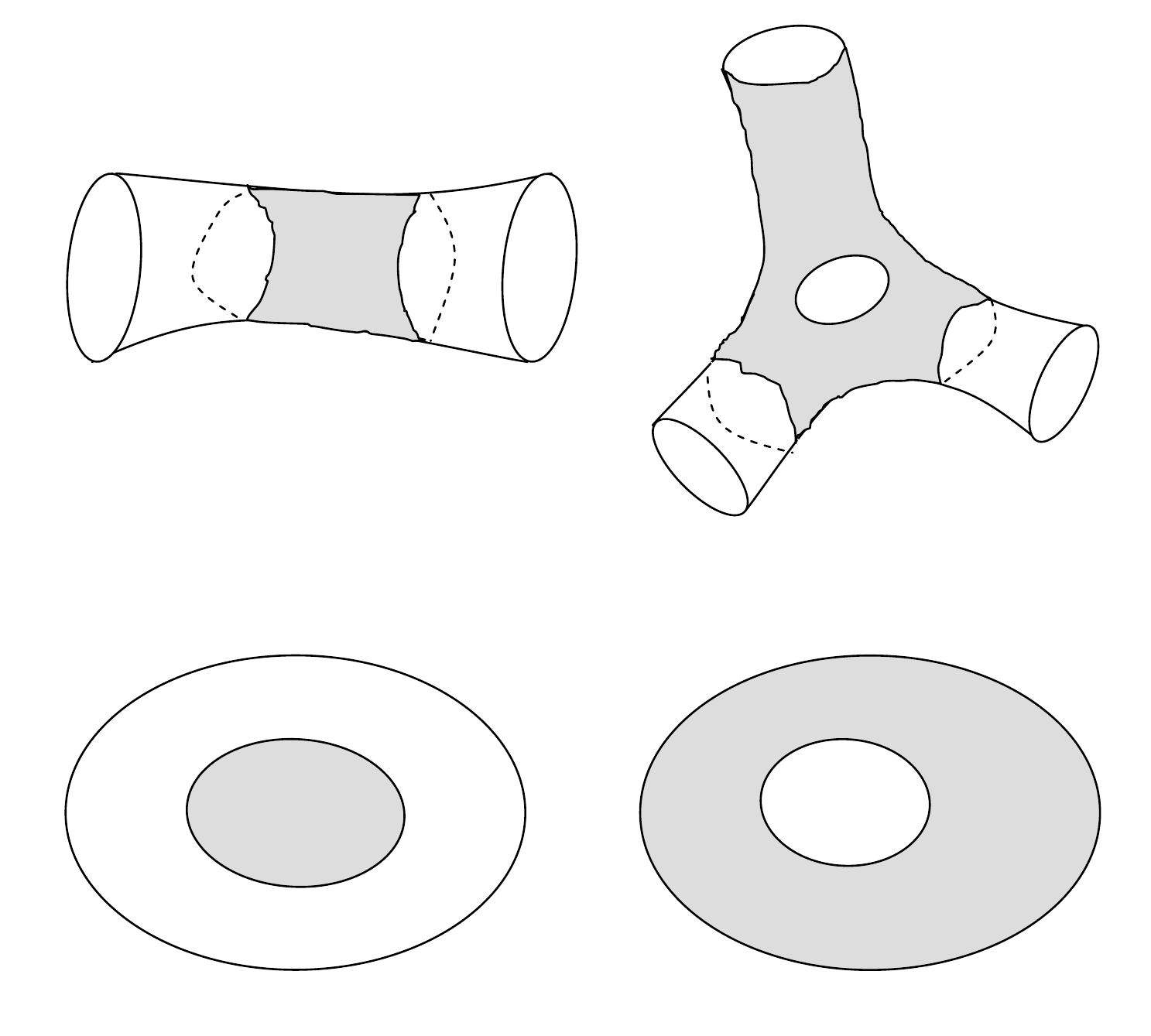}
\put(-210,195){$\bullet \ p_1$} %\put(-250,180){$\bullet \ q_1$}
\put(-205,180){$E_1$}  \put(-200,140){$S_1$}
 \put(-90,200){$E_2$}
\put(-70,140){$S_2$}
 \put(-234,55){$E_3$}
\put(-220,0){$S_3$} \put(-220,45){$\bullet \ p_3$}
\put(-220,60){$\bullet \ q_3$} \put(-100,32){$E_4$}
\put(-70,0){$S_4$} \put(-70,32){$\bullet \ p_4$}
\put(-80,60){$\bullet \ q_4$}
 \caption{Four examples: $E_i$ (shadow region) is parallel to
 $S_i$.  $p_i,q_i$ are marked points in $P$. Here, $S_1$ is an annulus with
 one marked point,  $S_2$ has three boundary curves and contains no
 marked point, both $S_3$ and $S_4$ are disks with two marked
 points.}}
\end{figure}%\label{fig5-1}

Note that if $T$ is parallel to $S$ and $A$ is an annular
component of $S\setminus T$, then one boundary curve of $A$ is on
$S$. Moreover, $\#(T\cap P)+\#\partial(T)\geq \#(S\cap
P)+\#\partial(S)\geq3$.

Here is an important property of the $\mathcal{S}$-pieces:

\begin{lem} \label{b2a2} For every $\mathcal{S}$-piece $S$, there is a unique
$\mathcal{E}$-piece parallel to $S$.
\end{lem}

\begin{proof} Let $\mathcal{C}_1$ be the collection of all curves of
$f^{-1}(\Sigma)$, contained in the interior of $S$ and non
peripheral in $\mathbb{\overline{C}}-P$. Since $\Gamma$ is
$(f,P)$-stable, each curve $\gamma\in \mathcal{C}_1$ is homotopic in
$\mathbb{\overline{C}}-P$ to exactly one boundary curve, say
$\tau_\gamma$, of $S$. Let $A(\gamma)$ be the open annulus bounded
by $\gamma$ and $\tau_\gamma$. Since distinct curves in
$f^{-1}(\Sigma)$ are disjoint, we conclude that for any two curves
$\gamma_1,\gamma_2\in \mathcal{C}_1$, the annuli $A(\gamma_1)$ and
$A(\gamma_2)$ either are disjoint or one contains another. Thus
$\cup_{\gamma\in \mathcal{C}_1}A(\gamma)$ consists of finitely many
annular components.

Let $\mathcal{C}_2$ be the collection of all curves of
$f^{-1}(\Sigma)$, contained in $S-\cup_{\gamma\in
\mathcal{C}_1}(A(\gamma)\cup\tau_{\gamma})$, peripheral or null
homotopic  in $\mathbb{\overline{C}}-P$. Then each curve $\alpha\in
\mathcal{C}_2$ bounds an open disk $D(\alpha)\subset S$. Moreover,
for any two curves $\alpha_1,\alpha_2\in \mathcal{C}_2$, the disks
$D(\alpha_1)$ and $D(\alpha_2)$ either are disjoint or one contains
another. Thus $\cup_{\alpha\in \mathcal{C}_2}D(\alpha)$ consists of
finitely many disk components.

The set $E_S:=S-\cup_{\gamma\in
\mathcal{C}_1}(A(\gamma)\cup\tau_{\gamma})-\cup_{\alpha\in
\mathcal{C}_2}D(\alpha)$ is a closed and non-empty set. It is
connected since  each component of $\mathbb{\overline{C}}-E_S$ is a
disk. The interior of $E_S$ contains no curve of $f^{-1}(\Sigma)$,
thus it is an $\mathcal{{E}}$-piece. It is in fact the unique
$\mathcal{{E}}$-piece parallel to $S$ by construction.
\end{proof}

See Figure 1 for the examples of `parallel' pieces. In the following
discussion, we always use $E_S$ to denote the $\mathcal{E}$-piece
parallel to $S$.

Based on Lemma \ref{b2a2}, we see that $(f,P)$ induces a
well-defined  map $f_*$ from $\mathcal{S}$ to itself:
$$f_*:\mathcal{S}\ni S\mapsto f(E_S)\in \mathcal{S}.$$

Since there are finitely many $\mathcal{S}$-pieces, every
$\mathcal{S}$-piece is pre-periodic.

% Recall that $E_S$ be the unique $\mathcal{E}$-piece that is parallel to $S$.
For each curve $\gamma\in\partial(S)$, there is a unique boundary
curve $\beta_\gamma\in \partial (E_S)$ such that either
$\beta_\gamma=\gamma$, or  $\beta_\gamma$ and $\gamma$ bound an
annulus in $S-P$. We define three
 sets
$\partial_0(S),
\partial_1(S), \partial_2(S)$ as follows: \bess
\partial_0(S)&=&\{\gamma\in\partial(S); \gamma\in\Gamma_0\},\\
\partial_1(S)&=&\{\gamma\in\partial(S); \gamma\ne\beta_\gamma\},\\
\partial_2(S)&=&\{\gamma\in\partial(S); \gamma=\beta_\gamma\}-\Gamma_0. \eess

\begin{lem} \label{b2a} If $\partial_0(S)\neq\emptyset$, then we have:

1. For any $\gamma\in \partial_0(S)$, $\gamma=\beta_\gamma$.

 2. $S$
is $f_*$-periodic.

3. $\#\partial_0(S)=\#\partial_0(f_*(S))$.
\end{lem}
\begin{proof}
1. Note that each component of $S- E_S$ is either a disk containing
at most one point in $P$, or an annulus in
$\mathbb{\overline{C}}-P$. It follows that  if
$\gamma\in\partial_0(S)$, then $\gamma\subset P$ and
$\gamma=\beta_\gamma$.

2. Take a curve $\gamma\in \partial_0(S)$ and  let $A_\gamma\in
\mathcal{A}$ be the rotation annulus  containing $\gamma$. Then from
1 we see that $S\cap A_\gamma=E_S\cap A_\gamma$. This implies
$f(S\cap A_\gamma)=f_*(S)\cap f(A_\gamma)$. Let $k\geq1$ be the
period of $A_\gamma$. Then we have $S\cap A_\gamma=f^k(S\cap
A_\gamma)=f^k_*(S)\cap f^k(A_\gamma)=f^k_*(S)\cap A_\gamma $. Thus
$f^k_*(S)=S$  and the period of $S$ is a divisor of $k$.

3. It follows from 1 that if $\gamma\in \partial_0(S)$, then
$f(\gamma)\in \partial_0(f_*(S))$. So
$\#\partial_0(S)\leq\#\partial_0(f_*(S))\leq \cdots$. Since $S$ is
$f_*$-periodic (by 2), we have
$\#\partial_0(S)=\#\partial_0(f_*(S))$.
\end{proof}

 It follows from Lemma \ref{b2a} that $\partial_i(S), i\in\{0,1,2\}$ are mutually disjoint and $\partial(S)=\partial_0(S)\sqcup
\partial_1(S)\sqcup\partial_2(S)$.

\begin{rem}
Suppose $\partial_0(S)\neq\emptyset$. For each $\gamma\in
\partial_0(S)$, let ${\rm per}(\gamma)$ be the period of $\gamma$. From  Lemma
\ref{b2a} we see that the $f_*$-period of $S$ is a devisor of ${\rm
gcd}\{{\rm per}(\gamma); \gamma\in\partial_0(S)\}$.
% the greatest common devisor of the numbers $\{{\rm per}(\gamma);\gamma\in\partial_0(S) \}$
 In particular, if ${\rm gcd}\{{\rm
per}(\gamma); \gamma\in\partial_0(S)\}=1$, then $f_*(S)=S$ and for
any $ \gamma\in\partial_0(S)$ and any $k\geq0$, we have
$f^k(\gamma)\in \partial_0(S)$.

For example, suppose that $(f,P)$ has two cycles of rotation annuli
whose periods are different prime numbers, say $p$ and $q$. If
$\partial_0(S)\neq\emptyset$, then $\#\partial_0(S)$ takes only four
possible values: 1, $p$, $q$ and $p+q$.
\end{rem}

\subsection{Marked disk extension}\label{3-11}

For each $\mathcal{S}$-piece $S$, we denote by
$\mathbb{\overline{C}}(S)$ the Riemann sphere containing $S$. We
always consider that different  $\mathcal{S}$-pieces are embedded
into different copies of   Riemann spheres.

In the following, we will
%will mark a point in every component of
%$\mathbb{\overline{C}}(S)-S$, and
 extend $f|_{E_S}$ to a  branched
covering $H_S: \mathbb{\overline{C}}(S)\rightarrow
\mathbb{\overline{C}}(f_*(S))$ with $\deg(H_S)=\deg(f|_{E_S})$.
The extension is canonical and unique up to c-equivalence.  If $f$ is  quasi-regular, we may also require
$H_S$ is quasi-regular.
 To
do this, we need to define the map
$H_S:\mathbb{\overline{C}}(S)-E_S\rightarrow
\mathbb{\overline{C}}(f_*(S))-f_*(S)$ such that $H_S|_{\partial
E_S}=f|_{\partial E_S}$. We will define $H_S$ component by
component.

Note that each  component of $\mathbb{\overline{C}}(S)-E_S$ is a
disk. Let $U$ be such a component with boundary curve $\gamma$.

We first deal with the case when $\gamma\in
\partial_0(S)$.
% the component of $\mathbb{\overline{C}}(S)-S$ whose
%boundary is a periodic cure of $f$. This means
%$\partial_0(S)\neq\emptyset$.
In this case, there is a rotation annulus $A_\gamma$ containing
$\gamma$. Let $k\geq1$ be the period of $A_\gamma$ and $\phi_0:
S\cap A_\gamma\rightarrow \mathbb{A}_R:=\{z\in \mathbb{C};
1<|z|<R\}$ be the conformal map such that $\phi_0 f^k
\phi^{-1}_0(z)=e^{2\pi i \theta}z$ for $z\in \mathbb{A}_R$. For
$1\leq j\leq k-1$, we define a conformal map from $f^j(S\cap
A_\gamma)$ onto $\mathbb{A}_R$ by $\phi_j=\phi_0 f^{k-j}|_{f^j(S\cap
A_\gamma)}$. Then we have the following commutative diagram

 $$
\xymatrix{& S\cap A_\gamma \ar[d]_{\phi_0} \ar[r]^{f} &f(S\cap
A_\gamma)\ar[r]^{f}\ar[d]_{\phi_1}
 & \cdots \ar[r]^{}
&f^{k-1}(S\cap A_\gamma)\ar[r]^{f}\ar[d]_{\phi_{k-1}}
 & S\cap A_\gamma\ar[d]^{\phi_0}\\
& \mathbb{A}_R\ar[r]_{z\mapsto e^{2\pi i \theta}z}
 &\mathbb{A}_R\ar[r]_{id}
& \cdots \ar[r]_{}& \mathbb{A}_R\ar[r]_{id}   & \mathbb{A}_R
 }
$$

Let $\mathbb{D}_R=\{z\in \mathbb{C}; |z|<R\}$. For $0\leq j<k$, we
consider the disk $\Delta_j$ obtained by gluing $f^j(S\cap
A_\gamma)$ and $\mathbb{D}_R$ via the map $\phi_j$. The disk
$\Delta_j$ inherits a natural complex structure from $\mathbb{D}_R$
since $\phi_j$ is holomorphic.

\begin{figure}[h]
\begin{center}
\includegraphics[height=10cm]{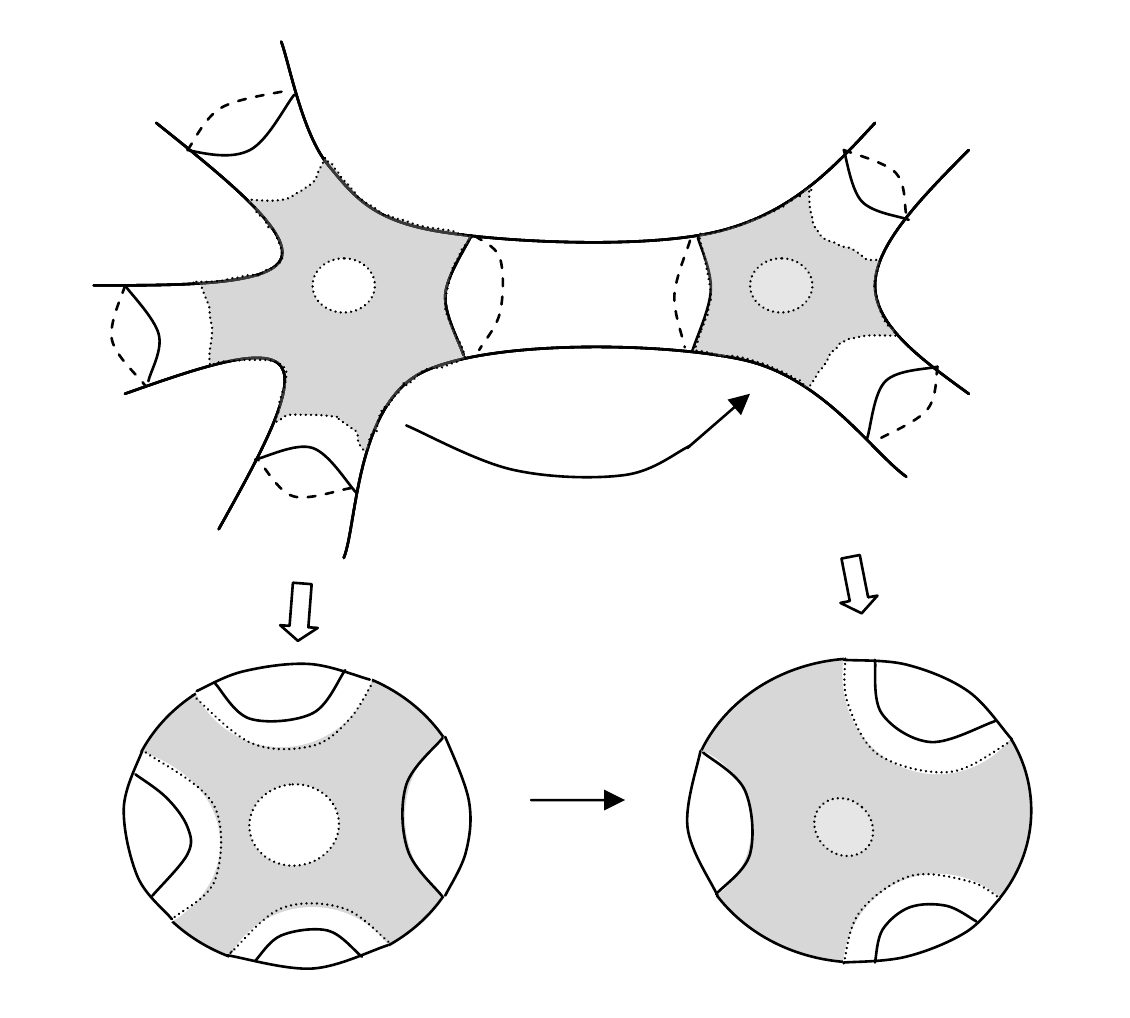}
 \put(-70,170){$\gamma_6$}\put(-70,240){$\gamma_7$}
\put(-140,200){$\gamma_5$}
\put(-160,160){$f_*$}\put(-165,200){$\bullet$}
\put(-190,200){$\gamma_4$}\put(-290,190){$\gamma_2$}\put(-210,233){$S$}
\put(-250,195){$E_S$} \put(-225,200){$\bullet$}
\put(-110,190){$E_{f_*(S)}$}\put(-130,233){$f_*(S)$}
\put(-265,260){$\gamma_1$}\put(-240,150){$\gamma_3$}
\put(-70,20){$\gamma_6$}\put(-60,20){$\bullet$}\put(-65,80){$\gamma_7$}\put(-60,90){$\bullet$}
\put(-120,62){$\gamma_5$} \put(-120,50){$\bullet$}
%\put(-160,10){$f_*$}
\put(-216,50){$\gamma_4$}\put(-195,52){$\bullet$}\put(-275,50){$\gamma_2$}\put(-280,50){$\bullet$}
\put(-255,30){$E_S$}\put(-235,50){$\bullet$}
\put(-65,50){$E_{f_*(S)}$}
\put(-250,85){$\gamma_1$}\put(-240,90){$\bullet$}
%\put(-245,210){$\bullet$}
\put(-230,25){$\gamma_3$}\put(-235,16){$\bullet$}
\put(-244,-5){$\mathbb{\overline{C}}(S)$}\put(-95,-5){$\mathbb{\overline{C}}(f_*(S))$}
\put(-170,70){$H_S$} %\put(-265,33){$R_\lambda((1+t)/2)$}
%\put(-86,30){$\partial B_\lambda$} \put(-93,110){$v_\lambda^+$}
 \caption{Marked disk extension. Here $\partial S=\gamma_1\cup\gamma_2\cup\gamma_3\cup\gamma_4$,
 $\partial f_*(S)=\gamma_5\cup\gamma_6\cup\gamma_7$. Marked points are labeled by `$\bullet$'.}
\end{center}\label{f5}
\end{figure}
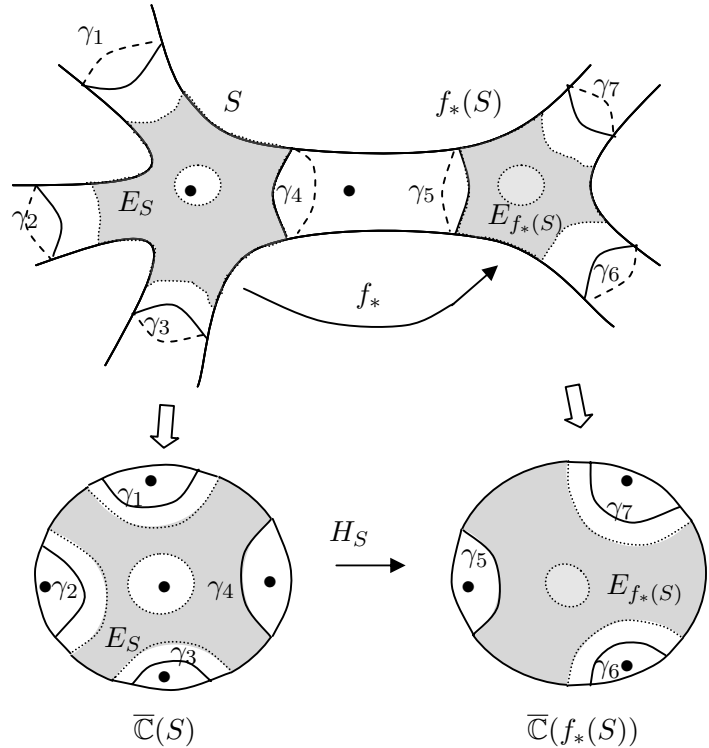

The map $H_{f_*^j(S)}: \Delta_j\rightarrow\Delta_{j+1}$ defined by
\begin{equation*}
H_{f_*^j(S)}(z)=\begin{cases}
 f(z),\ \  &z\in f^j(S\cap A_\gamma),\ 0\leq j<k,\\
e^{2\pi i \theta}z,\ \ &z\in\mathbb{D},\ j=0,\\
z, \ \ &z\in \mathbb{D},\ 1\leq j<k.
\end{cases}
\end{equation*}
 is  a holomorphic extension of $f|_{E_{f_*^j(S)}}$ along the
boundary curve $f^j(\gamma)\in
\partial_0(f_*^j(S))$. We call $\Delta_j$ a  {\it holomorphic
disk } of $H_{f_*^j(S)}$. This construction allows us  to define the
extensions of $f|_{E_{S}},\cdots,f|_{E_{f_*^{l-1}(S)}}$ (where $l$
is the $f_*$-period of $S$)
 along the curves in $\partial_0(S)\cup\cdots\cup
\partial_0(f_*^{l-1}(S))$ at the same time. We denote by
$\Delta_j^0\subset\Delta_j$ the sub-disk of $\Delta_j$ with boundary
curve $f^j(\gamma)$,  $0_j$ the center of $\Delta_j$. In this case,
we get a marked disk $(\Delta_j^0,0_j)$.

Now, we consider the case when $\gamma\in \partial(E_S)\setminus\partial_0(S)$. Note that either $U\subset S$ and $U$
contains at most one point in $P$, or $U$ contains exactly one
component $V$ of $\mathbb{\overline{C}}(S)-S$. In the former case,
if $U$ contains a marked point $p\in P$, then we get a marked disk
$(U,p)$; if $U\cap P=\emptyset$, then we don't mark any point in
$U$. In the latter case, we mark a  point $p\in V$ and get two
marked disks $(U,p)$ and $(V,p)$.

Now we extend $f|_{E_S}$ to $U$ in the following fashion:

We require that $H_S$ maps $U$ onto $(W,q)$ with
$\deg(H_S|_U)=\deg(f|_{\partial U})$, where $(W,q)$ is the marked
disk of $\mathbb{\overline{C}}(f_*(S))-f_*(S)$ whose boundary curve
is $f(\partial U)$. If $U$ contains a marked point $p$, we require
further $H(p)=q$ and the local degree of $H_S$ at $p$ is equal to
$\deg(f|_{\partial U})$. If $U$ contains no marked point, we require
that $q$ is the only possible critical value (this implies that $U$
contains at most one ramification point of $H_S$).

In this way, for each $\mathcal{S}$-piece $S$, we can get an
extension  $H_S: \mathbb{\overline{C}}(S)\rightarrow
\mathbb{\overline{C}}(f_*(S))$ of $f|_{E_S}$. Let $Z(S)=\{p; (V,p) \text{ is a marked disk in  } \mathbb{\overline{C}}(S)-S\}$,  ${D}(S)$ be  the
union of all  holomorphic   disks of $H_S$. Note that if
$\partial_0(S)=\emptyset$, then ${D}(S)=\emptyset$. Set
$$P(S)=(P\cap S)\cup Z(S)\cup {D}(S).$$ We call
$(\mathbb{\overline{C}}(S), P(S))$  a marked sphere of
$\mathbb{\overline{C}}(S)$. By the construction of $H_S$, we see
that $H_S(P(S))\subset P(f_*(S))$.

We know that every $\mathcal{S}$-piece is eventually periodic under
the map $f_*$. Let $n$ be the number of all $f_*$-cycles of
$\mathcal{S}$-pieces. These cycles are listed as follows:
$$S_\nu\mapsto f_*(S_\nu)\mapsto\cdots\mapsto f_*^{p_\nu-1}(S_\nu)\mapsto f_*^{p_\nu}(S_\nu)=S_\nu,\ \   1\leq \nu\leq n,$$
where $S_\nu$ is a representative of the $\nu$-th  cycle and $p_\nu$
is the period of  $S_\nu$.

Set $$h_\nu=H_{f_*^{p_\nu-1}(S_\nu)}\circ\cdots\circ
H_{f_*(S_\nu)}\circ H_{S_\nu}, \ P_\nu=P(S_\nu),\ \ 1\leq \nu\leq
n.$$
% for $1\leq \nu\leq n$.
Then $h_\nu: \mathbb{\overline{C}}(S_\nu)\rightarrow
\mathbb{\overline{C}}(S_\nu)$ is a branched covering
with $h_\nu(P_\nu)\subset P_\nu$.

 These resulting maps $(h_1,P_1),\cdots,(h_n,P_n)$
can be considered as  the renormalizations of the original map
$(f,P)$. There are three types of them:

$\bullet$  $\partial_0(S_\nu)\neq\emptyset$ or $S_\nu$ contains at
least one rotation disk of $(f,P)$. In this case, $(h_\nu, P_\nu)$
has at least one cycle of rotation disks,
% whose boundary contains  at least one critical point of $h_\nu$,
 so $(h_\nu, P_\nu)$ is a Siegel
map. Moreover, a curve $\gamma\in
\partial_0(S_\nu)$  in a rotation annulus of $(f,P)$ with
period $p$ and rotation number $\theta$ becomes a periodic curve in
a rotation disk of $(h_\nu,P_\nu)$, with period $p/p_\nu$ and
rotation number $\theta$. One may verify that the number of these
resulting Siegel maps is at least two, and at most
$2n_{RA}(f)+n_{RD}(f)$.

$\bullet$  $\partial_0(S_\nu)=\emptyset$,  $S_\nu$ contains no
rotation disk of $(f,P)$ and $\deg(h_\nu)>1$. In this case, $P_\nu$
is a finite set and $(h_\nu, P_\nu)$ is a Thurston map.

$\bullet$  $\partial_0(S_\nu)=\emptyset$,  $S_\nu$ contains no
rotation disk of $(f,P)$ and $\deg(h_\nu)=1$. In this case, $(h_\nu,
P_\nu)$ is a homeomorphism of $\mathbb{\overline{C}}(S_\nu)$ and
$h_\nu(P_\nu)= P_\nu$. So every point of $P_\nu$ is periodic.
Moreover, for any $S\in\{S_\nu,f_*(S_\nu),\cdots,
f_*^{p_\nu-1}(S_\nu)\}$, each component of $S-E_S$ is an annulus.

\

Let $\Lambda$ be the index set consisting of all
$\nu\in\{1,\cdots,n\}$ such that $\deg(h_\nu)>1$. That is, for each
$\nu\in \Lambda$, $(h_\nu,P_\nu)$ is either a Siegel map or a
Thurston map. Let $\Lambda^*=\{1,\cdots,n\}-\Lambda$.

   We use the
following notation to record the above decomposition and marked disk
extension procedure:
$${\rm Dec}(f,P)=\Bigg(\bigoplus_{\nu\in \Lambda\cup\Lambda^*}(h_\nu, P_\nu)\Bigg)_\Gamma.$$

\begin{lem}\label{b3aa}
%$(\mathbf{  \lambda<1\ implies\ no\ (2,2,2,2)-Thurston\ map})$  \label{2x}
If $\lambda(\Gamma,f)<1$, then

1.  For any $1\leq\nu\leq n$, every point in
$Z(S_\nu)$ is eventually
mapped to either the center of some rotation disk or a periodic
critical point of $(h_\nu, P_\nu)$.
% in $(\mathbb{\overline{C}}(S_\nu)-S_\nu)\cap P_\nu$ are critical.

2. $\Lambda^*=\emptyset$.

3.  If $(h_\nu, P_\nu)$ is a Thurston map, then the
 signature of the orbifold of $(h_\nu, P_\nu)$ is not $(2,2,2,2)$.
\end{lem}
\begin{proof}
Since $h_\nu(Z(S_\nu))\subset Z(S_\nu)$ and
$Z(S_\nu)$ is a finite set, we
conclude that every point in
$Z(S_\nu)$ is eventually
periodic under the iterations of $(h_\nu,P_\nu)$.

If $\Gamma=\emptyset$, then $\partial(S_\nu)\subset \Gamma_0$ and
all resulting maps $(h_\nu,P_\nu)$ are Siegel maps. The marked disk
extension procedure implies that every point in
$Z(S_\nu)$ is   the center of
some rotation disk. The conclusions follows immediately in this
case.

 In the following, we assume $\Gamma\neq\emptyset$. Let
$z_0$ be a periodic point in
$Z(S_\nu)$ with period $k$.
Suppose that $z_0$ is not  the center of any  rotation disk, and let
$\beta$ be the boundary curve of $S_\nu$ that encloses $z_0$. Then
there is a unique component of $h_\nu^{-k}(\beta)$, say $\alpha$,
contained in  $S_\nu$ and  homotopic to $\beta$ in
$\mathbb{\overline{C}}(S_\nu)-P_\nu$. Thus
$${\rm deg}(h_\nu^k, z_0)=
{\rm deg}(h_\nu^k:\alpha\rightarrow\beta)={\rm
deg}(f^{kp_\nu}:\alpha\rightarrow\beta)\geq
\lambda(\Gamma,f)^{-kp_\nu}>1.$$ This implies that $z_0$ lies in a
critical cycle and $\deg(h_\nu)>1$.  It follows that
$\Lambda^*=\emptyset$ and there is no $(2,2,2,2)$-type Thurston map
among  $(h_\nu, P_\nu)_{\nu\in \Lambda}$.
\end{proof}

\section{Combination part: decompositions of stable multicurve} \label{3-2}

In this section, we will  prove the following:

\begin{thm} \label{b3a} Let $(f,P)$ be a Herman map,  and
$${\rm Dec}(f,P)=\Bigg(\bigoplus_{\nu\in \Lambda\cup\Lambda^*}(h_\nu, P_\nu)\Bigg)_\Gamma.$$
Then $(f,P)$ has no Thurston obstructions if and only if
$\lambda(\Gamma,f)<1$  and for each $\nu\in \Lambda$,
$(h_\nu,P_\nu)$ has no Thurston obstructions.
\end{thm}

Note that if $(f,P)$ has no Thurston obstructions or
$\lambda(\Gamma,f)<1$, then $\Lambda^*=\emptyset$ (see Lemma
\ref{b3aa}).

The proof of the  `sufficiency' of Theorem \ref{b3a} is based on the
decomposition of   $(f,P)$-stable multicurves. We will show that
every $(f,P)$-stable multicurve  contains an `essential'
submulticurve (Lemma \ref{b3b}), and every such essential
submulticurve can be  decomposed into a `$\Gamma$-part' multicurve
%$(f,P)$-stable multicurve
together with a $(h_\nu,P_\nu)$-stable multicurve for each
$\nu\in[1,n]$.
 The important fact of this decomposition is that
%$(f,P)$-stable multicurves correspond to the block decompositions of
%matrices,  and
the leading eigenvalues of the transition matrices
satisfy the so-called `reduction identity' (Theorem \ref{b3c}).

To prove  the  `necessity' of Theorem \ref{b3a}, we will show that
every $(h_\nu,P_\nu)$-stable multicurve $\Sigma$ can generate a
$(f,P)$-stable multicurve $\mathcal{C}$ with
$\lambda(\Sigma,h_\nu)\leq\lambda(\mathcal{C},f)^{p_\nu}$.

\begin{lem}[`Essential' submulticurve]\label{b3b}   Let $\mathcal{C}_0$ be a  $(f,P)$-stable multicurve,  then
there is a  $(f,P)$-stable multicurve $\mathcal{C}$, such that

1. $\mathcal{C}$ is homotopically contained in $\mathcal{C}_0$.

2. Each curve of $\mathcal{C}$ is contained in the interior of some
$\mathcal{S}$-piece.

3. $\lambda(\mathcal{C},f)=\lambda(\mathcal{C}_0,f)$.
\end{lem}

\begin{proof} For $n\geq1$, we define a multicurve $\mathcal{C}_n$ inductively: $\mathcal{C}_n\subset f^{-1}(\mathcal{C}_{n-1})$
and $\mathcal{C}_n$ represents all homotopy classes of
non-peripheral curves of $f^{-1}(\mathcal{C}_{n-1})$. Since
$\mathcal{C}_0$ is a $(f,P)$-stable multicurve, we conclude that all
$\mathcal{C}_n$ are $(f,P)$-stable,  and $\mathcal{C}_n$ is
homotopically contained in $\mathcal{C}_{n-1}$. Let $W_n$ be the
$(f,P)$-transition matrix of $\mathcal{C}_n$ for $n\geq0$, then

$$W_{n}=\left(
\begin{array} {cc}
W_{n+1}  &   *\\
O  &    O
\end{array}
\right).$$

Thus
$\lambda(\mathcal{C}_0,f)=\lambda(\mathcal{C}_1,f)=\lambda(\mathcal{C}_2,f)=\cdots$.
By the construction of $\Gamma$, there is an integer $N\geq0$ such
that $\Gamma\subset f^{-n}(\Gamma_0)$ for all $n\geq N$, where $\Gamma_0$ is the
choice of a  collection of $f$-periodic curves in the rotation
annuli (see the previous section). Since  $\cup\Gamma_0$ has no
intersection with $\cup\mathcal{C}_0$,  we conclude that
$f^{-n}(\cup\Gamma_0)$ has no intersection with
$f^{-n}(\cup\mathcal{C}_0)$ for all $n\geq1$.  Thus when $n\geq N$,
we have $\cup\mathcal{C}_n\subset f^{-n}(\cup\mathcal{C}_0)
 \subset\mathbb{\overline{C}}\setminus f^{-n}(\cup\Gamma_0)\subset\mathbb{\overline{C}}\setminus \cup(\Gamma\cup\Gamma_0)$. This implies that  each curve of
$\mathcal{C}_n$ is contained in the interior of some
$\mathcal{S}$-piece. The proof is completed if we set
$\mathcal{C}=\mathcal{C}_n$ for some $n\geq N$.
\end{proof}

%Now, let $\Upsilon$ be a  $(f,P)$-stable multicurve, we suppose that
%each curve of  $\Upsilon$ is contained in a $\mathcal{S}$-piece. Let

%Notice that under the inclusion map $i_\nu: S_\nu\hookrightarrow
%\mathbb{\overline{C}}(S_\nu)$, the multicurve $\Sigma_\nu$ can be
%considered as a multicurve of $(h_\nu,P_\nu)$. We still use
% $\Sigma_\nu$ to denote the multicurve  $\{i_\nu(\gamma); \gamma\in\Sigma_\nu\}$  of $(h_\nu,P_\nu)$  if there is no confusion. We define
% \bess
% &&\omega_\Lambda(\Upsilon)=\max\{\sqrt[p_\nu]{\lambda(\Sigma_\nu,h_\nu)};\nu\in
%\Lambda\},\\
%&&\omega_{\Lambda^*}(\Upsilon)=\begin{cases}
%1,   & \text{ if }  \cup_{\nu\in\Lambda^*}\Sigma_\nu\neq\emptyset,\\
%0,   & \text{ if }  \cup_{\nu\in\Lambda^*}\Sigma_\nu=\emptyset.
% \end{cases}
% \eess

%,\cdots,\sqrt[p_n]{\lambda(\Sigma_n,h_n)}

\begin{thm} [Decomposition of stable multicurve] \label{b3c}    Let $\mathcal{C}$ be a  $(f,P)$-stable
multicurve. Suppose that  each curve of  $\mathcal{C}$ is contained
in the interior of some $\mathcal{S}$-piece. Let \bess
&\mathcal{C}_\Gamma=\{\gamma \in \mathcal{C}; \gamma \text{ is
homotopic in } \mathbb{\overline{C}}-P \text{
to a curve  of } \Gamma \},&\\
&\Sigma_\nu=\{\gamma \in \mathcal{C}-\mathcal{C}_\Gamma; \gamma
\text{ is contained in } S_\nu\}, \  \nu\in
\Lambda\cup\Lambda^*=[1,n].& \eess
%$\Upsilon_\Gamma=
%\{\gamma \in \Upsilon; \gamma \text{ is homotopic to a curve  in }
%\Gamma \}$, and  $\Sigma_\nu= \{\gamma \in \Upsilon-\Upsilon_\Gamma;
%\gamma \text{ is contained in } S_\nu\}$ for $1\leq \nu\leq n$.
Then $\mathcal{C}_\Gamma$ is a $(f,P)$-stable multicurve,
$\Sigma_\nu$ is a $(h_\nu,P_\nu)$-stable multicurve for each
$\nu\in[1,n]$, and we have the following reduction identity:
$$\lambda(\mathcal{C},f)=\max\Big\{\lambda(\mathcal{C}_\Gamma,f),
\sqrt[p_1]{\lambda(\Sigma_1,h_1)},\cdots,\sqrt[p_n]{\lambda(\Sigma_n,h_n)}\Big\}.$$
\end{thm}

\begin{rem} In Theorem \ref{b3c},  the multicurve $\Sigma_\nu$ can be viewed as a multicurve of
$(h_\nu,P_\nu)$, this is because  under the inclusion map
$\iota_\nu: S_\nu\hookrightarrow \mathbb{\overline{C}}(S_\nu)$, the
set $\iota_\nu(\Sigma_\nu):=\{\iota_\nu(\gamma);
\gamma\in\Sigma_\nu\}$ is a multicurve in
$\mathbb{\overline{C}}(S_\nu)-P_\nu$. We still use $\Sigma_\nu$ to
denote the multicurve $\iota_\nu(\Sigma_\nu)$ if there is no
confusion.

One may show directly that if $\Lambda^*\neq\emptyset$, then for any
$\nu\in\Lambda^*$, \bess \lambda(\Sigma_\nu,h_\nu)=\begin{cases}
1,   & \text{ if }  \Sigma_\nu\neq\emptyset,\\
0,   & \text{ if }  \Sigma_\nu=\emptyset.
 \end{cases}
 \eess

This observation can simplify the reduction identity.
% $\Sigma_\nu$ to denote the multicurve  $\{i_\nu(\gamma); \gamma\in\Sigma_\nu\}$  of $(h_\nu,P_\nu)$
\end{rem}

\begin{proof}
The fact that $\mathcal{C}_\Gamma$ is $(f,P)$-stable is easy to
verify since both $\Gamma$ and $\mathcal{C}$ are $(f,P)$-stable. Let
$\Sigma_\nu^k=\{\gamma \in \mathcal{C}-\mathcal{C}_\Gamma; \gamma
\text{ is contained in } f^k_*(S_\nu) \}$ for $0\leq k\leq p_\nu$.
It's obvious that $\Sigma_\nu^0=\Sigma_\nu^{p_\nu}=\Sigma_\nu$.
Since $\mathcal{C}$ is $(f,P)$-stable, each non-peripheral component
of $f^{-1}(\gamma)$ for $\gamma\in \Sigma_\nu^{k+1} (0\leq k<
p_\nu)$ is homotopic  in $\mathbb{\overline{C}}-P$  to  either  a
curve $\alpha\in \mathcal{C}_\Gamma$, or   a curve $\beta\in
\Sigma_\nu^{k}$, or a curve $\delta$ contained in a strictly
preperiodic $\mathcal{S}$-piece.

By the definition of the marked set $P(f^k_*(S_\nu))$, one can
verify that the set $\Sigma_\nu^{k}$ is a multicurve in
$\mathbb{\overline{C}}(f^k_*(S_\nu))-P(f^k_*(S_\nu))$. Moreover,
each curve $\gamma\in \mathcal{C}_\Gamma$ contained  in
$f^k_*(S_\nu)$  is peripheral or null-homotopic in
$\mathbb{\overline{C}}(f^k_*(S_\nu))-P(f^k_*(S_\nu))$. Thus for any
$0\leq k< p_\nu$ and any curve $\gamma\in \Sigma_\nu^{k+1} $, each
non-peripheral component of $H_{f^k_*(S_\nu)}^{-1}(\gamma)$ is
homotopic to a curve $\delta\in \Sigma_\nu^{k}$ in
$\mathbb{\overline{C}}(f^k_*(S_\nu))-P(f^k_*(S_\nu))$. It follows
that each  non-peripheral component of $h_{\nu}^{-1}(\gamma)$ with
$\gamma\in \Sigma_\nu$ is homotopic to a curve $\delta\in
\Sigma_\nu$ in $\mathbb{\overline{C}}(S_\nu)-P_\nu$. This means
$\Sigma_\nu$ is a $(h_\nu, P_\nu)$-stable multicurve.

In the following, we will prove the `reduction identity'. Let
$W_{\mathcal{C}_\Gamma}$ be the $(f,P)$-transition matrix of
$\mathcal{C}_\Gamma$. We define $\mathcal{C}_{s}:=\{\gamma\in
\mathcal{C}-\mathcal{C}_\Gamma; \gamma \text{ is contained in}$ $\text{a
strictly preperiodic } \mathcal{S}\text{-piece}\}$ with
$(f,P)$-transition matrix $W_{s}$. Let
$\mathcal{C}_\nu=\Sigma_\nu^0\cup \cdots \cup\Sigma_\nu^{p_\nu-1}$
with $(f,P)$-transition matrix $W_{\nu}$. Then the
$(f,P)$-transition matrix $W_{\mathcal{C}}$ of $\mathcal{C}$ has the
following block decomposition:
$$W_{\mathcal{C}}=\left(
\begin{array} {ccccc}
W_{\mathcal{C}_\Gamma} & * & * &\cdots &* \\ O & W_{s} & * &\cdots
&*
\\O&O
&W_1 &\cdots&*\\ \vdots &\vdots &\vdots &\ddots&\vdots  \\
O&O&O&\cdots & W_n
\end{array}
\right). $$

It follows that
$\lambda(\mathcal{C},f)=\max\big\{\lambda(\mathcal{C}_\Gamma,f),\lambda(\mathcal{C}_{s},f)
,\lambda(\mathcal{C}_{1},f),\cdots,\lambda(\mathcal{C}_{n},f)\big\}$.

We claim that $\lambda(\mathcal{C}_{s},f)=0$. To see this, let
$\mathcal{S}_s$ be the collection of all strictly preperiodic
$\mathcal{S}$-pieces. For each $S\in \mathcal{S}_s$, let $\tau(S)$
be the least integer $k\geq1$ such that $f_*^k(S)$  is a periodic
$\mathcal{S}$-piece. Set $M=\max\{\tau(S);S\in \mathcal{S}_s\}$. For
any $\gamma\in\mathcal{C}_{s}$, let  $\alpha$ be a non-peripheral
component of $f^{-M}(\gamma)$. If  $\alpha$ is not homotopic to any
curve in $\mathcal{C}_\Gamma$, then there is
$S_\alpha\in\mathcal{S}_s$ such that $\alpha$ is contained in the
$\mathcal{E}$-piece $E_{S_\alpha}$ parallel to $S_\alpha$. Moreover,
for any $1\leq j<M$, $f^{j}(\alpha)$ is not homotopic to any curve
in $\mathcal{C}_\Gamma$ and $f^{j}(\alpha)\subset
E_{f^j_*(S_\alpha)}$. In particular, $f^M(\alpha)=\gamma\subset
f^M_*(S_\alpha)\in \mathcal{S}_s$. This implies $\tau(S_\alpha)\geq
M+1$. But this contradicts the choice of $M$. Thus,
 $\alpha$ is either null-homotopic, or peripheral, or homotopic to a
curve $\delta\in\mathcal{C}_\Gamma$ in $\mathbb{\overline{C}}-P$.
Equivalently, $W^M_{s}=0$ and $\lambda(\mathcal{C}_{s},f)=0$. So we
have
$$\lambda(\mathcal{C},f)=\max\Big\{\lambda(\mathcal{C}_\Gamma,f),\lambda(\mathcal{C}_{1},f),\cdots,\lambda(\mathcal{C}_{n},f)\Big\}.$$

Notice that the $(f,P)$-transition matrix $W_{\nu}$ of
$\mathcal{C}_\nu$ takes the form

%$\begin{pmatrix}
%T_{++} \hfill & T_{+-} \\
%T_{-+} & T_{--} \hfill
%\end{pmatrix}$
$$W_\nu=\left(
\begin{array} {ccccc}
O & B_0 & O &\cdots &O\\ O & O& B_1  & \cdots &O\\
\vdots &\vdots &\vdots &\ddots&\vdots
\\O&O&O&\cdots& B_{p_\nu-2}\\
B_{p_\nu-1}&O&O&\cdots &O
\end{array}
\right),$$ where $B_{j}$ is a  $n_j\times n_{j+1}$ matrix, $n_j$ is
equal to the number of curves in $\Sigma_\nu^j$ for $0\leq j\leq
p_\nu-1$. A direct calculation yields

$$W^{p_\nu}_\nu=\left(
\begin{array} {ccccc}
B_0 B_1\cdots B_{p_\nu-1} & O & \cdots &O\\ O & B_1 B_2\cdots B_0& \cdots &O\\
\vdots &\vdots &\ddots&\vdots\\
O&O&\cdots &B_{p_\nu-1} B_0\cdots B_{p_\nu-2}
\end{array}
\right).$$

For any $k\geq1$, we have
$$\|(W_\nu^{p_\nu})^k\|^2=\|(B_0 B_1\cdots B_{p_\nu-1})^k\|^2+\cdots+\|(B_{p_\nu-1} B_0\cdots B_{p_\nu-2})^k\|^2.$$

It follows from Lemma \ref{b3d} that
$$\mathrm{sp}(W_\nu)^{p_\nu}=\mathrm{sp}(B_0 B_1\cdots B_{p_\nu-1}
)=\cdots=\mathrm{sp}(B_{p_\nu-1} B_0\cdots B_{p_\nu-2}).$$

On the other hand,  one can  verify that the
$(h_\nu,P_\nu)$-transition matrix of $\Sigma_\nu$ is $B_0 B_1\cdots
B_{p_\nu-1}$. It follows from  Perron-Frobenius Theorem that
$$\lambda(\Sigma_\nu,h_\nu)=\mathrm{sp}(B_0 B_1\cdots B_{p_\nu-1})
=\mathrm{sp}(W_\nu)^{p_\nu}=\lambda(\mathcal{C}_{\nu},f)^{p_\nu}.$$

Finally, we get the reduction identity \bess \lambda(\mathcal{C},f)
&=&\max\Big\{\lambda(\mathcal{C}_\Gamma,f),\sqrt[p_1]{\lambda(\Sigma_1,h_1)},\cdots,\sqrt[p_n]{\lambda(\Sigma_n,h_n)}\Big\}.\eess
\end{proof}

\begin{lem}\label{b3d}  Let $B_{\nu}$ be a $n_\nu\times n_{\nu+1}$
real matrix for $1\leq\nu\leq k$, $n_{k+1}=n_1$, then
$$\mathrm{sp}(B_{1}B_{2}\cdots B_{k})=\mathrm{sp}(B_{2}B_{3}\cdots B_{1})=\cdots=\mathrm{sp}(B_{k} B_{1}\cdots
B_{k-1}).$$
\end{lem}
\begin{proof}
 First we assume $n_1=\cdots=n_{k}$, fix some
$1\leq\nu\leq k$, by Cauchy-Schwarz inequality (i.e. $\|AB\|\leq\|A\|\|B\|$),
 \bess \mathrm{sp}(B_{1}B_{2}\cdots
B_{k})&=&\lim_{n\rightarrow\infty}\sqrt[n]{\|(B_{1}B_{2}\cdots
B_{k})^n\|}\\
&=&\lim_{n\rightarrow\infty}\sqrt[n]{\|(B_{1}\cdots B_{\nu-1})(B_\nu
B_{\nu+1}\cdots B_{\nu-1})^{n-1}(B_{\nu}\cdots B_{k})\|}\\
&\leq&\lim_{n\rightarrow\infty}\sqrt[n]{\|B_{1}\cdots
B_{\nu-1}\|\|(B_\nu B_{\nu+1}\cdots
B_{\nu-1})^{n-1}\|\|B_{\nu}\cdots B_{k}\|}\\
&=&\mathrm{sp}(B_{\nu}B_{\nu+1}\cdots B_{\nu-1}).
 \eess
 The same argument leads to the  other direction of the inequality. In the following, we
 deal with the general case. Choose $n\geq \max\{n_1,\cdots,n_k\}$,
 for any $1\leq \nu \leq k$,
 we define a $n\times n$ matrix $\widehat{B}_\nu$ by
 $$\widehat{B}_\nu=\left(
\begin{array} {cc}
{B}_\nu & O_{n_\nu\times(n-n_{\nu+1})}\\ O_{(n-n_{\nu})\times
n_{\nu+1}}&O_{(n-n_{\nu})\times(n-n_{\nu+1})}
\end{array}
\right),$$ where we use $O_{p\times q}$ to denote the $p\times q$
zero matrix. Then by the above argument,
$\mathrm{sp}(\widehat{B}_{1}\widehat{B}_{2}\cdots \widehat{B}_{k})=
\mathrm{sp}(\widehat{B}_{\nu}\widehat{B}_{\nu+1}\cdots
\widehat{B}_{\nu-1}).$ On the other hand,
$$\widehat{B}_{\nu}\widehat{B}_{\nu+1}\cdots
\widehat{B}_{\nu-1}=\left(
\begin{array} {cc}
B_{\nu}B_{\nu+1}\cdots B_{\nu-1} & O_{n_\nu\times(n-n_{\nu})}\\
O_{(n-n_{\nu})\times n_{\nu}}&O_{(n-n_{\nu})\times(n-n_{\nu})}
\end{array}
\right).$$ This implies that  $\|(B_{1}B_{2}\cdots
B_{k})^n\|=\|(\widehat{B}_{1}\widehat{B}_{2}\cdots
\widehat{B}_{k})^n\|$ for all $n\geq1$. So
$$\mathrm{sp}(B_{1}\cdots
B_{k})=\mathrm{sp}(\widehat{B}_{1}\cdots
\widehat{B}_{k})=
\mathrm{sp}(\widehat{B}_{\nu}\cdots
\widehat{B}_{\nu-1})=\mathrm{sp}(B_{\nu}\cdots
B_{\nu-1}).$$
\end{proof}

\noindent\textbf{Proof of Theorem \ref{b3a}.}
{\it Sufficiency.} Let  $\mathcal{C}$ be a  $(f,P)$-stable
multicurve in $\mathbb{\overline{C}}-P$. We may assume that each
curve $\gamma\in\mathcal{C}$ is contained in the interior of some
$\mathcal{S}$-piece by Lemma \ref{b3b}.  The multicurves
$\mathcal{C}_\Gamma, \Sigma_1, \cdots, \Sigma_n$ are the subsets of
$\mathcal{C}$ defined in Theorem \ref{b3c}. If $\lambda(\Gamma,f)<1$
(note that this implies $\Lambda^*=\emptyset$ by Lemma \ref{b3aa})
and $(h_\nu,P_\nu)$ has no Thurston obstructions for each $\nu\in
\Lambda$, then by Theorem \ref{b3c}, we have
 \bess \lambda(\mathcal{C},f)
&=&\max\Big\{\lambda(\mathcal{C}_\Gamma,f),\sqrt[p_1]{\lambda(\Sigma_1,h_1)},\cdots,\sqrt[p_n]{\lambda(\Sigma_n,h_n)}\Big\}\\
&\leq&\max\Big\{\lambda(\Gamma,f),\sqrt[p_1]{\lambda(\Sigma_1,h_1)},\cdots,\sqrt[p_n]{\lambda(\Sigma_n,h_n)}\Big\}<1.\eess
This means $(f,P)$ has no Thurston obstructions.

{\it Necessity.} Suppose that  $(f,P)$ has no Thurston obstructions.
Then $\lambda(\Gamma,f)<1$ and $\Lambda^*=\emptyset$. Let $\Sigma$
be  a $(h_\nu,P_\nu)$-stable multicurve  in
$\mathbb{\overline{C}}(S_\nu)-P_\nu$. Up to homotopy, we may assume
that each curve $\gamma\in \Sigma$ is contained in the interior of
$S_\nu$, so $ \Sigma$ can be considered as a multicurve in
$\mathbb{\overline{C}}-P$. In the following, we will use $\Sigma$ to
generate a $(f,P)$-stable multicurve $\mathcal{C}$.

For $k\geq0$, let $\Lambda_k\subset f^{-k}(\Sigma)$ be a multicurve
 in $\mathbb{\overline{C}}-P$, representing all homotopy classes of non-peripheral
curves in $f^{-k}(\Sigma)$. We claim that

{\it For any $\alpha\in\Lambda_i, \beta\in \Lambda_j$ with $0\leq
i<j$, if $\alpha$ is not homotopic to $\beta$ in
$\mathbb{\overline{C}}-P$, then $\alpha$ and $\beta$ are
homotopically disjoint. (`homotopically disjoint' means that the
homotopy classes of $\alpha$ and $\beta$ can be represented by two
disjoint Jordan curves.)}

In fact, the claim is obviously true in  the following two cases:

1. The curves $\alpha$ and $\beta$  are contained in two different
$\mathcal{S}$-pieces.

2.  Either $\alpha$ or $\beta$ is homotopic a curve in $\Gamma$.

In what follows, we assume that $\alpha$ and $\beta$ are contained
in the same $\mathcal{S}$-piece $S$, and neither  is homotopic to a
boundary curve of  $S$. We assume further that  they intersect
homotopically. In this case,  one may check that both $f^i(\alpha)$
and $f^i(\beta)$ are contained in $f^i_*(S)=S_\nu$, but neither of
$f^i(\alpha)$ and $f^i(\beta)$ is homotopic to a boundary curve of
$S_\nu$. So $f^i(\beta)$ is contained in the unique component of
$f^{i-j}(S_\nu)$ that is parallel to $S_\nu$. This implies $i \equiv
j {\ \rm mod }\ p_\nu$. Since $f^j(\beta)\in\Sigma$ and $\Sigma$ is
$(h_\nu,P_\nu)$-stable, we have that   $f^i(\beta)$ is homotopic in
$\mathbb{\overline{C}}-P$ to either  a curve of $\Sigma$ or a curve
of $\Gamma$. But this is a contradiction because we assume that
$\alpha$ and $\beta$ intersect homotopically. This ends the proof of
the claim.

For $k\geq0$, we define a collection of Jordan curves
$\mathcal{C}_k$ such that
$\Sigma\subset\mathcal{C}_k\subset\Lambda_0\cup\cdots\cup\Lambda_k$
and $\mathcal{C}_k$ represents  all homotopy classes of
non-peripheral curves in $\Lambda_0\cup\cdots\cup\Lambda_k$. It
follows from the above claim that we can consider $\mathcal{C}_k$ to
be a multicurve in $\mathbb{\overline{C}}-P$ up to
 homotopy. Note that $\mathcal{C}_k$ is homotopically contained in
$\mathcal{C}_{k+1}$, we have $\#\mathcal{C}_k\leq
\#\mathcal{C}_{k+1}$. Since $P$ has finitely many components,
$\#\mathcal{C}_k$ is uniformly bounded above for all $k$. So there
is an integer $N\geq0$, such that
$\#\mathcal{C}_n=\#\mathcal{C}_{N}$ for all $n\geq N$.

Let $\mathcal{C}=\mathcal{C}_N$, then $\mathcal{C}$ is a
$(f,P)$-stable multicurve by the choice of $N$. Let
$\mathcal{C}_\Gamma= \{\gamma \in\mathcal{C}; \gamma \text{ is
homotopic to a curve in } \Gamma \}$,
 one may verify that
$\Sigma= \{\gamma \in \mathcal{C}-\mathcal{C}_\Gamma; \gamma \text{
is contained in } S_\nu\}$. By Theorem \ref{b3c},
$$\lambda(\Sigma,h_\nu)\leq \lambda(\mathcal{C},f)^{p_\nu}<1.$$

Thus   $(h_\nu,P_\nu)$ has no Thurston obstructions.  \hfill $\Box$

\section{Realization part I: gluing holomorphic models}\label{3-3}

The aim of the following two sections is to prove:

\begin{thm}\label{b4a} Let $(f,P)$ be a Herman map,  and
$${\rm Dec}(f,P)=\Bigg(\bigoplus_{\nu\in \Lambda\cup\Lambda^*}(h_\nu, P_\nu)\Bigg)_\Gamma.$$
Then $(f,P)$ is c-equivalent to a rational map if and only if
$\lambda(\Gamma,f)<1$ and for each $\nu\in \Lambda$, $(h_\nu,P_\nu)$
is c-equivalent to a rational map.
\end{thm}

In the proof of Theorem \ref{b4a}, without loss of generality, we assume that
$(f,P)$ and $(h_k,P_k)$ are  quasi-regular,
and the rational realizations are   q.c-rational realizations. This assumption is not essential,  we need it  simply because we want to use the language of quasi-conformal surgery. Without this assumption, one
  just need   replace the  `Measurable Riemann Mapping Theorem' by the `Uniformization Theorem'  in the proof but with no other essential differences.

 In Section \ref{3-3-1}, we prove the necessity of
Theorem \ref{b4a}. The idea is as follows: we use the rational
realization of $(f,P)$, say $(R,Q)$, to generate the partial
holomorphic models of $(h_\nu,P_\nu)_{\nu\in\Lambda}$. These partial
holomorphic models take the form $R^{p_\nu}|_{E_\nu},
\nu\in\Lambda$, where $E_\nu$ is a multi-connected domain in the
Riemann sphere $\mathbb{\overline{C}}$. The holomorphic map
$R^{p_\nu}|_{E_\nu}$ can be extended to a Siegel map or a Thurston
map, say $(g_\nu,Q_\nu)$, q.c-equivalent to $(h_\nu,P_\nu)$.
Moreover, $(g_\nu,Q_\nu)$ can be made holomorphic outside a
neighborhood of the boundary $\partial E_\nu$.  In the final step,
we use quasi-conformal surgery to make the  map $(g_\nu,Q_\nu)$
globally holomorphic and get a rational realization of
$(h_\nu,P_\nu)$.

% These holomorphic models
%are in fact a collection of suitable restrictions of the rational
%realization of $(f,P)$. These holomorphic models can be extended to
%a `good' branched covering to which $(h_\nu,P_\nu)$ is c-equivalent
%(Here, `good' means `holomorphic' in most parts of the Riemann
%sphere $\mathbb{\overline{C}}$). In the final step, we apply
%quasi-conformal surgery to make the `good' branched covering globally
%holomorphic.

In Section \ref{3-3-2},  we prove the sufficiency  of Theorem
\ref{b4a} assuming $\Gamma=\emptyset$. This part is the inverse
procedure of Section \ref{3-3-1}. We use the rational realizations
of $(h_\nu,P_\nu), \nu\in\Lambda$ to generate the partial
holomorphic models of $(f,P)$. These partial holomorphic models can
be glued along $\Sigma=\Gamma_0$ into a branched covering $(g,Q)$,
holomorphic in most part of $\mathbb{\overline{C}}$ and
q.c-equivalent to $(f,P)$. Finally, we apply quasi-conformal surgery
to make  the map $(g,Q)$ globally holomorphic.

% By a suitable gluing of these holomorphic models along
%$\Sigma=\Gamma_0$, we can make $(f,P)$ c-equivalent to a branched
%covering,  holomorphic in most parts of the Riemann sphere
%$\mathbb{\overline{C}}$. Then

The proof the sufficiency   of Theorem \ref{b4a} in the more general
case $\Gamma\neq\emptyset$ is put in the next section.

%Here is a sketch of how the proof works:  The rational realizations
%of $(h_\nu,P_\nu), \nu\in\Lambda$ provide  us partial holomorphic
%models for $(f,P)$.  These holomorphic models are in fact the
%rational  realizations of a collection of suitable restrictions of
%$(f,P)$. The algebraic condition $\lambda(\Gamma,f)<1$ enables us to
%glue all  these holomorphic models in a proper way to get a branched
%covering, holomorphic in most parts of the Riemann sphere
%$\mathbb{\overline{C}}$. Then we apply quasi-conformal surgery to
%make this branched covering globally holomorphic.

\subsection{Proof of the necessity of Theorem \ref{b4a}}\label{3-3-1}

 To prove the necessity of Theorem \ref{b4a}, we need a result of McMullen \cite{McM2}:

\begin{thm}[Marked McMullen Theorem] \label{b4b}
Let $R$ be a rational map, $M$ be a closed set containing the
postcritical set $P_R$ and $R(M)\subset M$. Let $\Gamma$ be a
multicurve in $\mathbb{\overline{C}}-M$. Then
$\lambda(\Gamma,R)\leq1$. If $\lambda(\Gamma,R)=1$, then either $R$
is postcritically finite whose orbifold has signature (2, 2, 2, 2);
or $R$ is  postcritically infinite, and $\Gamma$ includes a curve
contained in a periodic Siegel disk or Herman ring.
\end{thm}

We  remark that the definition of the multicurve in
$\mathbb{\overline{C}}-M$ is similar to the definition of the
multicurve in $\mathbb{\overline{C}}-P$.
Theorem \ref{b4b} is slightly stronger than McMullen's original result, but
the proof works equally well.

\vskip 0.3cm
\noindent \textit{Proof of the necessity of Theorem \ref{b4a}.}  Suppose that
$(f,P)$ is q.c-equivalent to a rational map $(R,Q)$ via a pair of
quasi-conformal maps $(\phi_0,\phi_1)$. Then the $(f,P)$-stable
multicurve $\Gamma$ in $\mathbb{\overline{C}}-P$ induces a
$(R,Q)$-stable multicurve $\phi_0(\Gamma):= \{\phi_0(\gamma);
\gamma\in \Gamma\}$ in $\mathbb{\overline{C}}-Q$. Since the marked
set $Q$ contains all possible Siegel disks and Herman rings of $R$,
it follows from Theorem \ref{b4b} that $\lambda(\Gamma,
f)=\lambda(\phi_0(\Gamma), R)<1$.

Note that $\lambda(\Gamma,f)<1$ implies $\Lambda^*=\emptyset$ (Lemma
\ref{b3aa}). In the following, we will show that for each
$\nu\in\Lambda$, $(h_\nu,P_\nu)$ is q.c-equivalent to a rational
map.

Let $H_0:[0,1]\times \mathbb{\overline{C}}\rightarrow
\mathbb{\overline{C}}$ be an isotopy between $\phi_0$ and $\phi_1$
rel $P$. That is, $H_0:[0,1]\times \mathbb{\overline{C}}\rightarrow
\mathbb{\overline{C}}$ is a continuous map such that
 $H_0(0,\cdot)=\phi_0,
H_0(1,\cdot)=\phi_1$ and $H_0(t,z)=\phi_0(z)$ for all $(t,z)\in
[0,1]\times P$. Moreover, for any  $t\in [0,1]$, $H_0(t,\cdot):
\mathbb{\overline{C}}\rightarrow \mathbb{\overline{C}}$ is a
quasi-conformal map. Then by induction, for any $k\geq0$, there is a
unique lift of $H_k$, say $H_{k+1}$, such that $H_k(t,
f(z))=R(H_{k+1}(t,z))$ for all $(t,z)\in [0,1]\times
\mathbb{\overline{C}}$, with basepoint
$H_{k+1}(0,\cdot)=\phi_{k+1}$. Set $\phi_{k+2}=H_{k+1}(1,\cdot)$. In
this way, we can get a sequence of quasi-conformal maps
$\phi_0,\phi_1,\phi_2,\cdots$, such that the following diagram
commutes.
$$
\xymatrix{& \cdots\ar[r]^{f }
&(\mathbb{\overline{C}},P)\ar[r]^{f}\ar[d]_{\phi_3}
 &(\mathbb{\overline{C}},P)\ar[r]^{f} \ar[d]_{\phi_2}
&(\mathbb{\overline{C}},P)\ar[r]^{f}\ar[d]_{\phi_1}
 & (\mathbb{\overline{C}},P)\ar[d]_{\phi_0}\\
&\cdots\ar[r]_{R }
 & (\mathbb{\overline{C}},Q)\ar[r]_{R}
&(\mathbb{\overline{C}},Q)\ar[r]_{R} &
(\mathbb{\overline{C}},Q)\ar[r]_{R} & (\mathbb{\overline{C}},Q)
 }
$$
One can verify that for any $k\geq0$, $\phi_{k+1}$ is isotopic to
$\phi_{k}$ rel $f^{-k}(P)$.

Fix some $\nu\in\Lambda$, let $D_\nu$ be the union of all rotation
disks of $(h_\nu,P_\nu)$ intersecting  $\partial S_\nu$. We set
$D_\nu=\emptyset$ if $\partial_0(S_\nu)=\emptyset$.

Choose an integer $\ell\geq p_\nu$ such that $\cup\Gamma\subset
f^{-\ell+p_\nu}(P)$, then we extend $\phi_\ell|_{S_\nu}$ to  a
quasi-conformal map $\Phi: \mathbb{\overline{C}}(S_\nu)\rightarrow
\mathbb{\overline{C}}$. We require that $\Phi$ is holomorphic in
$D_\nu$ if $D_\nu\neq\emptyset$.

 Note that
there is a unique component $E_\nu$ of $f^{-p_\nu}(S_\nu)$  parallel
to $S_\nu$. The choice of $\ell$ implies $\partial E_\nu\subset
f^{-p_\nu}(\cup \Gamma)\subset f^{-\ell}(P)$. By the construction of
$\phi_k$, we know that  $\phi_{\ell+p_\nu}$ and $\phi_\ell$ are
isotopic rel $f^{-\ell}(P)$. In particular,
$\phi_{\ell+p_\nu}|_{\partial E_\nu}=\phi_{\ell}|_{\partial
E_\nu}=\Phi|_{\partial E_\nu}$.

 Denote the components of
$\mathbb{\overline{C}}(S_\nu)-(E_{\nu}\cup D_\nu)$ by $\{U_j; j\in
I\}$, where $I$ is a finite index set. Each $U_j$ is a disk,
containing at most one point in $P_\nu$. For any $j\in I$, let
$V_j\Subset U_j$ be a disk such that % $V_j\cap P_\nu=U_j\cap P_\nu$
 $U_j\setminus {V_j}\subset f^{-\ell}(P)\setminus P$. By Measurable
Riemann Mapping Theorem, there is a quasi-conformal homeomorphism
$\Psi_j: V_j\rightarrow\Phi(V_j)$ whose Beltrami coefficient
satisfies $\mu_{\Psi_j}(z)=\mu_{\Phi\circ h_\nu}(z)$ for $z\in V_j$.
If $U_j$ contains a point $p\in P_\nu$ (then $V_j$ necessarily
contains $p$), we further require that $\Psi_j(p)=\Phi(p)$.

We can define a quasi-conformal map
$\Psi:\mathbb{\overline{C}}(S_\nu)\rightarrow \mathbb{\overline{C}}$
by

\begin{equation*}
\Psi(z)=\begin{cases}
\Phi(z),\ \  &z\in D_\nu,\\
 \phi_{\ell+p_\nu}(z),\ \  &z\in E_\nu,\\
\Psi_j(z),\ \ &z\in V_j, j\in I,\\
q.c \ interpolation, \ \ &z\in U_j\setminus V_j, j\in I.
\end{cases}
\end{equation*}

One may verify that $\Phi$ is homotopic to $\Psi$ rel $P_\nu$. Thus
$(h_\nu,P_\nu)$ is q.c-equivalent to $(g_\nu,Q_\nu):=(\Phi\circ
h_\nu\circ \Psi^{-1}, \Phi(P_\nu))$ via $(\Phi,\Psi)$. Moreover,
$(g_\nu,Q_\nu)$ is holomorphic outside $\Psi(\cup_{j\in
I}(U_j\setminus \overline{V_j}))$.

In the following, we will construct a $(g_\nu,Q_\nu)$-invariant
complex structure. For each $j\in I$, we may assume that the annulus
$U_j\setminus \overline{V_j}$ is thin enough such that for some
$k>1$ large enough, the set $g_\nu^k(\Psi(U_j\setminus
\overline{V_j}))$ is contained either in a rotation disk of
$(g_\nu,Q_\nu)$, or in a neighborhood of a   critical cycle (note
that  $(g_\nu,Q_\nu)$ is holomorphic near this cycle). Let
$k_j\geq1$ be the first integer so that $(g_\nu,Q_\nu)$ is
holomorphic in $g_\nu^{k_j}(\Psi(U_j\setminus \overline{V_j}))$.
Define a complex structure in $\Psi(U_j\setminus \overline{V_j})$ by
pulling back the standard complex structure in
$g_\nu^{k_j}(\Psi(U_j\setminus \overline{V_j}))$ via $g_\nu^{k_j}$.
Then we define a complex structure in $g_\nu^{-k}(\Psi(\cup_{j\in
I}(U_j\setminus \overline{V_j})))$ by pulling back the complex
structure in $\Psi(\cup_{j\in I}(U_j\setminus \overline{V_j}))$ via
$g_\nu^{k}$ for all $k\geq0$ and define the standard complex
 structure elsewhere. In this way, we get a $(g_\nu,Q_\nu)$-invariant
 complex structure $\sigma$. The Beltrami coefficient $\mu$ of
 $\sigma$ satisfies $\|\mu\|_\infty<1$ since $(g_\nu,Q_\nu)$ is holomorphic outside $\Psi(\cup_{j\in
I}(U_j\setminus \overline{V_j}))$.

 By Measurable Riemann Mapping
Theorem, there is a quasi-conformal map
$\zeta:\mathbb{\overline{C}}\rightarrow \mathbb{\overline{C}}$ whose
Beltrami coefficient is $\mu$.
% $\zeta^*(\sigma_0)=\sigma$, where
%$\sigma_0$ is the standard complex structure .
Let $f_\nu=\zeta\circ g_\nu\circ\zeta^{-1}$, then $f_\nu$ is a
rational map and $(h_\nu, P_\nu)$ is q.c-equivalent to $(f_\nu,
\zeta\circ\Phi(P_\nu))$ via $(\zeta\circ\Phi, \zeta\circ\Psi)$. See
the following commutative diagram.
$$
\xymatrix{
 & \mathbb{\overline{C}}(S_\nu)\ar[r]^{\Psi} \ar[d]_{h_\nu}
& \mathbb{\overline{C}}\ar[r]^{\zeta}\ar[d]_{g_\nu}
 & \mathbb{\overline{C}}\ar[d]_{f_\nu}\\
&\mathbb{\overline{C}}(S_\nu)\ar[r]_{\Phi} &
\mathbb{\overline{C}}\ar[r]_{\zeta} & \mathbb{\overline{C}}
 }
$$

\subsection{Proof of the sufficiency of  Theorem \ref{b4a} ($\Gamma=\emptyset$)} \label{3-3-2}

Since  $\Gamma=\emptyset$, for each $\mathcal{S}$-piece $S$, we have
$\partial(S)=\partial_0(S)\subset\Gamma_0$, where $\Gamma_0$ is  the
collection of $(f,P)$-periodic curves defined in Section \ref{3-1}.
It follows from Lemma \ref{b2a} that $S$ is $f_*$-periodic. So
$\mathcal{S}$ can be written as $\{f^j_*(S_\nu);0\leq j<p_\nu,
\nu\in\Lambda\}$. Moreover, any two $\mathcal{S}$-pieces  contained
in the same $f_*$-cycle have the same number of boundary curves.

Suppose that $(h_\nu,P_\nu)$ is q.c-equivalent to a rational map
$(R_\nu,Q_\nu)$ via a pair of quasi-conformal maps
$(\Phi_\nu,\Psi_\nu)$ for $\nu\in\Lambda=[1,n]$.

\vskip0.3cm
\noindent\textbf{Step 1: Getting partial holomorphic models. }{\it  For each
$\mathcal{S}$-piece $S$, there exist a pair of quasi-conformal maps
$(\Phi_S,\Psi_S):\mathbb{\overline{C}}(S)\rightarrow\mathbb{\overline{C}}$
and  a rational map $R_S$ such that the following diagram commutes:

$$
\xymatrix{& \mathbb{\overline{C}}(S)\ar[r]^{H_S} \ar[d]_{\Psi_S}
& \mathbb{\overline{C}}(f_*(S))\ar[d]^{\Phi_{f_*(S)}}\\
&\mathbb{\overline{C}}\ar[r]_{R_S} & \mathbb{\overline{C}}
 }
$$

}

It suffices to show that for each $f_*$-cycle $\langle S_\nu,\cdots,
f_*^{p_\nu-1}(S_\nu)\rangle$, there exist
 a sequence of quasi-conformal maps
$\Psi_{S_\nu}, \ \Phi_{f_*^k(S_\nu)}, 0\leq k< p_\nu$ and a sequence
of rational maps $R_{f_*^k(S_\nu)}, 0\leq k< p_\nu$ such that the
following diagram commutes

\bess \xymatrix{&\mathbb{\overline{C}}(S_\nu)\ar[r]^{H_{S_\nu}}
\ar[d]_{\Psi_{S_\nu}}
&\mathbb{\overline{C}}(f_*(S_\nu))\ar[r]^{H_{f_*(S_\nu)}}\ar[d]_{\Phi_{f_*(S_\nu)}}
% & \mathbb{\overline{C}}(f_*^2(S_\nu))\ar[r]^{}\ar[d]_{\Phi_{f_*^2(S_\nu)}}
  &\cdots \ar[r]^{} &
\mathbb{\overline{C}}(f_*^{p_\nu-1}(S_\nu))\ar[r]^{H_{f_*^{p_\nu-1}(S_\nu)}}\ar[d]_{\Phi_{f_*^{p_\nu-1}(S_\nu)}}
 & \mathbb{\overline{C}}(S_\nu)\ar[d]_{\Phi_{S_\nu}}\\
& \mathbb{\overline{C}}\ar[r]_{R_{S_\nu}}
 &\mathbb{\overline{C}}\ar[r]_{R_{f_*(S_\nu)}}
%&\mathbb{\overline{C}}\ar[r]_{}\cdots
 & \cdots \ar[r]_{}
&\mathbb{\overline{C}}\ar[r]_{R_{f_*^{p_\nu-1}(S_\nu)}}   &
\mathbb{\overline{C}}}\eess

The constructions of the two sequences of maps are as follows:
First, we set $\Phi_{S_\nu}=\Phi_\nu$ and $\Psi_{S_\nu}=\Psi_\nu$.
By Measurable Riemann Mapping Theorem, there is a quasi-conformal map
$\Phi_{f_*^{p_\nu-1}(S_\nu)}:\mathbb{\overline{C}}(f_*^{p_\nu-1}(S_\nu))\rightarrow
\mathbb{\overline{C}}$ such that
$\Phi_{f_*^{p_\nu-1}(S_\nu)}^*(\sigma_0)=(\Phi_{S_\nu}\circ
H_{f_*^{p_\nu-1}(S_\nu)})^*(\sigma_0)$, where $\sigma_0$ is the
standard complex structure.
 Then
$R_{f_*^{p_\nu-1}(S_\nu)}=\Phi_{S_\nu}\circ
H_{f_*^{p_\nu-1}(S_\nu)}\circ\Phi_{f_*^{p_\nu-1}(S_\nu)}^{-1}$ is a
rational map.

Inductively, for $i=p_\nu-2,\cdots,1$, we can get a quasi-conformal
map
$\Phi_{f^i_{*}(S_\nu)}:\mathbb{\overline{C}}(f_*^{i}(S_\nu))\rightarrow
\mathbb{\overline{C}}$ so that
$R_{f_*^{i}(S_\nu)}=\Phi_{f_*^{i+1}(S_\nu)}\circ
H_{f_*^{i}(S_\nu)}\circ\Phi_{f_*^{i}(S_\nu)}^{-1}$ is a rational
map.

Finally, we set $R_{S_\nu}=\Phi_{f_*(S_\nu)}\circ
H_{S_\nu}\circ\Psi_{S_\nu}^{-1}$. Then the relation
$R_\nu=R_{f_*^{p_\nu-1}(S_\nu)}\circ\cdots\circ R_{f_*(S_\nu)}\circ
R_{S_\nu}$ implies that $R_{S_\nu}$ is also a rational map.

 Set
$\Psi_{f^i_{*}(S_\nu)}=\Phi_{f^i_{*}(S_\nu)}$ for $1\leq i<p_\nu$.
The pair of quasi-conformal maps
$(\Phi_{f^i_{*}(S_\nu)},\Psi_{f^i_{*}(S_\nu)})$ and the rational map
$R_{f^i_{*}(S_\nu)}  (0\leq i<p_\nu)$  are as required.

\vskip0.3cm
\noindent\textbf{Step 2: Gluing holomorphic models. } For each
$\mathcal{S}$-piece $S$, recall that $E_S$ is the unique
$\mathcal{E}$-piece parallel to $S$. Since $\Gamma=\emptyset$, each
boundary curve of $S$ is also a boundary curve of $E_S$. So each
component of $S-E_S$ is a disk, containing at most one point in $P$.
 Let
$\{U_k; k\in I_S\}$ be the collection of all components of
$S\setminus E_S$, where $I_S$ is the finite index set induced by
$S$. For any $k\in I_S$, let $V_k\Subset U_k$ be a disk such that
$U_k\setminus {V_k}\subset f^{-1}(P)\setminus P$ (this implies
$V_k\cap P=U_k\cap P$). By the Measurable Riemann Mapping Theorem,
there is a quasi-conformal homeomorphism $\phi_{k}:V_k\rightarrow
\Psi_S(V_k)$ whose Beltrami coefficient satisfies

$$\mu_{\phi_{k}}(z)=\sum_{\mathcal{E}\ni E\subset U_k}\chi_E(z)\mu_{\Phi_{f(E)}\circ f}(z), \ z\in V_k.$$
Here the sum is taken over all the $\mathcal{E}$-pieces
contained in $U_k$. If $V_k$ contains a point $p\in P$, we further
require that $\phi_{k}(p)=\Phi_S(p)$.

Now we define a quasi-conformal homeomorphism
$\psi_S:S\rightarrow\Phi_S(S)$ by

\begin{equation*}
\psi_S(z)=\begin{cases}
 \Psi_S(z),\ \  &z\in E_S,\\
\phi_{k}(z),\ \ &z\in V_k, k\in I_S,\\
q.c \ interpolation, \ \ &z\in U_k\setminus V_k, k\in I_S.
\end{cases}
\end{equation*}

Define a quasi-conformal map $\Theta:\mathbb{\overline{C}}\rightarrow
\mathbb{\overline{C}}$ by $\Theta|_S=\psi_S^{-1}\circ \Phi_S$ for
all $S\in\mathcal{S}$. The map $\Theta$ is isotopic to the identity
map rel $P$. Let
$\Phi:\mathbb{\overline{C}}\rightarrow\mathbb{\overline{C}}$ be a
quasi-conformal map such that
$$\mu_\Phi(z)=\sum_{S\in\mathcal{S}}\chi_S(z)\mu_{\Phi_S}(z), \  z\in\mathbb{\overline{C}}.$$

%and define $g=\Phi\circ f\circ \Theta\circ \Phi^{-1}$.
Let $\Psi=\Phi\circ\Theta^{-1}$. The  pair of quasi-conformal maps
$(\Phi,\Psi)$ can be considered to be the gluing of
$(\Phi_S|_S,\Psi_S|_S)_{S\in \mathcal{S}}$. In this way, $(f,P)$ is
q.c-equivalent to the Herman map $(g,Q):=(\Phi\circ f\circ\Psi^{-1},
\Phi(P))$ via $(\Phi,\Psi)$.

 %will
%show in following that $g$ is holomorphic outside
%$X:=\Psi(\cup_{S\in\mathcal{S}} \cup_{k\in \Lambda_S}(U_k\setminus
%V_k))$. To see this, we consider $g|_{\Phi(S)}$ for some fixed
%$\mathcal{S}$-piece $S$. Let $E$ be an $\mathcal{E}$-piece that is
%contained in $S$, then \bess
%g&=&(\Phi\circ\Phi_{f(E)}^{-1})\circ(\Phi_{f(E)}\circ
%f\circ\Theta\circ\Phi_S^{-1})\circ(\Phi_S\circ\Phi^{-1})\\
%&=&(\Phi\circ\Phi_{f(E)}^{-1})\circ(\Phi_{f(E)}\circ
%f\circ\psi_S^{-1})\circ(\Phi_S\circ\Phi^{-1})
% \eess

\vskip0.3cm
\noindent\textbf{Step 3: Applying quasi-conformal surgery.} We first show
that the Herman map  $(g,Q)$ is holomorphic in most parts of
$\mathbb{\overline{C}}$. In fact, it is holomorphic outside
$X:=\Psi(\cup_{S\in\mathcal{S}} \cup_{k\in I_S}(U_k\setminus V_k))$.
To see this, we fix some $\mathcal{S}$-piece $S$. The restriction
$g|_{\Psi(E_S)}$ can be decomposed into %:\Psi(E_S)\rightarrow \Phi(f(E_S))
$$g|_{\Psi(E_S)}=(\Phi\circ\Phi_{f(E_S)}^{-1})\circ(\Phi_{f(E_S)}\circ
f\circ\Psi_S^{-1})|_{\Psi_S(E_S)}\circ(\Phi_S\circ\Phi^{-1})|_{\Psi(E_S)}.$$

For any $k\in I_S$, any $\mathcal{E}$-piece $E \subset U_k$, the
restriction $g|_{\Psi(V_k\cap E)}$  can be decomposed into
$$g|_{\Psi(V_k\cap E)}=(\Phi\circ \Phi_{f(E)}^{-1})\circ(\Phi_{f(E)}\circ
f\circ\phi_{k}^{-1})|_{\phi_{k}(V_k\cap
E)}\circ(\Phi_S\circ\Phi^{-1})|_{\Psi(V_k\cap E)}.$$

In either case, each factor of the decompositions of $g$ is
holomorphic in its domain of definition. So $g|_S$ is holomorphic
outside $\Psi(\cup_{k\in I_S}(U_k\setminus V_k))$.
%When $S$ ranges over all $\mathcal{S}$-pieces,
It follows that $(g,Q)$ is holomorphic outside $X$.

Let $R_A$ be the union of all rotation annuli of $g$. Then one can
check that $X\subset g^{-1}(R_A)\setminus R_A$. Let $\sigma_0$ be
the standard complex structure in $\mathbb{\overline{C}}$, we define
a $(g,Q)$-invariant complex structure $\sigma$ by

\begin{equation*}
\sigma=\begin{cases}
 (g^k)^*(\sigma_0),\ \  &\text{ in } g^{-k}(R_{A})\setminus g^{-k+1}(R_A), \ k\geq1,\\
\sigma_0,\ \ &\text{ in } \mathbb{\overline{C}}-\cup_{k\geq1}(
g^{-k}(R_A)\setminus g^{-k+1}(R_A)).
\end{cases}
\end{equation*}

Since $(g,Q)$ is  holomorphic outside $X$, the Beltrami coefficient
$\mu$ of $\sigma$ satisfies $\|\mu\|_\infty<1$. By Measurable
Riemann Mapping Theorem, there is a quasi-conformal map
$\zeta:\mathbb{\overline{C}}\rightarrow \mathbb{\overline{C}}$ such
that $\zeta^*(\sigma_0)=\sigma$. Let $R=\zeta\circ
g\circ\zeta^{-1}$, then $R$ is a rational map and $(f, P)$ is
q.c-equivalent to $(R, \zeta\circ\Phi(P))$ via
$(\zeta\circ\Phi,\zeta\circ\Psi)$. \hfill $\Box$

\section{Realization part II: general case} \label{3-3-3}

In this section,  we prove the sufficiency   of Theorem \ref{b4a} in
the more general case $\Gamma\neq\emptyset$. This is the technical
part. We assume in this section that $\Gamma\neq\emptyset$,
$\lambda(\Gamma,f)<1$ and for each $\nu\in \Lambda=[1,n]$, the map
$(h_\nu,P_\nu)$ is q.c-equivalent to a rational map, we will show
that $(f,P)$ is q.c-equivalent to a rational map.

 The idea
 is to glue the holomorphic models of $(h_k,P_k)_{1\leq k\leq n}$ along the curves in $\Gamma\cup \Gamma_0$, similar to
Section \ref{3-3-2}. But this section provides very interesting and
technical flavor because of the algebraic condition
$\lambda(\Gamma,f)<1$. In most part of this section, we deal with
this condition and we will show that it is actually
equivalent to the Gr\"otzsch inequality in the holomorphic setting.
Thus it enables us to glue the partial holomorphic models of $(f,P)$
along $\Sigma$ (in a suitable fashion) into a branched covering
$(g,Q)$, holomorphic in most part of $\mathbb{\overline{C}}$ and
q.c-equivalent to $(f,P)$. The last step is similar to the previous
sections, it is a quasi-conformal surgery procedure.

\subsection{The algebraic condition $\lambda(\Gamma,f)<1$}\label{a1}

To begin, we recall a result on non-negative matrix.
 Let $W$ be a $m\times m$  non-negative  square matrix (i.e. each entry is
a non-negative real number). It's known from Perron-Frobenius
Theorem that the spectral radius of $W$ is an eigenvalue of  $W$,
named the {\it leading eigenvalue}. Let
$v=(v_1,\cdots,v_m)^t\in\mathbb{R}^m$ be a vector, we say $v>0$ if
for each $i$, $v_i>0$.
% The following Lemma can be found in
%\cite{CT1}, Lemma A.1.

\begin{lem}[\cite{CT1}, Lemma A.1]\label{b4ccc}
 Let $W$ be a non-negative square matrix with
leading eigenvalue $\lambda$. Then $\lambda<1$ iff there is a vector
$v>0$ such that $Wv<v$.
\end{lem}

%Lemma \ref{b4ccc} give a explanation of what the algebraic condition
%$\lambda<1$ means.

With the help of Lemma \ref{b4ccc}, we turn to our discussion.
First,   $\lambda(\Gamma,f)<1$ implies $Wv<v$, where $W$ is the
$(f,P)$-transition matrix of $\Gamma$  and $v\in \mathbb{R}^\Gamma$
is a positive vector. That is, there is a positive function
$v:\Gamma\rightarrow \mathbb{R}^+$ such that for any
$\gamma\in\Gamma$,
$$(Wv)(\gamma)=\sum_{\beta\in \Gamma}
\sum_{\alpha\sim\gamma}\frac{v(\beta)}{{\rm
deg}(f:\alpha\rightarrow\beta)}< v(\gamma),$$ where the second
sum is taken over all components $\alpha$ of $f^{-1}(\beta)$
 homotopic to $\gamma$ in $\mathbb{\overline{C}}-P$.

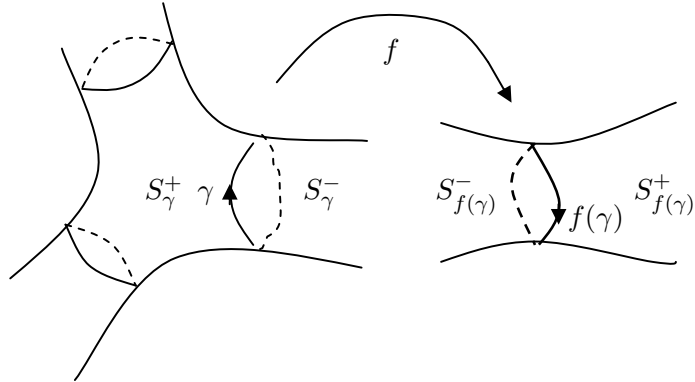
\begin{figure}[h]
\centering{
\includegraphics[height=6cm]{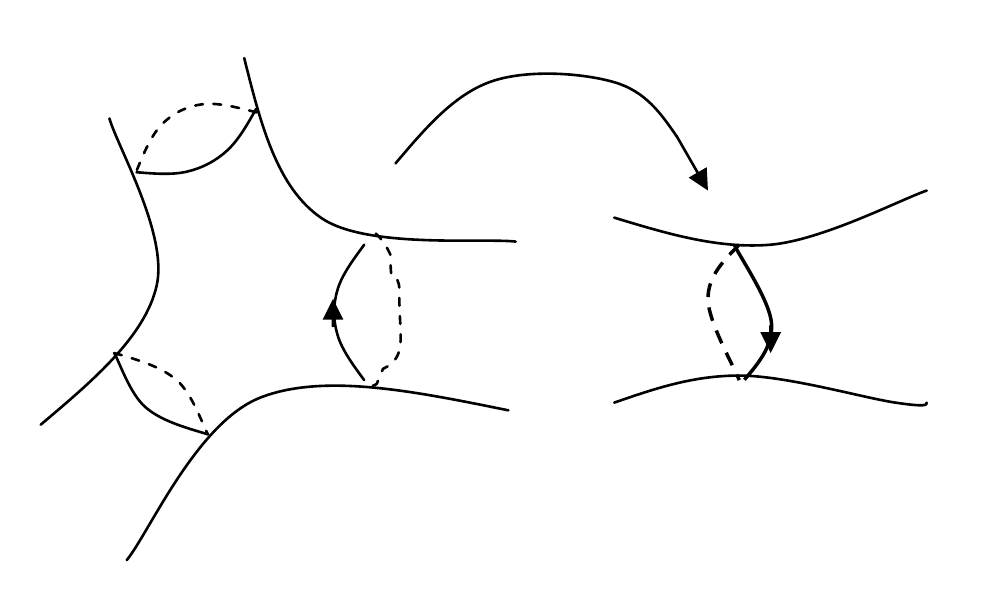}
  \put(-220,80){$S_\gamma^+$} \put(-200,80){$\gamma$}
  \put(-130,133){$f$} \put(-160,80){$S_\gamma^-$}
  \put(-110,80){$S_{f(\gamma)}^-$} \put(-60,70){$f(\gamma)$} \put(-35,80){$S_{f(\gamma)}^+$}
 \caption{Orientation and labeling}}
\end{figure}%\label{fig5-1}

Recall that for each curve $\gamma\in \Sigma$, there exist exactly
two $\mathcal{S}$-pieces, say $S_\gamma^+$ and $S_\gamma^-$, such
that $S_\gamma^+\cap S_\gamma^-=\gamma$. For each curve $\gamma\in
\Sigma$, we can associate an orientation preserved by $f$.
 % we have already associated an orientation for each $\gamma\in\Sigma$, such that $f$ preserves
%the orientation.
 We may assume that the notations  $S_\gamma^+$ and
$S_\gamma^-$  are chosen such that $S_\gamma^+$ lies on the left
side of $\gamma$ and $S_\gamma^-$ lies on the right side of
$\gamma$.
%for each $\gamma\in\Gamma$.

Here, we borrow some notations from Lemma \ref{b2a0}. Recall that
$\Gamma_0$ is the collection of the $(f,P)$-periodic curves that
generates $\Gamma$. For
$n\geq1$, the set $\Gamma_n$ is defined by $\Gamma_n=\{\gamma\in \Gamma;  n \text{ is
the first integer such that } f^n(\gamma)\in \Gamma_0\}$.

One may verify that   if $\delta\in f^{-1}(\Gamma)$ is homotopic to
a curve $\gamma\in\Gamma$ in $\mathbb{\overline{C}}-P$, then
$\delta$ is necessarily contained in $S^+_\gamma\cup S^-_\gamma$.
One may verify that if $\gamma\in\Gamma_1$, then $\delta\neq \gamma$; if
$\gamma\in\Gamma_k$ for some $k\geq 2$, it can happen that
$\delta=\gamma$.

For each curve $\gamma\in \Gamma=\bigcup_{n\geq1}\Gamma_n$, we will
associate two positive numbers $\rho(S_\gamma^+,\gamma)$ and
$\rho(S_\gamma^-,\gamma)$ inductively, as follows:

If $\gamma\in \Gamma_1$, we choose two positive numbers
$\rho(S_\gamma^+,\gamma)$ and $\rho(S_\gamma^-,\gamma)$ such that
$$\rho(S_\gamma^+,\gamma)+\rho(S_\gamma^-,\gamma)=1,$$
$$\sum_{\beta\in \Gamma}
\sum_{\alpha\sim\gamma, \alpha\subset
S_\gamma^\omega}\frac{v(\beta)}{{\rm
deg}(f:\alpha\rightarrow\beta)}<
v(\gamma)\rho(S_\gamma^\omega,\gamma), \ \omega\in \{\pm\}.$$

Suppose that for each curve $\alpha\in  \Gamma_1\cup\cdots\cup
\Gamma_k$, we have already chosen two numbers
$\rho(S_\alpha^+,\alpha)$ and $\rho(S_\alpha^-,\alpha)$. Then for
$\gamma\in \Gamma_{k+1}$ (note
 that  $f(\gamma)\in \Gamma_{k}$), we can find two positive
numbers $\rho(S_\gamma^+,\gamma)$ and $\rho(S_\gamma^-,\gamma)$ such
that:
$$\rho(S_\gamma^+,\gamma)+\rho(S_\gamma^-,\gamma)=1,$$
$$\frac{v(f(\gamma))}{{\rm
deg}(f|_{\gamma})}\rho(S_{f(\gamma)}^\omega, f(\gamma))
+\sum_{\beta\in \Gamma} \sum_{\alpha\sim\gamma, \alpha\subset
S_\gamma^\omega\setminus\gamma}\frac{v(\beta)}{{\rm
deg}(f:\alpha\rightarrow\beta)}<
v(\gamma)\rho(S_\gamma^\omega,\gamma),\  \omega\in \{\pm\}.$$ In
fact, we can take
$$\rho(S_\gamma^\omega,\gamma)=\frac{\displaystyle\frac{v(f(\gamma))}{{\rm
deg}(f|_{\gamma})}\rho(S_{f(\gamma)}^\omega, f(\gamma))
+\sum_{\beta\in \Gamma} \sum_{\alpha\sim\gamma, \alpha\subset
S_\gamma^\omega\setminus\gamma}\frac{v(\beta)}{{\rm
deg}(f:\alpha\rightarrow\beta)}}{\displaystyle\sum_{\beta\in \Gamma}
\sum_{\alpha\sim\gamma}\frac{v(\beta)}{{\rm
deg}(f:\alpha\rightarrow\beta)}},\  \omega\in \{\pm\}.$$

\subsection{Equipotentials in the marked disks of rational
maps}\label{a2}

Suppose that $(f,P)$  is either a Thurston rational  map or a Siegel
rational map, with a non-empty Fatou set. Recall that $P$ is a
marked set containing the postcritical set $P_f$. Then each periodic
Fatou component is either a superattracting domain or a Siegel disk.
If $f$ has a superattracting Fatou component $D$, then every Fatou
component $\Delta$ which is eventually mapped onto $D$ can be marked
by a unique pre-periodic point $a\in \Delta$. We call
$(\Delta,a)$ a I-type marked disk of $f$. Note that every
equipotential  in a superattracting Fatou component corresponds to a
round circle in B\"{o}ttcher coordinate. If $f$ has a Siegel disk
$D$, then it is known that the boundary $\partial D$ is contained in
the postcritical set $P_f$. Let $z_0$ be the center of the Siegel
disk $D$, the intersection $P\cap(D-\{z_0\})$ is either empty or
consists of finitely many $(f,P)$-periodic analytic curves. Let
$D_0\subset D$ be the component of
$\mathbb{\overline{C}}-(P\setminus\{z_0\})$ containing $z_0$.
 For
any $k\geq0$ and any component $\Delta$ of $f^{-k}(D_0)$, one can
verify that $\Delta$ is a disk and there is a unique point $a\in
\Delta\cap f^{-k}(z_0)$. We call $(\Delta, a)$ a II-type marked disk
of $(f,P)$.

In this way, for each Fatou component, we can associate a marked disk $(\Delta,a)$.
 An equipotential
$\gamma$ of  $(\Delta, a)$ is an analytic curve that is mapped to a
round circle with center zero under a Riemann mapping
$\phi:\Delta\rightarrow\mathbb{D}=\{|z|<1\}$ with $\phi(a)=0$. The potential
$\varpi(\gamma)$ of $\gamma$ is defined to be
$\textrm{mod}(A(\partial\Delta,\gamma))$, the modulus of the annulus
between $\partial\Delta$ and $\gamma$. One may check that these
definitions are independent of  the choice of the map $\phi$.
% the
%postcritical set either is disjoint from $D$ ir intersect with $D$
%5in finitely many $f$-periodic Jordan curves. Let $z\in D$ be the
%center of the Siegel disk and $D_0\subset D$ be the component of
%$\mathbb{\overline{C}}-P_f$ which contains $z$. Notice that if
%$P_f\cap D\neq\emptyset$, then $D_0$ is a proper subset of $D$.

%In the following, we will

\subsection{A positive function}\label{a3}

For each curve $\gamma\in \Sigma$, we associate an open annular
neighborhood $A^\gamma$ of $\gamma$. The annulus $A^\gamma$ is
chosen as follows: If $\gamma\in \Gamma_0$, we take  $A^\gamma$ as
a proper subset of the rotation annulus containing $\gamma$ such
that $f(A^\gamma)=A^{f(\gamma)}$ and $\overline{A^\gamma}\cap
f(P-\cup\mathcal{A})=\emptyset$. If $\gamma\in \Gamma_k$ for some
$k\geq1$, then $A^\gamma$ is the component of
$f^{-k}(A^{f^k(\gamma)})$ containing $\gamma$.

We define \bess \ &\mathcal{S}^\star=\{U;\ U \text{ is a connected
component of } \mathbb{\overline{C}}-\cup_{\gamma\in
\Sigma}A^\gamma\},& \ \\
\ &\mathcal{E}^\star=\{{V};\  V \text{ is a connected component of }
\mathbb{\overline{C}}-f^{-1}(\cup_{\gamma\in \Sigma}A^\gamma)\}.& \
\eess
 Each element of $\mathcal{S}^\star$ (resp.
$\mathcal{E}^\star$) is called an $\mathcal{S}^\star$-piece (resp.
$\mathcal{E}^\star$-piece).
 We will use $S^\star$ (resp. $E^\star$) to
denote an $\mathcal{S}^\star$-piece (resp.
$\mathcal{E}^\star$-piece).
 We remark that if we use $S$
to denote an $\mathcal{S}$-piece, then the notation $S^*$ means the
unique $\mathcal{S}^\star$-piece  contained in $S$; on the other
hand, if we use $S^*$ to denote an $\mathcal{S}^*$-piece, then the
notation $S$ means the unique $\mathcal{S}$-piece  containing $S^*$.
The convention also applies to the $\mathcal{E}$-pieces and
$\mathcal{E}^\star$-pieces.

 Similarly as in Section \ref{3-1}, we define
$E_{S^\star}$ to be the unique $\mathcal{E}^\star$-piece parallel to
$S^\star$. The map $f_*:\mathcal{S}^\star\rightarrow
\mathcal{S}^\star$ is defined by $f_*(S^\star)=f(E_{S^\star})$. The
marked sphere  $\mathbb{\overline{C}}(S^\star)$,  the marked disk
extension
$H_{S^\star}:\mathbb{\overline{C}}(S^\star)\rightarrow\mathbb{\overline{C}}(f_*(S^\star))$,
the marked set $P(S^\star)$ and also the sets
$\partial_0(S^\star),\partial_1(S^\star),\partial_2(S^\star)$ are
defined in the same way. Set
$$h^\star_\nu=H_{f_*^{p_\nu-1}(S^\star_\nu)}\circ\cdots\circ
H_{f_*(S^\star_\nu)}\circ H_{S^\star_\nu}, \
{P}^\star_\nu=P(S^\star_\nu), \ 1\leq \nu\leq n.$$

Consider the maps  $(h_\nu,P_\nu)$ and
$({h}^\star_\nu,{P}^\star_\nu)$ for $\ 1\leq \nu\leq n$. It is clear
that

$\bullet$  $(h_\nu,P_\nu)$  has no Thurston obstructions iff
$({h}^\star_\nu,{P}^\star_\nu)$ has no Thurston obstructions.

$\bullet$  $(h_\nu,P_\nu)$  is q.c-equivalent to a rational map iff
 $({h}^\star_\nu,{P}^\star_\nu)$  is q.c-equivalent to a rational
map.

We will use $({h}^\star_\nu,{P}^\star_\nu)$ in place of
$({h}_\nu,{P}_\nu)$ in the following discussions.  This will allow us to
construct deformations in a neighborhood of each curve $\gamma\in
\Sigma$. The advantage of this replacement will be seen in the last
step of the proof of Theorem \ref{b4a} (see Section \ref{aab}) where
we  apply the quasi-conformal surgery to glue all holomorphic models
together to obtain a rational realization of $(f,P)$.

\

\begin{figure}[h]
\centering{
\includegraphics[height=6cm]{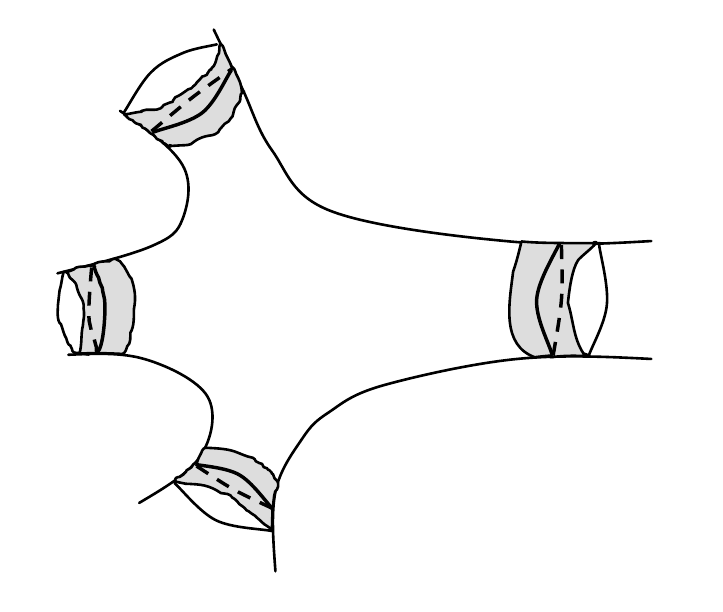}  \put(-90, 30){$S$}
 \put(-163, 80){$\beta_2$} \put(-75, 80){$\alpha_4$}  \put(-52, 105){$\gamma_4$}
 \put(-30, 80){$\beta_4$} \put(-110, 80){$S^\star$}
\put(-170, 150){$\alpha_1$} \put(-147, 120){$\beta_1$} \put(-135,
150){$\gamma_1$}\put(-200, 80){$\alpha_2$}\put(-180,
102){$\gamma_2$} \put(-150, 15){$\alpha_3$} \put(-135,
47){$\beta_3$} \put(-125, 30){$\gamma_3$}
 %%\put(-230, 30){$S$}
%\put(-250,180){$E_1$}  \put(-263,140){$S_1$}
% \put(-130,200){$E_2$}
%\put(-119,140){$S_2$}
% \put(-273,55){$E_3$}
%\put(-263,5){$S_3$} \put(-260,45){$\bullet \ p_3$}
%\put(-260,65){$\bullet \ q_3$} \put(-100,35){$E_4$}
%\put(-130,5){$S_4$} \put(-130,35){$\bullet \ p_4$}
%\put(-130,65){$\bullet \ q_4$}
 \caption{An $\mathcal{S}$-piece $S$ with boundary curves $\gamma_1,\gamma_2,\gamma_3, \gamma_4$. $S$ contains an $\mathcal{S}^\star$-piece
 $S^\star$, whose boundary curves are
 $\beta_1,\beta_2,\beta_3$ and $\alpha_4$.}}
\end{figure}
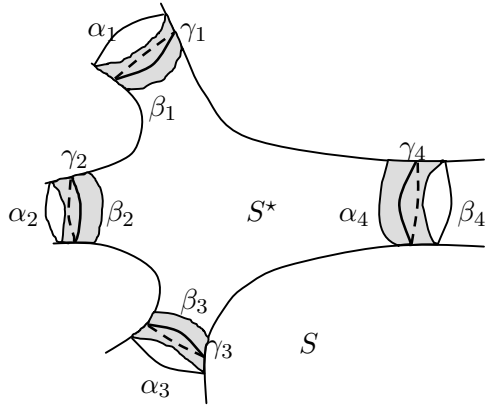%\label{fig5-1}

For each curve $\gamma\in \Sigma$, let $\alpha_\gamma,\beta_\gamma$
be the two boundary curves of $A^\gamma$.  Define
$$\Sigma^\star=\{\alpha_\gamma,\beta_\gamma; \gamma\in \Sigma\},\ \
\Gamma^\star=\{\alpha_\gamma,\beta_\gamma; \gamma\in \Gamma\},\ \
\Gamma_k^\star=\{\alpha_\gamma,\beta_\gamma; \gamma\in \Gamma_k\},
k\geq0.$$ We define a map $\pi: \Sigma^\star\rightarrow\Sigma$ by
$\pi(\alpha)=\gamma$ if $\alpha$ is a boundary curve of $A^\gamma$.
Obviously, for each curve $\gamma\in\Sigma$, we have
$\pi^{-1}(\gamma)=\{\alpha_\gamma,\beta_\gamma\}$. For 
$\delta\in\Sigma^\star$, let $S_\delta$ (resp. $S^\star_\delta$) be
the unique $\mathcal{S}$-piece (resp. $\mathcal{S}^\star$-piece)
containing $\delta$.

\

Now we define a positive function $\sigma_t: \Sigma^\star\rightarrow
\mathbb{R}^+$, where $t$ is a positive parameter, as follows:

First, we consider $\gamma\in \Gamma_0^\star$. In this case, some iterate
$f^k(\gamma)$ is contained in the rotation disk $\Delta$ of some Siegel
map $(h^\star_\nu,P^\star_\nu)$. Note that there is a largest open
annulus $A\subset \Delta$ such that

$\bullet$ the inner boundary of $A$ is $f^k(\gamma)$,

%$\bullet$ the outer boundary of $A$ is a
%$(h^\star_\nu,P^\star_\nu)$-periodic curve in the marked set
%$P^\star_\nu$,

$\bullet$ $A\cap h^\star_\nu(P^\star_\nu-R_D)=\emptyset$, where
$R_D$ is the union of all rotation disks of
$(h^\star_\nu,P^\star_\nu)$.

 We define
$\sigma_t(\gamma)$ to be the modulus of $A$. By definition,
$\sigma_t(\gamma)=\sigma_t(f(\gamma)).$

Now, we consider $\gamma\in \Gamma^\star$. In this case, $\gamma\in
\partial_1(S^\star_\gamma)\cup \partial_2(S^\star_\gamma)$. If
$\gamma\in\partial_1(S^\star_\gamma)$, we define
$$\sigma_t(\gamma)=t\cdot\rho(S_\gamma, \pi(\gamma))\cdot v(\pi(\gamma)).$$
If $\gamma\in\partial_2(S^\star_\gamma)$, we define
\begin{equation*}
\sigma_t(\gamma)=\begin{cases}
\dfrac{\sigma_t(f(\gamma))}{{\rm deg}(f|_\gamma)},\ \  &\text{if } S^\star_\gamma\in \mathcal{S}^\star-\{S^\star_1,\cdots,S^\star_n\},\\
t\cdot\rho(S_\gamma, \pi(\gamma))\cdot v(\pi(\gamma)),\ \ &\text{if
} S^\star_\gamma\in \{S^\star_1,\cdots,S^\star_n\}.
 \end{cases}
 \end{equation*}

In  this way, for all curves $\gamma\in \Sigma^\star$, the quantity
$\sigma_t(\gamma)$ is well-defined.

\begin{lem}\label{b4c}  When $t$ is large enough, the function $\sigma_t:
\Sigma^\star\rightarrow \mathbb{R}^+$ satisfies:

1. For any $\gamma\in \Gamma^\star$, $\sigma_t(\gamma)\leq
t\cdot\rho(S_\gamma, \pi(\gamma))\cdot v(\pi(\gamma))$.

2. For every $\gamma\in \Sigma$, suppose that
$\pi^{-1}(\gamma)=\{\alpha_\gamma,\beta_\gamma\}$. Then
\begin{equation*}
\sigma_t(\alpha_\gamma)+\sigma_t(\beta_\gamma)\leq \begin{cases}
tv(\gamma),\ \ &\text{if } \gamma\in
\Gamma,\\
{\rm mod}(A_\gamma),\ \ &\text{if } \gamma\in \Gamma_0,
 \end{cases}
 \end{equation*}
 where $A_\gamma$ is
the rotation annulus of $(f,P)$ that contains $\gamma$ if $\gamma\in
\Gamma_0$.

3. For every $\gamma\in \Gamma^\star$, if $\gamma\in
\partial_1(S^\star_\gamma)$, then we have the following
inequality:
$$\sum_{\beta\in \Gamma^\star}
\sum_{\alpha\sim\gamma, \alpha\subset
S^\star_\gamma}\frac{\sigma_t(\beta)}{{\rm
deg}(f:\alpha\rightarrow\beta)}<\sigma_t(\gamma),$$ where the second
sum is taken over all components of $f^{-1}(\beta)$  contained
in $S^\star_\gamma$ and homotopic to $\gamma$ in
$\mathbb{\overline{C}}-P$ .

\end{lem}

\begin{proof} 1. Notice that if $\gamma\in \Gamma^\star$, then
$\gamma\in
\partial_1(S^\star_\gamma)\cup\partial_2(S^\star_\gamma)$.
If $\gamma\in
\partial_1(S^\star_\gamma)$ or $S^\star_\gamma\in\{S^\star_1,\cdots,S^\star_n\}$, then by evaluation,
$\sigma_t(\gamma)= t\rho(S_\gamma, \pi(\gamma))v(\pi(\gamma))$. Now
suppose $\gamma\in
\partial_2(S^\star_\gamma)$ and $S^\star_\gamma\in\mathcal{S}^\star-\{S^\star_1,\cdots,S^\star_n\}$.
Let $p\geq1$ be the first integer such that
$f_*^p(S^\star_\gamma)\in \{S^\star_1,\cdots,S^\star_n\}$. There is
a largest number $k\in \{0,\cdots, p\}$ such that $f^j(\gamma)\in
\partial_2(f^j_*(S^\star_\gamma))$ for $0\leq j<k$. Thus we have
$$\sigma_t(\gamma)=\frac{\sigma_t(f(\gamma))}{{\rm deg}(f|_\gamma)}=
\cdots=\frac{\sigma_t(f^k(\gamma))}{{\rm deg}(f^k|_\gamma)}.$$

If $f^k(\gamma)\in
\partial_0(f^k_*(S^\star_\gamma))$, then $\sigma_t(f^k(\gamma))$
is a constant independent of $t$, thus  $\sigma_t(\gamma)\leq
t\rho(S_\gamma, \pi(\gamma))v(\pi(\gamma))$ when $t$ is large.

If $f^k(\gamma)\in
\partial_1(f^k_*(S^\star_\gamma))$, then
$$\sigma_t(\gamma)=\frac{t\cdot\rho(S_{f^k(\gamma)},\pi(f^k(\gamma)))\cdot v(\pi(f^k(\gamma)))}{{\rm
deg}(f^k|_\gamma)}.$$

By the choice of the numbers $\{\rho(S_\gamma^+,\gamma),
\rho(S_\gamma^-,\gamma); \gamma\in\Gamma\}$, we see that for any
curve $\beta\in\Gamma-\Gamma_1=\cup_{j\geq2}\Gamma_j$,

$$\frac{v(f(\beta))\rho(S_{f(\beta)}^\omega, f(\beta))}{{\rm
deg}(f|_{\beta})} < v(\beta)\rho(S_\beta^\omega,\beta),\ \omega\in
\{\pm\}.$$

 Since for each $\gamma\in \Gamma^\star$, ${\rm
deg}(f|_\gamma)={\rm deg}(f|_{\pi(\gamma)})$, we have that \bess
 \sigma_t(\gamma)&<&\frac{t\rho(S_{f^{k-1}(\gamma)},\pi(f^{k-1}(\gamma)))
 v(\pi(f^{k-1}(\gamma)))}{ {\rm deg}(f^{k-1}|_\gamma)}<\cdots\\
 &<&\frac{t\rho(S_{f(\gamma)},\pi(f(\gamma)))
 v(\pi(f(\gamma)))}{ {\rm deg}(f|_\gamma)}<t\rho(S_{\gamma},\pi(\gamma))
 v(\pi(\gamma)).
 \eess

If $f^k(\gamma)\in
\partial_2(f^k_*(S^\star_\gamma))$, then we have $k=p$ by
the choice of $k$ and
%2). For all $j=\{0,\cdots, p-1\}$,
%$f^j(\gamma)\in
%\partial_2(f^j_*(S^\star_\gamma))$. Then $f^p(\gamma)\in
%\partial(f^p_*(S^\star_\gamma))$ and $\sigma_t(\gamma)=
%\frac{\sigma_t(f^p(\gamma))}{{\rm deg}(f^p|_\gamma)}.$ If
%$f^p(\gamma)\in \partial_0(f^p_*(S^\star_\gamma))$, then
%$\sigma_t(f^p(\gamma))$ is a constant independent of $t$. It follows
%that $\sigma_t(\gamma)\leq t\rho(S_\gamma,
%\pi(\gamma))v(\pi(\gamma))$ when $t$ is large enough. If
%$f^p(\gamma)\in
%\partial_1(f^p_*(S^\star_\gamma))\cup
%\partial_2(f^p_*(S^\star_\gamma))$, then
$$\sigma_t(\gamma)=\frac{t\rho(S_{f^p(\gamma)},\pi(f^p(\gamma)))v(\pi(f^p(\gamma)))}{{\rm
deg}(f^p|_\gamma)}.$$ With the same argument as above, we have
$\sigma_t(\gamma)< t\rho(S_\gamma, \pi(\gamma))v(\pi(\gamma))$.

2. The conclusion  follows from 1 and  the definition of
$\sigma_t$.

3. We verify the inequality directly, as follows:

\bess &&\sum_{\beta\in \Gamma^\star} \sum_{\alpha\sim\gamma,
\alpha\subset S^\star_\gamma}\frac{\sigma_t(\beta)}{{\rm
deg}(f:\alpha\rightarrow\beta)}\\
&=& \sum_{\beta\in \Gamma^\star} \sum_{\alpha\sim\gamma,
\alpha\subset
S^\star_\gamma\setminus{\gamma}}\frac{\sigma_t(\beta)}{{\rm
deg}(f:\alpha\rightarrow\beta)}+\frac{\sigma_t(f(\gamma))}{{\rm
deg}(f|_\gamma)}\\
&\leq & \sum_{\beta\in \Gamma^\star} \sum_{\alpha\sim\gamma,
\alpha\subset
S^\star_\gamma\setminus{\gamma}}\frac{\sigma_t(\beta)}{{\rm
deg}(f:\alpha\rightarrow\beta)}+\frac{ t\rho(S_{f(\gamma)},
\pi({f(\gamma)}))v(\pi({f(\gamma)}))}{{\rm deg}(f|_\gamma)} \ \ (By \ 1)\\
&=& \sum_{\delta\in \Gamma}\sum_{\zeta\in\pi^{-1}(\delta)}
\sum_{\alpha\sim\gamma, \alpha\subset
S^\star_\gamma\setminus{\gamma}}\frac{\sigma_t(\zeta)}{{\rm
deg}(f:\alpha\rightarrow\zeta)}+\frac{ t\rho(S_{f(\gamma)},
\pi({f(\gamma)}))v(\pi({f(\gamma)}))}{{\rm deg}(f|_\gamma)}\\
&=&\sum_{\delta\in \Gamma}\sum_{\alpha\sim\pi(\gamma), \alpha\subset
{S}_\gamma\setminus{\pi(\gamma)}}
\frac{\sum_{\zeta\in\pi^{-1}(\delta)}\sigma_t(\zeta)}{{\rm
deg}(f:\alpha\rightarrow\delta)}+\frac{ t\rho(S_{f(\gamma)},
\pi({f(\gamma)}))v(\pi({f(\gamma)}))}{{\rm deg}(f|_\gamma)}\\
&\leq&\sum_{\delta\in \Gamma}\sum_{\alpha\sim\pi(\gamma),
\alpha\subset {S}_\gamma\setminus{\pi(\gamma)}} \frac{t
v(\delta)}{{\rm deg}(f:\alpha\rightarrow\delta)}+\frac{
t\rho(S_{f(\gamma)},
\pi({f(\gamma)}))v(\pi({f(\gamma)}))}{{\rm deg}(f|_\gamma)}\ \ (By \ 2)\\
&<& t\rho(S_\gamma, \pi(\gamma))v(\pi(\gamma)) =\sigma_t(\gamma). \
\ (By\ the \ choice \ of \ the \ number \ \rho)\eess

\end{proof}

\subsection{Holomorphic models} \label{a4}

We first decompose $\mathcal{S}^\star$ into
$\mathcal{S}^\star_0\sqcup \mathcal{S}^\star_1\sqcup\cdots$, where
\bess &\mathcal{S}^\star_0=\{f_*^j(S^\star_\nu); 0\leq
j<p_\nu, 1\leq \nu\leq n\},&\\
&\mathcal{S}^\star_k=\{S^\star\in \mathcal{S}^\star; k \text{ is the
first integer such that } f_*^k(S^\star)\in\mathcal{S}^\star_0 \},
k\geq1.&  \eess
%For each $\mathcal{S}^\star$-piece $S^\star$, let
%$\ell(S^\star)$ be the smallest integer $k\geq0$ such that
%$f_*^k(S^\star)$ is a periodic $\mathcal{S}^\star$-piece. We
%define
%$\mathcal{S}^\star_k=\{S^\star\in\mathcal{S}^\star;
%\ell(S^\star)=k \}$ for $k\geq0$.
It's obvious that $\mathcal{S}^\star_0$ consists of all
$f_*$-periodic $\mathcal{S}^\star$-pieces.

\begin{lem}[Pre-holomorphic models]\label{b4d}   Suppose 
 $({h}_\nu^\star,{P}^\star_\nu)$ is q.c-equivalent to a rational map $(R_\nu,Q_\nu)$
via a pair of quasi-conformal maps $(\Phi_\nu,\Psi_\nu)$ for
$1\leq\nu\leq n$. Then for each $\mathcal{S}^\star$-piece $S^\star$,
there exist a pair of quasi-conformal maps
$(\Phi_{S^\star},\Psi_{S^\star}):\mathbb{\overline{C}}(S^\star)\rightarrow\mathbb{\overline{C}}$
and  a rational map $R_{S^\star}$ such that $\Phi_{S^\star}$ is
isotopic to $\Psi_{S^\star}$ rel $P(S^\star)$ and the following
diagram commutes:
$$
\xymatrix{& \mathbb{\overline{C}}(S^\star)\ar[r]^{H_{S^\star}}
\ar[d]_{\Psi_{S^\star}}
& \mathbb{\overline{C}}(f_*(S^\star))\ar[d]^{\Phi_{f_*(S^\star)}}\\
&\mathbb{\overline{C}}\ar[r]_{R_{S^\star}} & \mathbb{\overline{C}}
 }
$$
\end{lem}

\begin{proof} Using the  same argument as the proof of the sufficiency of  Theorem \ref{b4a} (see Section \ref{3-3-2}, step 1),  one
can show that for any $1\leq\nu\leq n$ and any $0\leq k< p_\nu$,
 there exist a quasi-conformal map  $\Phi_{f_*^k(S^\star_\nu)}$ and a rational map $R_{f_*^k(S^\star_\nu)}$ such that the
following diagram commutes
$$
\xymatrix{&
\mathbb{\overline{C}}(S^\star_\nu)\ar[r]^{H_{S^\star_\nu}}
\ar[d]_{\Psi_{S^\star_\nu}=\Psi_\nu}
&\mathbb{\overline{C}}(f_*(S^\star_\nu))\ar[r]^{H_{f_*(S^\star_\nu)}}\ar[d]_{\Phi_{f_*(S^\star_\nu)}}
 & \cdots \ar[r]^{}
&\mathbb{\overline{C}}(f_*^{p_\nu-1}(S^\star_\nu))\ar[r]^{H_{f_*^{p_\nu-1}(S^\star_\nu)}}\ar[d]_{\Phi_{f_*^{p_\nu-1}(S^\star_\nu)}}
 & \mathbb{\overline{C}}(S^\star_\nu)\ar[d]^{\Phi_{S^\star_\nu}=\Phi_\nu}\\
& \mathbb{\overline{C}}\ar[r]_{R_{S^\star_\nu}}
 &\mathbb{\overline{C}}\ar[r]_{R_{f_*(S^\star_\nu)}}
& \cdots \ar[r]_{}&
\mathbb{\overline{C}}\ar[r]_{R_{f_*^{p_\nu-1}(S^\star_\nu)}}   &
\mathbb{\overline{C}}
 }
$$
We set $\Psi_{f_*^k(S^\star_\nu)}=\Phi_{f_*^k(S^\star_\nu)}$ for
$0<k< p_\nu$.

For each $S^\star\in\mathcal{S}^\star_1$, notice that
$f_*(S^\star)\in\mathcal{S}^\star_0$, we pull back the standard
complex structure of $\mathbb{\overline{C}}$ to
$\mathbb{\overline{C}}(S^\star)$ via $\Phi_{f_*(S^\star)}\circ
H_{S^\star}$ and integrate it to get a quasi-conformal map
$\Phi_{S^\star}:\mathbb{\overline{C}}(S^\star)\rightarrow\mathbb{\overline{C}}$.
Then $R_{S^\star}:=\Phi_{f_*(S^\star)}\circ
H_{S^\star}\circ\Phi_{S^\star}^{-1} $ is a rational map. We set
$\Psi_{S^\star}=\Phi_{S^\star}$.

%By the same way, we can then  get a pair of quasi-conformal maps
%$(\Phi_{S^\star},\Psi_{S^\star})$ and a rational map
%$R_{S^\star}$ for each $S^\star\in \mathcal{S}^\star_2$.

By the inductive procedure, for each $\mathcal{S}^\star_k$-piece
($k=2,3,\cdots$), we can  get a pair of quasi-conformal maps
$(\Phi_{S^\star},\Psi_{S^\star})$ and a rational map $R_{S^\star}$,
as required.
%In this way, suppose that for each $\mathcal{S}^\star_n$-piece,
%we have already get a pair of  $\mathcal{S}^\star$-piece
%$S^\star$ in $\mathcal{S}^\star_1, \mathcal{S}^\star_2,
%\cdots$, we can get a pair of quasi-conformal maps
%$(\Phi_{S^\star},\Psi_{S^\star})$ and a rational map
%$R_{S^\star}$, as required. The proof is completed by the
%inductive procedure.
\end{proof}

\begin{lem}[Holomorphic models for periodic pieces]\label{b4e}  Fix a periodic piece
$S^\star\in\mathcal{S}^\star_0$. Let $p$ be the period of $S^\star$.
Then for any large parameter $t>0$, there exist a pair of
quasi-conformal maps
$(\Phi^t_{S^\star},\Psi^t_{S^\star}):\mathbb{\overline{C}}(S^\star)\rightarrow\mathbb{\overline{C}}$
such that

1. $\Psi^t_{S^\star}$ is isotopic to $\Phi^t_{S^\star}$  rel
$P(S^\star)$.

2. $\Phi^t_{f_*(S^\star)}\circ f\circ
(\Psi^t_{S^\star})^{-1}|_{\Psi^t_{S^\star}(E_{S^\star})}=R_{S^\star}|_{\Psi^t_{S^\star}(E_{S^\star})}$,
where $R_{S^\star}$ is defined in Lemma \ref{b4d}.

3. The return map $f_i:=R_{f_*^{i-1}(S^\star)}\circ\cdots\circ
R_{S^\star}\circ R_{f_*^{p-1}(S^\star)}\circ\cdots\circ
R_{f_*^{i}(S^\star)}$ is either a Siegel map or a Thurston map.

4. For each $i\geq0$ and each curve $\gamma\in
\partial(f^i_*(S^\star))$, let $\beta_\gamma$ be the  boundary  curve
of $ E_{f^i_*(S^\star)}$ such that either $\gamma=\beta_\gamma$, or
$\gamma$ and  $\beta_\gamma$ bound  an annulus in
  $S^\star-P$. Then both
$\Phi^t_{f_*^{i}(S^\star)}(\gamma)$ and
$\Psi^t_{f_*^{i}(S^\star)}(\beta_\gamma)$ are equipotentials in the
same marked disk of $f_i$, with potentials

$$\varpi(\Phi^t_{f_*^{i}(S^\star)}(\gamma))=\sigma_t(\gamma),\
\varpi(\Psi^t_{f_*^{i}(S^\star)}(\beta_\gamma))=\frac{\sigma_t(f(\beta_\gamma))}{{\deg(f|_{\beta_\gamma})}}.$$
\end{lem}

\begin{proof} For each $\nu\in [1,n]$ and each $i\geq0$, the critical
values of $H_{f_*^i(S^\star_\nu)}$ are contained in
$P(f_*^{i+1}(S^\star_\nu))$ and
$H_{f_*^i(S^\star_\nu)}(P(f_*^{i}(S^\star_\nu)))\subset
P(f_*^{i+1}(S^\star_\nu))$. Let $(\Phi_{f_*^{i}(S^\star_\nu)},
\Psi_{f_*^{i}(S^\star_\nu)}):\mathbb{\overline{C}}(f_*^{i}(S^\star_\nu))\rightarrow\mathbb{\overline{C}}$
be the quasi-conformal maps constructed in Lemma \ref{b4d}. Since
$\Phi_{S^\star_\nu}$ is isotopic to $\Psi_{S^\star_\nu}$ rel
${P}^\star_\nu=P(S^\star_\nu)$, there is a % unique
quasi-conformal map $\phi_{f_*^{p_\nu-1}(S^\star_\nu)}:
\mathbb{\overline{C}}(f_*^{p_\nu-1}(S^\star_\nu))\rightarrow\mathbb{\overline{C}}$
% such that $\phi_{f_*^{p_\nu-1}(S^\star_\nu)}$ is
isotopic to $\Phi_{f_*^{p_\nu-1}(S^\star_\nu)}$ rel
$P(f_*^{p_\nu-1}(S^\star_\nu))$ and $\Psi_{S^\star_\nu}\circ
H_{f_*^{p_\nu-1}(S^\star_\nu)}=R_{f_*^{p_\nu-1}(S^\star_\nu)}\circ
\phi_{f_*^{p_\nu-1}(S^\star_\nu)}$. Inductively, there is a sequence
of quasi-conformal maps $\phi_{f_*^{i}(S^\star_\nu)}:
\mathbb{\overline{C}}(f_*^{i}(S^\star_\nu))\rightarrow\mathbb{\overline{C}}$
for $i=p_\nu-2,\cdots,1$, such that $\phi_{f_*^{i}(S^\star_\nu)}$ is
isotopic to $\Phi_{f_*^{i}(S^\star_\nu)}$ rel
$P(f_*^i(S^\star_\nu))$ and the following diagram commutes:

$$
\xymatrix{&
\mathbb{\overline{C}}(f_*(S^\star_\nu))\ar[r]^{H_{f_*(S^\star_\nu)}}
\ar[d]_{\phi_{f_*(S^\star_\nu)}}%^{\Phi_{f_*(S^\star_\nu)}}
&\mathbb{\overline{C}}(f_*^2(S^\star_\nu))\ar[r]^{H_{f_*^2(S^\star_\nu)}}\ar[d]_{\phi_{f_*^2(S^\star_\nu)}}%^{\Phi_{f_*^2(S^\star_\nu)}}
 & \cdots \ar[r]^{}
&\mathbb{\overline{C}}(f_*^{p_\nu-1}(S^\star_\nu))\ar[r]^{H_{f_*^{p_\nu-1}(S^\star_\nu)}}
\ar[d]_{\phi_{f_*^{p_\nu-1}(S^\star_\nu)}}%^{\Phi_{f_*^{p_\nu-1}(S^\star_\nu)}}
 & \mathbb{\overline{C}}(S^\star_\nu)\ar[d]_{\Psi_{S^\star_\nu}}\\%^{\Phi_{S^\star_\nu}}
& \mathbb{\overline{C}}\ar[r]_{R_{f_*(S^\star_\nu)}}
 &\mathbb{\overline{C}}\ar[r]_{R_{f_*^2(S^\star_\nu)}}
& \cdots \ar[r]_{}&
\mathbb{\overline{C}}\ar[r]_{R_{f_*^{p_\nu-1}(S^\star_\nu)}}   &
\mathbb{\overline{C}}
 }
$$

%\bess
%\xymatrix{&\mathbb{\overline{C}}(f_*^i(S^\star_\nu))\ar[r]^{H_{f_*^i(S^\star_\nu)}}
%\ar[d]_{\phi_{S^\star_\nu}}
%&\cdots&\mathbb{\overline{C}}(f_*^{p_\nu-1}(S^\star_\nu))\ar[r]^{H_{f_*^{p_\nu-1}(S^\star_\nu)}}\ar[d]_{\phi_{f_*^{p_\nu-1}(S^\star_\nu)}}
% &\mathbb{\overline{C}}(S^\star_\nu)\ar[r]^{H_{S^\star_\nu}}\ar[d]_{\Psi_{S^\star_\nu}}
%&\mathbb{\overline{C}}(f_*(S^\star_\nu))\ar[r]^{H_{f_*(S^\star_\nu)}}\ar[d]_{\Phi_{f_*(S^\star_\nu)}}
% &\cdots& \mathbb{\overline{C}}(f_*^i(S^\star_\nu))\ar[d]^{\Phi_{f_*^i(S^\star_\nu)}}\\
%& \mathbb{\overline{C}}\ar[r]_{R_{f_*^i(S^\star_\nu)}}  &\cdots
%& \mathbb{\overline{C}}\ar[r]_{R_{f_*^{p_\nu-1}(S^\star_\nu)}}
%&\mathbb{\overline{C}}\ar[r]_{R_{S^\star_\nu}}
% &\mathbb{\overline{C}}\ar[r]_{R_{f_*(S^\star_\nu)}}
%& \cdots  & \mathbb{\overline{C}}
% }
%\eess

This diagram together with the diagram in Lemma \ref{b4d} implies
that for any $1\leq i<p_\nu$, the map
$H_{f_*^{i-1}(S^\star_\nu)}\circ\cdots\circ H_{S^\star_\nu}\circ
H_{f_*^{p_\nu-1}(S^\star_\nu)}\circ\cdots\circ
H_{f_*^{i}(S^\star_\nu)}$ is q.c-equivalent to
$f_i=R_{f_*^{i-1}(S^\star_\nu)}\circ\cdots\circ R_{S^\star_\nu}\circ
R_{f_*^{p_\nu-1}(S^\star_\nu)}\circ\cdots\circ
R_{f_*^{i}(S^\star_\nu)}$ via
$(\Phi_{f_*^{i}(S^\star_\nu)},\phi_{f_*^{i}(S^\star_\nu)})$. Notice
that
$f_i(\phi_{f_*^{i}(S^\star_\nu)}(P(f_*^{i}(S^\star_\nu))))\subset
\phi_{f_*^{i}(S^\star_\nu)}(P(f_*^{i}(S^\star_\nu)))$,  $f_i$ is
either a Siegel map or a Thurston map.

The relation $f_{i+1}\circ
R_{f_*^i(S^\star_\nu)}=R_{f_*^i(S^\star_\nu)}\circ f_i$ with
$f_{p_\nu}=R_\nu$ (here, $R_\nu$ is the rational map defined in
Lemma \ref{b4d}) means that $R_{f_*^i(S^\star_\nu)}$ is a
semi-conjugacy between $f_{i+1}$ and $f_i$, so their Julia sets
satisfy $J(f_i)=R_{f_*^i(S^\star_\nu)}^{-1}(J(f_{i+1}))$. One can
check that $R_{f_*^i(S^\star_\nu)}$ maps the marked disks of $f_i$
onto the marked disks of $f_{i+1}$, and maps the equipotentials of
$f_i$ to the  equipotentials of $f_{i+1}$.

In the following, we will construct a pair of  quasi-conformal maps
$(\Phi^t_{S^\star},\Psi^t_{S^\star}):\mathbb{\overline{C}}(S^\star)\rightarrow\mathbb{\overline{C}}$
that satisfy the required properties.

\vskip0.3cm
\noindent\textbf{Step 1: Construction of $\Phi^t_{S^\star_\nu}$ and
${\Phi}^t_{f_*^{p_\nu-1}(S^\star_\nu)}$. }
 We first modify $\Phi_{S^\star_\nu}$ to  a new quasi-conformal map
$\Phi^t_{S^\star_\nu}:
\mathbb{\overline{C}}(S^\star_\nu)\rightarrow\mathbb{\overline{C}}$
such that $\Phi^t_{S^\star_\nu}$ is isotopic to $\Phi_{S^\star_\nu}$
rel $P(S^\star_\nu)$, and for each curve $\gamma\in
\partial(S^\star_\nu)$, the curve
$\Phi^t_{S^\star_\nu}(\gamma)$ is the equipotential in a  marked
disk of $f_{p_\nu}=R_\nu$ with potential
$\varpi(\Phi^t_{S^\star_\nu}(\gamma))=\sigma_t(\gamma)$.
Then, we lift $\Phi^t_{S^\star_\nu}$ via
$R_{f_*^{p_\nu-1}(S^\star_\nu)}$ and
$H_{f_*^{p_\nu-1}(S^\star_\nu)}$ and get a quasi-conformal  map
$\widehat{\Phi}^t_{f_*^{p_\nu-1}(S^\star_\nu)}$ isotopic to
$\Phi_{f_*^{p_\nu-1}(S^\star_\nu)}$ rel
$P(f_*^{p_\nu-1}(S^\star_\nu))$. See the following commutative
diagram:

$$
\xymatrix{&
\mathbb{\overline{C}}(f_*^{p_\nu-1}(S^\star_\nu))\ar[r]^{H_{f_*^{p_\nu-1}(S^\star_\nu)}}
\ar[d]_{\widehat{\Phi}^t_{f_*^{p_\nu-1}(S^\star_\nu)}(\sim
\Phi_{f_*^{p_\nu-1}(S^\star_\nu)})} &
\mathbb{\overline{C}}(S^\star_\nu)\ar[d]^{\Phi^t_{S^\star_\nu}(\sim
\Phi_{S^\star_\nu})}\\
&\mathbb{\overline{C}}\ar[r]_{R_{f_*^{p_\nu-1}(S^\star_\nu)}} &
\mathbb{\overline{C}}
 }
$$

Now, we modify $H_{f_*^{p_\nu-1}(S^\star_\nu)}$ to another marked
disk extension of $f|_{E_{f_*^{p_\nu-1}(S^\star_\nu)}}$, say
$\widehat{H}_{f_*^{p_\nu-1}(S^\star_\nu)}$, %(We assume that the
%marked set $P(f_*^{p_\nu-1}(S^\star_\nu))$ is fixed)
 such that
for each curve $\gamma\in
\partial_1(f_*^{p_\nu-1}(S^\star_\nu))$, the curve
$\Phi^t_{S^\star_\nu}(\widehat{H}_{f_*^{p_\nu-1}(S^\star_\nu)}(\gamma))$
is an equipotential in some marked disk  of $f_{p_\nu}=R_\nu$. Since
$\gamma\in
\partial_1(f_*^{p_\nu-1}(S^\star_\nu))$, the potential of
 $\Phi^t_{S^\star_\nu}(\widehat{H}_{f_*^{p_\nu-1}(S^\star_\nu)}(\gamma))$  should be larger than
$\varpi(\Phi^t_{S^\star_\nu}(f(\beta_\gamma)))=\sigma_t(f(\beta_\gamma))$.
It follows from Lemma \ref{b4c} that
$\deg(f|_{\beta_\gamma})\sigma_t(\gamma)>\sigma_t(f(\beta_\gamma))$
when $t$ is large. So it is reasonable to designate
$\varpi(\Phi^t_{S^\star_\nu}(\widehat{H}_{f_*^{p_\nu-1}(S^\star_\nu)}(\gamma)))$
to be $\deg(f|_{\beta_\gamma})\cdot\sigma_t(\gamma)$.

Since both $H_{f_*^{p_\nu-1}(S^\star_\nu)}$ and
$\widehat{H}_{f_*^{p_\nu-1}(S^\star_\nu)}$ are marked disk
extensions of $f|_{E_{f_*^{p_\nu-1}(S^\star_\nu)}}$, there is  a
quasi-conformal map
$\xi_{p_\nu-1}:\mathbb{\overline{C}}(f_*^{p_\nu-1}(S^\star_\nu))
\rightarrow\mathbb{\overline{C}}(f_*^{p_\nu-1}(S^\star_\nu))$
isotopic to the identity map rel $E_{f_*^{p_\nu-1}(S^\star_\nu)}\cup
P(f_*^{p_\nu-1}(S^\star_\nu))$ such that
$\widehat{H}_{f_*^{p_\nu-1}(S^\star_\nu)}=H_{f_*^{p_\nu-1}(S^\star_\nu)}\circ\xi_{p_\nu-1}$.

% with
%potential
%$$\varpi(\Phi^t_{S^\star_\nu}(\widehat{H}_{f_*^{p_\nu-1}(S^\star_\nu)}(\gamma)))=\deg(f|_{\beta_\gamma})\sigma_t(\gamma).$$
%$\widehat{H}_{f_*^{p_\nu-1}(S^\star_\nu)}=H_{f_*^{p_\nu-1}(S^\star_\nu)}\circ\xi_{p_\nu-1}$,
%where
%$\xi_{p_\nu-1}:\mathbb{\overline{C}}(f_*^{p_\nu-1}(S^\star_\nu))
%\rightarrow\mathbb{\overline{C}}(f_*^{p_\nu-1}(S^\star_\nu))$ is
%isotopic to the identity map rel
%$E_{f_*^{p_\nu-1}(S^\star_\nu)}\cup
%P(f_*^{p_\nu-1}(S^\star_\nu))$,

%One may easy check that
%$\varpi(\Phi^t_{S^\star_\nu}(\widehat{H}_{f_*^{p_\nu-1}(S^\star_\nu)}(\gamma)))>\sigma_t(f(\beta_\gamma))
%=\varpi(\Phi^t_{S^\star_\nu}(f(\beta_\gamma)))$.
 We set
${\Phi}^t_{f_*^{p_\nu-1}(S^\star_\nu)}=\widehat{\Phi}^t_{f_*^{p_\nu-1}(S^\star_\nu)}\circ
\xi_{p_\nu-1}$. It's obvious that
$\Phi^t_{S^\star_\nu}\circ\widehat{H}_{f_*^{p_\nu-1}(S^\star_\nu)}
=R_{f_*^{p_\nu-1}(S^\star_\nu)}\circ{\Phi}^t_{f_*^{p_\nu-1}(S^\star_\nu)}$.

\vskip0.3cm
\noindent\textbf{Step 2: Construction of $\Phi^t_{f^i_*(S^\star_\nu)}$ for
$i=p_\nu-2,\cdots,1$ and ${\Psi}^t_{S^\star_\nu}$. } By the same
argument as in Step 1, we can lift
${\Phi}^t_{f_*^{p_\nu-1}(S^\star_\nu)}$ via
$R_{f_*^{p_\nu-2}(S^\star_\nu)}$ and
$H_{f_*^{p_\nu-2}(S^\star_\nu)}$ and get a map
$\widehat{\Phi}^t_{f_*^{p_\nu-2}(S^\star_\nu)}$ isotopic to
$\Phi_{f_*^{p_\nu-2}(S^\star_\nu)}$ rel
$P(f_*^{p_\nu-2}(S^\star_\nu))$. Then we modify
$H_{f_*^{p_\nu-2}(S^\star_\nu)}$ to another marked disk extension of
$f|_{E_{f_*^{p_\nu-2}(S^\star_\nu)}}$, say
$\widehat{H}_{f_*^{p_\nu-2}(S^\star_\nu)}=H_{f_*^{p_\nu-2}(S^\star_\nu)}\circ\xi_{p_\nu-2}$,
where
$\xi_{p_\nu-2}:\mathbb{\overline{C}}(f_*^{p_\nu-2}(S^\star_\nu))
\rightarrow\mathbb{\overline{C}}(f_*^{p_\nu-2}(S^\star_\nu))$ is a
quasi-conformal map isotopic to the identity map rel
$E_{f_*^{p_\nu-2}(S^\star_\nu)}\cup P(f_*^{p_\nu-2}(S^\star_\nu))$,
such that for each $\gamma\in
\partial_1(f_*^{p_\nu-2}(S^\star_\nu))$, the curve
$\Phi^t_{f_*^{p_\nu-1}(S^\star_\nu)}(\widehat{H}_{f_*^{p_\nu-2}(S^\star_\nu)}(\gamma))$
is an equipotential of $f_{p_\nu-1}$ with potential equal to
$\deg(f|_{\beta_\gamma})\sigma_t(\gamma)$. We set
${\Phi}^t_{f_*^{p_\nu-2}(S^\star_\nu)}=\widehat{\Phi}^t_{f_*^{p_\nu-2}(S^\star_\nu)}\circ
\xi_{p_\nu-2}$.

Inductively, we can get a sequence of new marked disk extensions
$\widehat{H}_{f_*^{i}(S^\star_\nu)}, i= p_\nu-1,\cdots, 0,$ and a
sequence of quasi-conformal maps $\Phi^t_{f_*^{i}(S^\star_\nu)},
i=p_\nu-1,\cdots, 1,\ \Psi^t_{S^\star_\nu}$ such that the following
diagram commutes

$$
\xymatrix{&
\mathbb{\overline{C}}(S^\star_\nu)\ar[r]^{\widehat{H}_{S^\star_\nu}}
\ar[d]_{\Psi^t_{S^\star_\nu}}
&\mathbb{\overline{C}}(f_*(S^\star_\nu))\ar[r]^{\widehat{H}_{f_*(S^\star_\nu)}}\ar[d]_{\Phi^t_{f_*(S^\star_\nu)}}
 & \cdots \ar[r]^{}
&\mathbb{\overline{C}}(f_*^{p_\nu-1}(S^\star_\nu))\ar[r]^{\widehat{H}_{f_*^{p_\nu-1}(S^\star_\nu)}}\ar[d]_{\Phi^t_{f_*^{p_\nu-1}(S^\star_\nu)}}
 & \mathbb{\overline{C}}(S^\star_\nu)\ar[d]^{\Phi^t_{S^\star_\nu}}\\
& \mathbb{\overline{C}}\ar[r]_{R_{S^\star_\nu}}
 &\mathbb{\overline{C}}\ar[r]_{R_{f_*(S^\star_\nu)}}
& \cdots \ar[r]_{}&
\mathbb{\overline{C}}\ar[r]_{R_{f_*^{p_\nu-1}(S^\star_\nu)}}   &
\mathbb{\overline{C}}
 }
$$

Moreover, for each $i\in[0,p_\nu-1]$ and each curve $\gamma\in
\partial_1(f_*^{i}(S^\star_\nu))$, we require
$$\varpi(\Phi^t_{f_*^{i+1}(S^\star_\nu)}(\widehat{H}_{f_*^{i}(S^\star_\nu)}(\gamma)))=\deg(f|_{\beta_\gamma})\sigma_t(\gamma).$$

Finally, we set
$\Psi^t_{f_*^{i}(S^\star_\nu)}=\Phi^t_{f_*^{i}(S^\star_\nu)}$ for
$1\leq i\leq p_\nu-1$.

\vskip0.3cm
\noindent\textbf{Step 3: Prescribed potentials. } In this step, we will show
that for each $0\leq i\leq p_\nu-1$ and each curve $\gamma\in
\partial(f_*^{i}(S^\star_\nu))$,

\begin{equation}\label{lb3}
\varpi(\Phi^t_{f_*^{i}(S^\star_\nu)}(\gamma))=\sigma_t(\gamma),\
\varpi(\Psi^t_{f_*^{i}(S^\star_\nu)}(\beta_\gamma))=\frac{\sigma_t(f(\beta_\gamma))}{{\deg(f|_{\beta_\gamma})}}.
\end{equation}

Notice that for each curve $\gamma\in
\partial(S^\star_\nu)\cup\cup_{0<i<p_\nu}\partial_0(f_*^{i}(S^\star_\nu))$,
the first equation of  (\ref{lb3}) holds by the   evaluation of
$\varpi$.

If $\gamma\in
\partial_1(f_*^i(S^\star_\nu))$ for some $0< i<p_\nu$, then
by Step 2,
$\Phi^t_{f_*^{i+1}(S^\star_\nu)}(\widehat{H}_{f_*^{i}(S^\star_\nu)}(\gamma))$
is an equipotential in a marked disk $(\Delta_{i+1},a)$ of
$f_{i+1}$. Since
$\Phi^t_{f_*^{i+1}(S^\star_\nu)}\circ\widehat{H}_{f_*^{i}(S^\star_\nu)}(\gamma)
=R_{f_*^{i}(S^\star_\nu)}\circ{\Phi}^t_{f_*^{i}(S^\star_\nu)}(\gamma)$,
we conclude that $\Phi^t_{f_*^{i}(S^\star_\nu)}(\gamma)$ is also an
equipotential of some marked disk of $f_i$, denoted  by
$(\Delta_i,b)$. Then $R_{f_*^{i}(S^\star_\nu)}:\Delta_i-\{b\}
\rightarrow \Delta_{i+1}-\{a\}$ is a covering map of degree
${\deg}(f|_{\beta_\gamma})$.  The  potential of
${\Phi}^t_{f_*^{i}(S^\star_\nu)}(\gamma)$ satisfies (here, we use
$A(\alpha,\beta)$ to denote the annulus bounded by $\alpha$ and
$\beta$) \bess
\varpi({\Phi}^t_{f_*^{i}(S^\star_\nu)}(\gamma))&=&{\rm
mod}\Big(A(\partial\Delta_i,
{\Phi}^t_{f_*^{i}(S^\star_\nu)}(\gamma))\Big) \\
&=&{\rm mod}\Big(A(\partial\Delta_{i+1},
{\Phi}^t_{f_*^{i+1}(S^\star_\nu)}(\widehat{H}_{f_*^{i}(S^\star_\nu)}(\gamma)))\Big)/{\deg}(f|_{\beta_\gamma})\\
&=&\varpi(\Phi^t_{f_*^{i+1}(S^\star_\nu)}(\widehat{H}_{f_*^{i}(S^\star_\nu)}(\gamma)))/{\deg}(f|_{\beta_\gamma})\\
&=&\sigma_t(\gamma). \eess

Now we consider  $\gamma\in
\partial_2(f_*^i(S^\star_\nu))$ for some $0<i<p_\nu$. In this
case, by the same argument as above, we can see that
$$\varpi(\Phi^t_{f_*^{i}(S^\star_\nu)}(\gamma))
=\frac{\varpi(\Phi^t_{f_*^{i+1}(S^\star_\nu)}(f(\gamma)))}{{\deg}(f|_{\gamma})}.$$
By the definition of $\sigma_t$,  for $\gamma\in
\partial_2(f_*^i(S^\star_\nu))$,  we have
 $$\sigma_t(\gamma)
=\frac{\sigma_t(f(\gamma))}{{\deg}(f|_{\gamma})}.$$ Based on this
observation, we conclude by induction  that
$\varpi({\Phi}^t_{f_*^{i}(S^\star_\nu)}(\gamma))=\sigma_t(\gamma)$.

Finally, we show that the second equation of (\ref{lb3}) holds.
Since for each $i\in[0,p_\nu-1]$ and each curve $\gamma\in
\partial(f_*^{i}(S^\star_\nu))$, the curve $\Phi^t_{f_*^{i+1}(S^\star_\nu)}(f(\beta_\gamma))$
is an equipotential, % that for each periodic piece $S^\star$,
it follows from the relation
$$\Phi^t_{f^{i+1}_*(S^\star_\nu)}\circ f\circ
(\Psi^t_{f^{i}_*(S^\star_\nu)})^{-1}|_{\Psi^t_{f^{i}_*(S^\star_\nu)}(E_{f^{i}_*(S^\star_\nu)})}
=R_{f^{i}_*(S^\star_\nu)}|_{\Psi^t_{f^{i}_*(S^\star_\nu)}(E_{f^{i}_*(S^\star_\nu)})}$$
%Thus for each curve $\gamma\in
%\partial(S^\star)$,
%$\Phi^t_{f_*(S^\star)}(f(\beta_\gamma))$ is an equipotential
%curve implies that
that $\Psi^t_{f_*^{i}(S^\star_\nu)}(\beta_\gamma)$ is also an
equipotential. The same argument as above  yields
$$\varpi(\Psi^t_{f_*^{i}(S^\star_\nu)}(\beta_\gamma))=\frac{\varpi(\Phi^t_{f_*^{i+1}(S^\star_\nu)}(f(\beta_\gamma)))}{{\deg}(f|_{\beta_\gamma})}
=\frac{\sigma_t(f(\beta_\gamma))}{{\deg}(f|_{\beta_\gamma})}.$$ The
proof is completed.
\end{proof}

\

Now, we deal with the strictly pre-periodic
$\mathcal{S}^\star$-pieces. Let $S^\star\in \mathcal{S}^\star_k$ for
some $k\geq1$. Then $f_*^{k}(S^\star)$ is a $f_*$-periodic
$\mathcal{S}^\star$-piece. Notice that for $0\leq i<k$,
$H_{f_*^{i}(S^\star)}(P(f_*^{i}(S^\star)))\subset
P(f_*^{i+1}(S^\star))$ and each critical value of
$H_{f_*^{i}(S^\star)}$ is contained in $P(f_*^{i+1}(S^\star))$, we
have that
$R_{f_*^i(S^\star)}\circ\Phi_{f_*^i(S^\star)}(P(f_*^{i}(S^\star)))\subset
\Phi_{f_*^{i+1}(S^\star)} (P(f_*^{i+1}(S^\star)))$ and every
critical value of $R_{f_*^{k-1}(S^\star)}\circ\cdots\circ
R_{S^\star}$ is contained in
$\Phi_{f_*^{k}(S^\star)}(P(f_*^{k}(S^\star)))
=\Phi^t_{f_*^{k}(S^\star)}(P(f_*^{k}(S^\star)))$, here
$R_{f_*^{i}(S^\star)}$ and $\Phi_{f_*^{i}(S^\star)}$ are defined in
Lemma \ref{b4d}. For any marked point $a\in
P(S^\star)\cap(\mathbb{\overline{C}}(S^\star)-S^\star)$, the point
$R_{f_*^{k-1}(S^\star)}\circ\cdots\circ
R_{S^\star}(\Phi_{S^\star}(a))$ is the center of some marked disk
$(\Delta,q)$ of some $f_j$, where $f_j$ is a return map defined in
Lemma \ref{b4e}. The component $\Delta_{\Phi_{S^\star}(a)}$ of
$(R_{f_*^{k-1}(S^\star)}\circ\cdots\circ R_{S^\star})^{-1}(\Delta)$
that contains $\Phi_{S^\star}(a)$ is also a disk. We call
$(\Delta_{\Phi_{S^\star}(a)}, \Phi_{S^\star}(a))$  a marked disk of
$R_{f_*^{k-1}(S^\star)}\circ\cdots\circ R_{S^\star}$.

With the same argument as that of  Lemma \ref{b4e}, we can show that

\begin{lem}\label{b4f}  For any $k\geq1$, any $S^\star\in\mathcal{S}^\star_k$ and any large parameter $t>0$, there exist a pair of
quasi-conformal maps
$\Phi^t_{S^\star}=\Psi^t_{S^\star}:\mathbb{\overline{C}}(S^\star)\rightarrow\mathbb{\overline{C}}$
such that

1. $\Phi^t_{f_*(S^\star)}\circ f\circ
(\Psi^t_{S^\star})^{-1}|_{\Psi^t_{S^\star}(E_{S^\star})}=R_{S^\star}|_{\Psi^t_{S^\star}(E_{S^\star})}$,
where $R_{S^\star}$ is defined in Lemma \ref{b4d}.

2. For each curve $\gamma\in \partial(S^\star)$, let $\beta_\gamma$
be the unique curve in $\partial(E_{S^\star})$ homotopic to $\gamma$
in $\mathbb{\overline{C}}-P$. Then both $\Phi^t_{S^\star}(\gamma)$
and $\Phi^t_{S^\star}(\beta_\gamma)$ are equipotentials in the same
marked disk   of $R_{f_*^{k-1}(S^\star)}\circ\cdots\circ
R_{S^\star}$, with potentials

$$\varpi(\Phi^t_{S^\star}(\gamma))=\sigma_t(\gamma),\
\varpi(\Phi^t_{S^\star}(\beta_\gamma))=\frac{\sigma_t(f(\beta_\gamma))}{{\deg(f|_{\beta_\gamma})}}.$$
\end{lem}

\

We decompose $\mathcal{E}^\star$ into $\mathcal{E}^\star_{ess}\sqcup
\mathcal{E}^\star_{A}\sqcup \mathcal{E}^\star_{D}$, where

$\bullet$
  $\mathcal{E}^\star_{ess}=\{E_{S^\star};
S^\star\in \mathcal{S}^\star\}$, it consists of all
$\mathcal{E}^\star$-pieces parallel to some
$\mathcal{S}^\star$-piece;

$\bullet$
 $\mathcal{E}^\star_{A}$
is the collection of all $\mathcal{E}^\star$-pieces $E^\star$
contained essentially in an annular component of
$S^\star-E_{S^\star}$ for some $\mathcal{S}^\star$-piece $S^\star$
(here, `essential' means at least one boundary curve of $E^\star$ is
non-peripheral in $\mathbb{\overline{C}}-P$);

$\bullet$
  $\mathcal{E}^\star_{D}=\mathcal{E}^\star-(\mathcal{E}^\star_{ess}\sqcup
\mathcal{E}^\star_{A})$. One may verify that each
$\mathcal{E}^\star_{D}$-piece is contained in a disk component of
$S^\star-E_{S^\star}\cup (\cup\mathcal{E}^\star_{ess})$ for some
$\mathcal{S}^\star$-piece $S^\star$.

%It's obvious that these three sets are mutually disjoint.
%One can
%easily check that   each $\mathcal{E}^\star_{D}$-piece
%$E^\star$ is contained in a component of
%$$\bigcup_{S^\star\in
%\mathcal{S}^\star}S^\star-\bigcup_{E^\star\in
%\mathcal{E}^\star_{ess}\cup
%\mathcal{E}^\star_{A}} E^\star$$ , there is a unique
%boundary curve $\gamma\in
%\partial(E^\star)$ which bounds a unique disk $D_{E^\star}$ satisfying $E^\star\subset D_{E^\star}\subset
%S^\star_{E^\star}$, here $S^\star_{E^\star}$ is the
%$\mathcal{S}^\star$-piece containing $E^\star$, and
%$D_{E^\star}$ contains at most one point in $P$.

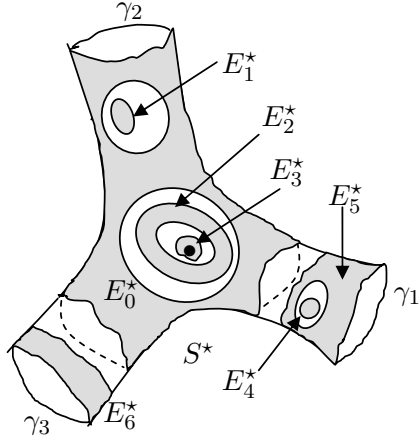
\begin{figure}[h]
\centering{
\includegraphics[height=6cm]{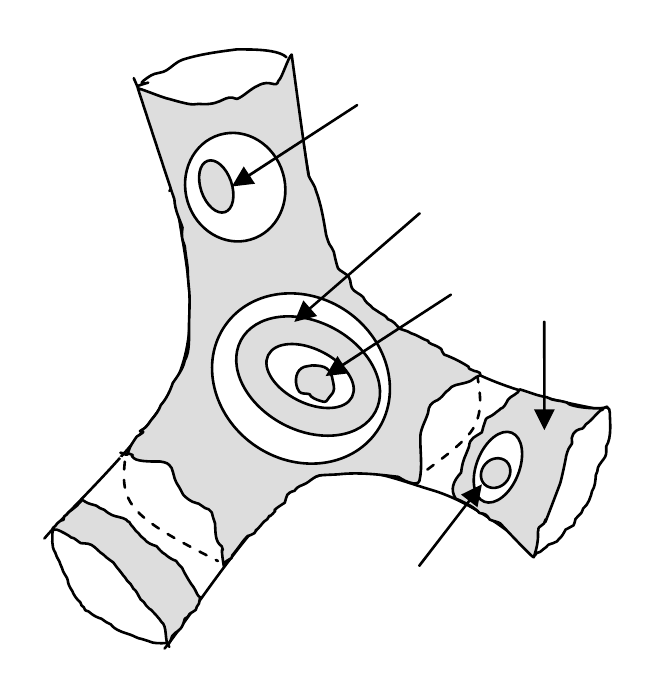}  \put(-90, 30){$S^\star$}  \put(-120, 54){$E^\star_0$} \put(-75, 20){$E^\star_4$} \put(-75,
140){$E^\star_1$} \put(-60, 120){$E^\star_2$} \put(-57,
100){$E^\star_3$}\put(-120, 8){$E^\star_6$}
  \put(-90, 70){$\bullet$} \put(-35, 90){$E^\star_5$}
\put(-10, 54){$\gamma_1$} \put(-115, 162){$\gamma_2$} \put(-150,
5){$\gamma_3$}
   %%\put(-230, 30){$S$}
%\put(-250,180){$E_1$}  \put(-263,140){$S_1$}
% \put(-130,200){$E_2$}
%\put(-119,140){$S_2$}
% \put(-273,55){$E_3$}
%\put(-263,5){$S_3$} \put(-260,45){$\bullet \ p_3$}
%\put(-260,65){$\bullet \ q_3$} \put(-100,35){$E_4$}
%\put(-130,5){$S_4$} \put(-130,35){$\bullet \ p_4$}
%\put(-130,65){$\bullet \ q_4$}
 \caption{Different types  of $\mathcal{E}^\star$-pieces: Here, $S^\star$ is
 an
 $\mathcal{S}^\star$-piece with boundary $\partial S^\star=\gamma_1\cup \gamma_2\cup\gamma_3$.
 $E^\star_0$ is an $\mathcal{E}^\star_{ess}$-piece.
  $E^\star_5$ and $E^\star_6$ are $\mathcal{E}^\star_A$-pieces. $E^\star_1,
   E^\star_2$, $E^\star_3$ and $E^\star_4$ are $\mathcal{E}^\star_D$-pieces.}}
\end{figure}%\label{fig5-1}

In the following, for every $\mathcal{E}^\star_{A}$-piece $E^\star$,
we will construct a holomorphic model for $f|_{E^\star}$. Given an
$\mathcal{E}^\star_{A}$-piece $E^\star$, first notice that $E^\star$
has no intersection with the marked set $P$. As we did before, we
also associate a Riemann sphere $\mathbb{\overline{C}}(E^\star)$ for
$E^\star$. We mark a point in each component of
$\mathbb{\overline{C}}(E^\star)-E^\star$, and let $P(E^\star)$ be
the collection of all these marked points. We can get a marked disk
extension of $f|_{E^\star}$, say
$H_{E^\star}:\mathbb{\overline{C}}(E^\star)\rightarrow
\mathbb{\overline{C}}(f(E^\star))$, such that
$H_{E^\star}|_{E^\star}=f|_{E^\star}$,
$H_{E^\star}(P(E^\star))\subset P(f(E^\star))$ and all
 critical values  (if any) of $H_{E^\star}$ are contained in
$P(f(E^\star))$. Let $\Phi^t_{E^\star}:
\mathbb{\overline{C}}(E^\star)\rightarrow \mathbb{\overline{C}} $ be
a quasi-conformal map such that
$R_{E^\star}:=\Phi^t_{f(E^\star)}\circ
H_{E^\star}\circ(\Phi^t_{E^\star})^{-1}$ is holomorphic. We give a
remark that if we change $\Phi^t_{f(E^\star)}$ to another
quasi-conformal map $\Phi^{t_1}_{f(E^\star)}$ isotopic to
$\Phi^t_{f(E^\star)}$ rel $P(f(E^\star))$, then we can modify
$\Phi^t_{E^\star}$ to a new map $\Phi^{t_1}_{E^\star}$, isotopic to
$\Phi^t_{E^\star}$ rel $P(E^\star)$, such that
$R_{E^\star}=\Phi^{t_1}_{f(E^\star)}\circ
H_{E^\star}\circ(\Phi^{t_1}_{E^\star})^{-1}$. This means that once
we get the holomorphic map $R_{E^\star}$, we can always assume that
it is independent of the parameter $t$. We set
$\Psi^t_{E^\star}=\Phi^t_{E^\star}$.

The  $\mathcal{E}^\star_{A}$-piece $E^\star$ has exactly two
boundary curves $\alpha$ and $\beta$ which are non-peripheral and
homotopic to each other in $\mathbb{\overline{C}}-P$. By the choice
of $\Phi^t_{S^\star}$ for $S^\star\in\mathcal{S}^\star$, both
$\Phi^t_{f(E^\star)}(f(\alpha))$ and $\Phi^t_{f(E^\star)}(f(\beta))$
are  equipotentials in the marked disks of some  $f_j$ (defined in
Lemma \ref{b4e}) or some $R_{f_*^{k-1}(S^\star)}\circ\cdots\circ
R_{S^\star}$. We denote the marked disk  that contains
$\Phi^t_{f(E^\star)}(f(\alpha))$ (resp.
$\Phi^t_{f(E^\star)}(f(\beta))$) by $(\Delta_a, a)$ (resp.
$(\Delta_b,b)$). It can happen that $(\Delta_a, a)=(\Delta_b,b)$.
Let $\Delta_\alpha$  (resp. $\Delta_\beta$) be the component of
$R_{E^\star}^{-1}(\Delta_a)$  (resp. $R_{E^\star}^{-1}(\Delta_b)$)
that contains $\Phi^t_{E^\star}(\alpha)$ (resp.
$\Phi^t_{E^\star}(\beta)$). Then $\Delta_\alpha$ (resp.
$\Delta_\beta$) contains a marked point in $P(E^\star)$, say
$z_\alpha$ (resp. $z_\beta$). The marked disks
$(\Delta_\alpha,z_\alpha)$ and $(\Delta_\beta,z_\beta)$ are called
the marked disks of $R_{E^\star}$. They are independent of the
choice of $t$. Clearly, $\Phi^t_{E^\star}(\alpha)$ is an
equipotential in the   marked disk $(\Delta_\alpha,z_\alpha)$ and
$\Phi^t_{E^\star}(\beta)$ is an equipotential in the   marked disk
$(\Delta_\beta,z_\beta)$, with potentials \bess
\varpi(\Phi^t_{E^\star}(\alpha))&=&\frac{\varpi(\Phi^t_{f(E^\star)}(f(\alpha)))}{\deg{(f|_\alpha})}=\frac{\sigma_t(f(\alpha))}{\deg(f|_\alpha)},\\
\varpi(\Phi^t_{E^\star}(\beta))&=&\frac{\varpi(\Phi^t_{f(E^\star)}(f(\beta)))}{\deg{(f|_\beta})}=\frac{\sigma_t(f(\beta))}{\deg(f|_\beta)}.\eess

We denote by $A(E^\star)\subset\mathbb{\overline{C}}(E^\star)$ the
annulus bounded by $\alpha$ and $\beta$. By the `reversed Gr\"otzsch
inequality' (see Theorem 2.2.3 in \cite{W}, or Lemma B.1 in
\cite{CT1}), there is a constant $C(E^\star)$, independent of the
parameter $t$, such that

$${\rm mod} (\Phi^t_{E^\star}(A(E^\star)))\leq
\frac{\sigma_t(f(\alpha))}{\deg(f|_\alpha)}+\frac{\sigma_t(f(\beta))}{\deg(f|_\beta)}+C(E^\star).$$

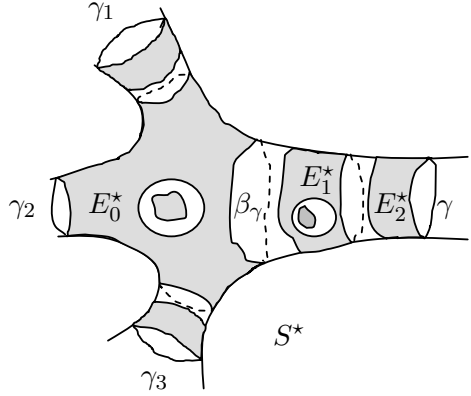
\begin{figure}[h]
\centering{
\includegraphics[height=6cm]{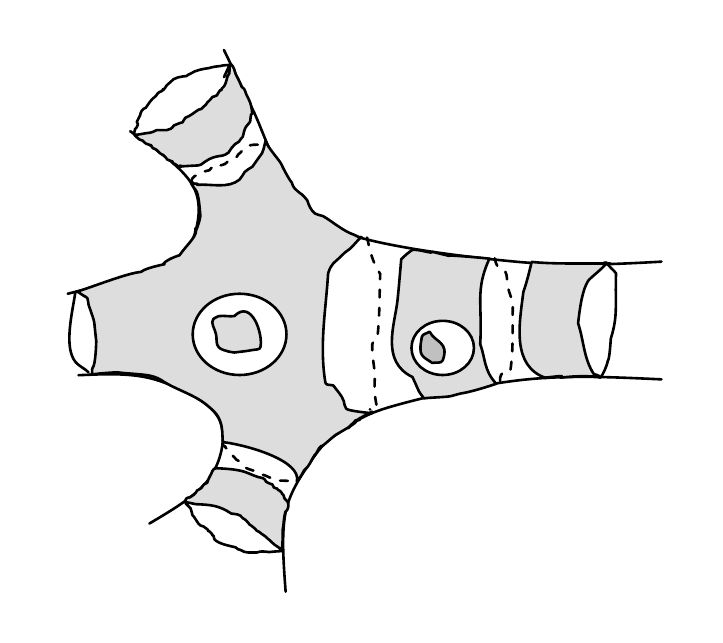}  \put(-90, 30){$S^\star$}
 \put(-160, 80){$E^\star_0$} \put(-80, 90){$E^\star_1$}  \put(-52, 80){$E^\star_2$} \put(-28, 80){$\gamma$} \put(-105, 80){$\beta_\gamma$}
\put(-160, 155){$\gamma_1$} \put(-190, 80){$\gamma_2$}  \put(-140,
15){$\gamma_3$}
 %%\put(-230, 30){$S$}
%\put(-250,180){$E_1$}  \put(-263,140){$S_1$}
% \put(-130,200){$E_2$}
%\put(-119,140){$S_2$}
% \put(-273,55){$E_3$}
%\put(-263,5){$S_3$} \put(-260,45){$\bullet \ p_3$}
%\put(-260,65){$\bullet \ q_3$} \put(-100,35){$E_4$}
%\put(-130,5){$S_4$} \put(-130,35){$\bullet \ p_4$}
%\put(-130,65){$\bullet \ q_4$}
 \caption{A $\mathcal{S}^\star$-piece $S^\star$ with boundary $\partial S^\star=\gamma\cup \gamma_1\cup\gamma_2\cup\gamma_3$.
 Here, $E^\star_0$ is the $\mathcal{E}^\star_{ess}$-piece parallel to $S^\star$, $E^\star_1$ and $E^\star_2$ are two
 $\mathcal{E}^\star_A$-pieces between $\gamma$  and $\beta_\gamma$.}}
\end{figure}%\label{fig5-1}

\subsection{ $\lambda(\Gamma,f)<1$ implies Gr\"otzsch inequality}
\label{a5}

For any $\mathcal{S}^\star$-piece $S^\star$ and any
$\gamma\in\partial_1(S^\star)$, let $A_{S^\star}^\gamma$ be the
annulus bounded by $\gamma$ and $\beta_\gamma$. By the construction
of $\Phi^t_{S^\star},
\Psi^t_{S^\star}:\mathbb{\overline{C}}(S^\star)\rightarrow\mathbb{\overline{C}}$,
both $\Phi^t_{S^\star}(\gamma)$ and $\Psi^t_{S^\star}(\beta_\gamma)$
are equipotentials. We denote the annulus  between
$\Phi^t_{S^\star}(\gamma)$ and $\Psi^t_{S^\star}(\beta_\gamma)$ by
$A^t(S^\star,\gamma)$. It's obvious that

$${\rm
mod}(A^t(S^\star,\gamma))=\varpi(\Phi^t_{S^\star}(\gamma))-\varpi(\Psi^t_{S^\star}(\beta_\gamma))
=\sigma_t(\gamma)-\frac{\sigma_t(f(\beta_\gamma))}{\deg(f|_{\beta_\gamma})}.$$

Then we have the following

\begin{lem}[Large parameter implies Gr\"otzsch inequality]\label{b4g}  When $t$ is large enough, for any $\mathcal{S}^\star$-piece $S^\star$ and any
$\gamma\in\partial_1(S^\star)$, we have
$$\sum_{\mathcal{E}^\star_A\ni E^\star\subset \overline{A_{S^\star}^\gamma}}{\rm
mod}(\Psi_{E^\star}^t(A(E^\star)))< {\rm
mod}(A^t(S^\star,\gamma)),$$ where the sum is taken over all
the $\mathcal{E}^\star_A$-pieces contained in
$\overline{A_{S^\star}^\gamma}$.
\end{lem}

\begin{proof}
%For any $\mathcal{E}^\star_A$-piece $E^\star\subset \overline{A_{S^\star}^\gamma}$,
% by the construction of $\Psi^t_{E^\star}$, we have that
%$${\rm mod} (\Psi^t_{E^\star}(A(E^\star)))\leq
%\frac{\sigma_t(f(\alpha_{E^\star}))}{\deg(f|_{\alpha_{E^\star}})}
%+\frac{\sigma_t(f(\beta_{E^\star}))}{\deg(f|_{\beta_{E^\star}})}+C(E^\star).$$
%where $\alpha_{E^\star}, \beta_{E^\star}$ are the two
%boundary curves of $E^\star$ that are homotopic  to $\gamma$ in
%$\mathbb{\overline{C}}-P$.
It suffices to show that when $t$ is large enough,
$$\sum_{\mathcal{E}^\star_A\ni E^\star\subset \overline{A_{S^\star}^\gamma}}\Bigg(
\frac{\sigma_t(f(\alpha_{E^\star}))}{\deg(f|_{\alpha_{E^\star}})}
+\frac{\sigma_t(f(\beta_{E^\star}))}{\deg(f|_{\beta_{E^\star}})}+C(E^\star)\Bigg)+
\frac{\sigma_t(f(\beta_\gamma))}{\deg(f|_{\beta_\gamma})}<
t\cdot\rho(S_\gamma, \pi(\gamma))\cdot v(\pi(\gamma)),$$ where
$\alpha_{E^\star}$ and $\beta_{E^\star}$ are the boundary curves of
$E^\star$,   homotopic  to $\gamma$ in $\mathbb{\overline{C}}-P$.

One can verify that
 \bess  \sum_{\mathcal{E}^\star_A\ni E^\star\subset \overline{A_{S^\star}^\gamma}}\Bigg(
\frac{\sigma_t(f(\alpha_{E^\star}))}{\deg(f|_{\alpha_{E^\star}})}
+\frac{\sigma_t(f(\beta_{E^\star}))}{\deg(f|_{\beta_{E^\star}})}\Bigg)+
\frac{\sigma_t(f(\beta_\gamma))}{\deg(f|_{\beta_\gamma})}=\sum_{\beta\in
\Sigma^\star} \sum_{\alpha\sim\gamma,\alpha\subset S^\star_\gamma
}\frac{\sigma_t(\beta)}{{\rm deg}(f:\alpha\rightarrow\beta)}.\eess

Since $\Sigma^\star=\Gamma_0^\star\cup \Gamma^\star$, we can
decompose the sum into two parts:
$$I=\sum_{\beta\in \Gamma^\star}
\sum_{\alpha\sim\gamma,\alpha\subset S^\star_\gamma
}\frac{\sigma_t(\beta)}{{\rm deg}(f:\alpha\rightarrow\beta)}, \
II=\sum_{\beta\in\Gamma_0^\star}
\sum_{\alpha\sim\gamma,\alpha\subset S^\star_\gamma
}\frac{\sigma_t(\beta)}{{\rm deg}(f:\alpha\rightarrow\beta)}.$$

It follows from the proof of Lemma \ref{b4c} that $I\leq t
\omega(\gamma)$, where
$$\omega(\gamma):=\frac{
\rho(S_{f(\gamma)}, \pi({f(\gamma)}))v(\pi({f(\gamma)}))}{{\rm
deg}(f|_\gamma)}+\sum_{\delta\in \Gamma}\sum_{
\pi(\gamma)\sim\alpha\subset{S}_{\gamma}\setminus{\pi(\gamma)}}
\frac{v(\delta)}{{\rm deg}(f:\alpha\rightarrow\delta)},$$ if
$f(\gamma)\in \Gamma^\star$ (or equivalently
$\gamma\in\Gamma_2^\star\cup \Gamma_3^\star\cup\cdots$); and
$$\omega(\gamma):=\sum_{\delta\in
\Gamma}\sum_{
\pi(\gamma)\sim\alpha\subset{S}_{\gamma}\setminus{\pi(\gamma)}}
\frac{v(\delta)}{{\rm deg}(f:\alpha\rightarrow\delta)},$$ if
$f(\gamma)\in \Gamma_0^\star$ (or equivalently
$\gamma\in\Gamma_1^\star$).

For the second term, we have

$$II\leq\sum_{A\in\mathcal{A}}\sum_{\alpha\in
f^{-1}(\Gamma_0)\setminus\Gamma_0}\frac{{\rm
mod}(A)}{\deg(f|_\alpha)} ,$$
 where $\mathcal{A}$
is the collection of all rotation annuli of $(f,P)$.

So if we choose $t$ large enough such that for any
$\gamma\in\cup_{S^\star\in\mathcal{S}^\star}
\partial_1(S^\star)$,
\bess
\sum_{E^\star\in\mathcal{E}^\star_A}C(E^\star)+\sum_{A\in\mathcal{A}}\sum_{\alpha\in
f^{-1}(\Gamma_0)\setminus\Gamma_0}\frac{{\rm
mod}(A)}{\deg(f|_\alpha)}<t\Big(\rho(S_\gamma, \pi(\gamma))\cdot
v(\pi(\gamma))- \omega(\gamma)\Big), \eess
  then the conclusion
follows (notice that by the choice of the number $\rho$, we have
 $\rho(S_\gamma,
\pi(\gamma))\cdot v(\pi(\gamma))- \omega(\gamma)>0$ for all
$\gamma\in\Gamma^\star$).
\end{proof}

\subsection{Proof of the sufficiency of  Theorem \ref{b4a} ($\Gamma\neq\emptyset$)} \label{a6}

Now, we are ready to complete the proof  of  Theorem \ref{b4a}. Here
is a fact used in the proof, which is equivalent to the Gr\"otzsch
inequality. Let $A,B\subset\mathbb{\overline{C}}$ be two annuli. We
say that $B$ can be embedded into $A$ essentially and
holomorphically if there is a holomorphic injection
$\phi:B\rightarrow A$ such that $\phi(B)$ separates the  two
boundary components of $A$.

\vskip0.3cm
\noindent{\bf Fact} {\it Let $A,A_1,\cdots, A_n\subset\mathbb{\overline{C}}$
be annuli, then $A_1,\cdots, A_n$ can be embedded into   $A$
essentially and holomorphically such that the closures of the images
of $A_i$'s are mutually disjoint if and only if
$$\sum_{i=1}^n {\rm mod}(A_i)< {\rm mod}(A).$$}

\noindent\textit{Proof of the sufficiency of  Theorem \ref{b4a},
assuming   $\Gamma\neq\emptyset$.} \label{aab}
The idea of the proof is to glue the holomorphic models in a suitable fashion
along the stable multicurve $\Gamma$.

Recall that for each $\mathcal{S}^\star$-piece $S^\star$, we use $S$
to denote the  $\mathcal{S}$-piece that contains $S^\star$. For each
curve $\gamma\in \Sigma $,    $A^\gamma$ is the annular neighborhood
of $\gamma$ chosen at the beginning of Section \ref{a3}. The
collection of all rotation annuli of $(f,P)$ is still denoted by
$\mathcal{A}$.

 For each  $\mathcal{S}^\star$-piece
$S^\star$, we extend $\Phi_{S^\star}^t:S^\star\rightarrow
\Phi_{S^\star}^t(S^\star)$ to a quasi-conformal homeomorphism
$\Phi_{S}:{S}\rightarrow \Phi_{S}(S)$ such that $\Phi_{S}$ is
holomorphic in $(S-S^\star)\cap(\cup\mathcal{A})$.

We first choose $t$ large enough such that Lemma \ref{b4g} holds.
This means, one can embedded $\Psi_{E^\star}^t(E^\star)$
holomorphically into the interior of $\Phi_S(S)$ for each
$\mathcal{E}^\star_A$-piece $E^\star$ contained in $S$ according to
the original order of their non-peripheral boundary curves  so that
the embedded images are mutually disjoint.  In other words, there is
a quasi-conformal homeomorphism  $\psi_S:S\rightarrow \Phi_S(S)$ such
that

$\bullet$ $\psi_S|_{\partial S}=\Phi_S|_{\partial S}$ and $\psi_S$
is isotopic to $\Phi_S$ rel $\partial S\cup (S\cap P)$. Moreover,
$\psi_S|_{S\cap(\cup\mathcal{A})}=\Phi_S|_{S\cap(\cup\mathcal{A})}$.

$\bullet$ $\psi_S|_{E_{S^\star}}=\Psi_{S^\star}^t|_{E_{S^\star}}$,
where $E_{S^\star}$ is the unique $\mathcal{E}^\star$-piece parallel
to $S^\star$.

$\bullet$ For each curve $\gamma\in  \partial_1(S)$, $\Phi_S(S\cap
A^\gamma)=\psi_S(S\cap A^\gamma)$.

$\bullet$ For every $\mathcal{E}^\star_A$-piece $E^\star$ with
$E^\star\subset S$, the map $\Psi_{E^\star}^t\circ \psi_S^{-1}$ is
holomorphic in $\psi_S(E^\star)$.

We define a subset $\mathcal{E}_{A}$ of $\mathcal{E}$ by
$\mathcal{E}_{A}=\{E; E^\star\in \mathcal{E}^\star_{A}\}$. Let
$\mathcal{D}(S)$ be the collection of all disk components of
$S-E_S\cup(\cup_{\mathcal{E}_A\ni E\subset S}E)$, here $E_{S}$ is
the unique $\mathcal{E}$-piece parallel to $S$. For each $D\in
\mathcal{D}(S)$, we construct a quasi-conformal homeomorphism
$\zeta_D:D\rightarrow \psi_S(D)$, whose Beltrami coefficient
satisfies
$$\mu_{\zeta_D}(z)=\sum_{\mathcal{E}\ni E\subset D}\chi_E(z)\mu_{\Phi_{f(E)}\circ f}(z),$$
here the sum is taken over all $\mathcal{E}$-pieces contained
in $D$. We further require $\zeta_D(p)=\psi_S(p)$ if $D$ contains a
marked point $p\in P$.

Let $\Gamma_S$ be the collection of all  boundary curves of
$\cup_{D\in \mathcal{D}(S)} D$.  For each $\gamma\in \Gamma_S$,
notice that $f(\gamma)\in \Sigma$.  Let $A^\gamma$ be the component
of $f^{-1}(A^{f(\gamma)})$ containing $\gamma$. It's obvious that
$A^\gamma$ is an annular neighborhood of $\gamma$. We define a
quasi-conformal homeomorphism $\Psi_S:S\rightarrow \Phi_S(S)$ by

\begin{equation*}
\Psi_S(z)=\begin{cases}
 \zeta_D(z),\ \  &z\in D, D\in \mathcal{D}(S),\\
\psi_{S}(z),\ \ &z\in  S-(\cup_{D\in \mathcal{D}(S)}D)\cup(\cup_{\gamma\in \Gamma_S}A^\gamma),\\
q.c \  interpolation, \ \ &z\in \cup_{\gamma\in
\Gamma_S}A^\gamma-\cup_{D\in \mathcal{D}(S)} D.
\end{cases}
\end{equation*}

The map $\Psi_S$ satisfies:

$\bullet$ $\Psi_S|_{\partial S}=\Phi_S|_{\partial S}$ and $\Psi_S$
is isotopic to $\Phi_S$ rel $\partial S\cup (S\cap P)$. Moreover,
$\Psi_S|_{S\cap(\cup\mathcal{A})}=\Phi_S|_{S\cap(\cup
\mathcal{A})}$.

$\bullet$ For every $\mathcal{E}^\star_{ess}\cup
\mathcal{E}^\star_A$-piece $E^\star\subset S$, the map
$\Phi_{f(E)}\circ f\circ \Psi_S^{-1}$ is holomorphic in
$\Psi_S(E^\star)$.

$\bullet$ For every $\mathcal{E}$-piece ${E}\subset \cup_{D\in
\mathcal{D}(S)} D$, the map $\Phi_{f(E)}\circ f\circ \Psi_S^{-1}$ is
holomorphic in $\Psi_S(E)$.

Now, we define a quasi-conformal map
$\Theta:\mathbb{\overline{C}}\rightarrow \mathbb{\overline{C}}$ by
$\Theta|_S=\Psi_S^{-1}\circ \Phi_S$ for all $S\in\mathcal{S}$. It's
obvious that  $\Theta$ is isotopic to the identity map rel $P$.
Moreover, for each curve $\gamma\in \Gamma$, we have
$\Theta(\gamma)=\gamma$ and $A^\gamma\subset\Theta^{-1}(A^\gamma)$.
Let $\Phi:\mathbb{\overline{C}}\rightarrow\mathbb{\overline{C}}$ be
the quasi-conformal map whose Beltrami coefficient satisfies
$$\mu_\Phi(z)=\sum_{S\in\mathcal{S}}\chi_S(z)\mu_{\Phi_S}(z), \  z\in\mathbb{\overline{C}}.$$

%and define $g=\Phi\circ f\circ \Theta\circ \Phi^{-1}$.
Set $\Psi=\Phi\circ\Theta^{-1}$. Then $(f,P)$ is q.c-equivalent to
the Herman map $(g,Q):=(\Phi\circ f\circ\Psi^{-1}, \Phi(P))$ via
$(\Phi,\Psi)$.

One can verify that $g$ is holomorphic outside
$X:=\Psi(\cup_{\gamma\in \Gamma\cup (\cup_{S\in\mathcal{S}
}\Gamma_S)}A^\gamma)$. To see this, notice that if $E^\star\in
\mathcal{E}^\star_{ess}\cup \mathcal{E}^\star_A$ and $E^\star$ is
contained in some $\mathcal{S}$-piece $S$, then the decomposition
$$g|_{\Psi(E^\star)}=(\Phi\circ\Phi_{f(E)}^{-1})\circ(\Phi_{f(E)}\circ
f\circ\Psi_S^{-1})\circ(\Phi_S\circ\Phi^{-1})|_{\Psi(E^\star)}$$
implies that $g$ is holomorphic in $\Psi(E^\star)$ since each factor
is holomorphic. If $E\in \mathcal{E}$ and $E\subset D\in
\mathcal{D}(S)$, then
$$g|_{\Psi(E)}=(\Phi\circ\Phi_{f(E)}^{-1})\circ(\Phi_{f(E)}\circ
f\circ\zeta_D^{-1})\circ(\Phi_S\circ\Phi^{-1})|_{\Psi(E)},$$  so $g$
is holomorphic in $\Psi(E)$.

%So $g$ is holomorphic outside $\mathbb{\overline{C}}-\Psi(\cup_{D\in
%\mathcal{D}(S)}D\cup \cup_{E^\star\in
%\mathcal{E}^\star_{ess}\cup
%\mathcal{E}_A)}E^\star$

%\begin{equation*}
%g|_{U}=\begin{cases}
% (\Phi\circ\Phi_{f(E_S)}^{-1})\circ(\Phi_{f(E_S)}\circ
%f\circ\Psi_S^{-1})\circ(\Phi_S\circ\Phi^{-1})|_{\Psi(E_{S^\star})},\ \
%&\text{if } U=\Psi(E_{S^\star}), E_{S^\star}\in\mathcal{E}^\star_{ess},\\
% (\Phi\circ\Phi_{f(E)}^{-1})\circ(\Phi_{f(E)}\circ
%f\circ\Psi_S^{-1})\circ(\Phi_S\circ\Phi^{-1})|_{\Psi(E^\star)},\ \  &\text{if } U=\Psi(E^\star), E^\star\in \mathcal{E}_A ,\\
%interpolation, \ \ &z\in \cup_{\gamma\in
%\Omega_S}A(\gamma)-\cup_{D\in \mathcal{D}(S)} D.
%\end{cases}
%\end{equation*}

The last step is to apply the quasi-conformal surgery. For each curve
$\gamma\in \Gamma$, let $\iota(\gamma)$ be the first integer
$p\geq1$ such that $f^{p}(\gamma)\in \Gamma_0$ and
$L=\max_{\gamma\in \Gamma}\iota(\gamma)$.
%Let $R_A$ be the union of
%all rotation annuli of $(f,P)$,
 One may verify by induction that for
any $j\geq1$,
$$g^{-j}(\Psi(\cup\mathcal{A}))=\Psi((\Theta\circ f)^{-j}(\cup\mathcal{A}))\supset \Psi(\cup_{\gamma\in \Gamma, \iota(\gamma)\leq j}A^\gamma).$$

In particular, $g^{-L-1}(\Psi(\cup\mathcal{A}))\supset X$. Let
$\sigma_0$ be the standard complex structure in
$\mathbb{\overline{C}}$. Define a $(g,Q)$-invariant complex
structure $\sigma$ by

\begin{equation*}
\sigma=\begin{cases}
 (g^k)^*(\sigma_0),\ \  &\text{ in } g^{-k}(\Psi(\cup\mathcal{A}))\setminus g^{-k+1}(\Psi(\cup\mathcal{A})), \ k\geq1,\\
\sigma_0,\ \ &\text{ in } \mathbb{\overline{C}}-\cup_{k\geq1}(
g^{-k}(\Psi(\cup\mathcal{A}))\setminus
g^{-k+1}(\Psi(\cup\mathcal{A}))).
\end{cases}
\end{equation*}

Since $(g,Q)$ is  holomorphic outside $X$, the Beltrami coefficient
$\mu$ of $\sigma$ satisfies $\|\mu\|_\infty<1$. By Measurable
Riemann Mapping Theorem, there is a quasi-conformal map
$\zeta:\mathbb{\overline{C}}\rightarrow \mathbb{\overline{C}}$ such
that $\zeta^*(\sigma_0)=\sigma$. Let $R=\zeta\circ
g\circ\zeta^{-1}$, then $R$ is a rational map and $(f, P)$ is
q.c-equivalent to $(R, \zeta\circ\Phi(P))$ via
$(\zeta\circ\Phi,\zeta\circ\Psi)$. \hfill $\Box$

\section{Analytic part: renormalizations}\label{3-6}

In this section, we discuss rational-like maps, renormalizations of rational maps and prove Theorem \ref{b3-11}.

\subsection{Rational-like\ maps} \label{3-4-1}

A {\it rational-like map} $g: U\rightarrow V$ is a proper and
holomorphic map between two
 multi-connected domains such that  $\overline{U}\subset
 V\subset \mathbb{\overline{C}}$ and the complementary set $\mathbb{\overline{C}}-{X}$  of $X\in\{U,V\}$ consists of finitely
 many  topological disks. In our discussion, we always assume $V\neq
 \mathbb{\overline{C}}$ and the degree of $g$ is at least two.
 The {\it filled
  Julia set } is defined by $K(g)=\bigcap_{n\geq 1}g^{-n}(V)$,  the {\it Julia set } is defined by
  $J(g)=\partial K(g)$. %The filled Julia set $K(g)$ is not  necessarily a full set. This implies that
  The Julia set $J(g)$ is not necessarily connected even if $K(g)$ is connected (but in the polynomial-like case,
  the connectivity of $K(g)$ always implies
   the connectivity of $J(g)$ by the Maximum Modulus Principle).

     Two rational-like maps $g_1$ and $g_2$ are {\it hybrid equivalent} if there is
 a quasi-conformal conjugacy $\phi$ between $g_1$ and $g_2$, defined
 in a neighborhood of $K(g_1)$, such
 that $\overline{\partial}\phi=0$ on $K(g_1)$. We call
 $\phi$ a {\it hybrid conjugacy} between $g_1$ and $g_2$. These definitions are
 simply the generalizations of
 Douady-Hubbard's definitions for polynomial-like maps \cite{DH3}.

The following is an analogue of Douady-Hubbard's straightening
theorem:

 \begin{thm}[Straightening theorem] \label{b5a0}  Let $g: U\rightarrow V$
be a rational-like map of degree $d\geq2$, then

1. The map $g$ is hybrid equivalent to a rational map $R$ of degree
$d$.

2. If $K(g)$ is connected, then $g$ is hybrid equivalent to a
rational map $R$ of degree $d$, which is postcritically finite
outside $\phi(K(g))$. Here $\phi$ is the hybrid conjugacy. Such $R$
is unique up to M\"obius conjugation.
\end{thm}

%if we assume $\deg(R)=d$ and
%$R$ is poscritically finite outside $\phi(K(g))$, where $\phi$ is
%the hybrid conjugacy.
 %We don't require $U$ is compactly contained in $V$

\begin{rem}\label{b5a2011} 1.  A rational-like map $g: U\rightarrow V$ can be  hybrid equivalent to a rational map  of
degree greater than $d$.

2. In the second statement of Theorem \ref{b5a0}, if we do not
require the degree of $R$, then $R$ may be not unique up to M\"obius
conjugation even if $R$ is postcritically finite outside
$\phi(K(g))$. %Even if $K(g)$ is connected, the rational-like map $g$
%can be hybrid equivalent to a rational map  of degree greater than
%$d$, which is postcritically finite outside $\phi(K(g))$.
 For example, we can consider the McMullen map:
$f_\lambda(z)=z^n+\lambda/z^n$ with $n\geq3$. Here  $\lambda$ is a
complex parameter such that $f_\lambda$ is postcritically finite and
the Julia set is a Sierpinski curve. We denote by $B_\lambda$ the
immediate attracting basin of $\infty$. In this case,   $f_\lambda$
is strictly expanding on $\partial B_\lambda$.
%If all critical points of $f_\lambda$ are attracted by
%$\infty$ and the map is postcritically finite, then each Fatou
%component is a quasidisk (By a Theorem of Devaney).
%The map $f_\lambda$ is conjugate to a $z\mapsto z^n$ on the boundary of attracting basin $B_\lambda$ of $\infty$.
There is % a restriction of $f_\lambda$ in
 an annular neighborhood $A$
of $\partial B_\lambda$ such that $f_\lambda|_A: A\rightarrow
f_\lambda(A)$ is a  rational-like map. It is  hybrid equivalent to
the power map $z\mapsto z^n$, with degree lower than that of
$f_\lambda$. More details can be found in \cite{QWY}.

3. If $K(g)$ is connected and $\mathbb{\overline{C}}-K(g)$ consists
of two components, then there are two annuli $U',V'$ such that
$K(g)\Subset U'\Subset V' \Subset V$ and the restriction $g|_{U'}:
U'\rightarrow V'$ is a rational-like map. In this case, $K(g)$ is a
quasi-circle.
\end{rem}

\begin{proof} 1. The proof is a standard surgery procedure. By shrinking $V$ a little bit,  we may
assume that each boundary curve of $U$ and $V$ is a quasi-circle. We
then extend $g: U\rightarrow V$ to a quasi-regular branched covering
$G:\mathbb{\overline{C}}\rightarrow \mathbb{\overline{C}}$ such that
$G$ is holomorphic in $\mathbb{\overline{C}}-\overline{V}$ and $G$
maps each component $U_k$ of $\mathbb{\overline{C}}-{U}$ onto a
connected component $V_j$ of $\mathbb{\overline{C}}-{V}$, with
degree equal to $\deg{(g|_{\partial U_k})}$. Such extension keeps
the degree. By pulling back the standard complex structure
$\sigma_0$ on $\mathbb{\overline{C}}-{V}$ via $G$, we get a
$G$-invariant complex structure

\begin{equation*}
\sigma=\begin{cases}
 (G^k)^*(\sigma_0),\ \  &\text{ in } G^{-k}(\mathbb{\overline{C}}-{V}), \ k\geq1,\\
\sigma_0,\ \ &\text{ in } K(g).
\end{cases}
\end{equation*}

The  Beltrami coefficient $\mu$ of $\sigma$ satisfies
$\mu|_{K(g)}=0$ and $\|\mu\|_\infty<1$. Let  $\phi$ solve the
Beltrami equation $\overline{\partial} \phi=\mu \partial \phi$. Then
$R=\phi\circ G\circ\phi^{-1}$ is a rational map and $\phi$ is a
hybrid conjugacy between $g$ and $R$.

2. By a  hole-filling process (see Theorem 5.1 in \cite{CT1}, or
Proposition 6.5.1 in \cite{W}), we can find a suitable restriction
$g|_{U'}: U'\rightarrow V'$ of $g$
 with
$K(g)\Subset U' \Subset V'\Subset V$ such that

a). All postcritical points of  $g|_{U'}$ in $V'$ are contained in
$K(g)$.

b). Each connected component of $V'-\overline{U'}$ is either an
annulus or a disk.

Note that such $V'$ can be chosen arbitrarily close to the filled
Julia set $K(g)$. (To see this, one may replace $V'$ by $g^{-k}(V')$
for some large $k$, and a), b) still holds.)

In this way, each component $U_i$ of
$\mathbb{\overline{C}}-\overline{U'}$ either  is contained in $V'$
or contains a unique component $V_j$ of
$\mathbb{\overline{C}}-\overline{V'}$.  In the former case, we mark
a point $p\in U_i$ and get a marked disk $(U_i,p)$; in the latter
case, we mark a point $p\in V_j$, and get two marked disks $(V_j,p)$
and $(U_i,p)$. We extend $g|_{U'}$ to a quasi-regular branched
covering $G:\mathbb{\overline{C}}\rightarrow \mathbb{\overline{C}}$
such that

1). For each component $U_i$ of
$\mathbb{\overline{C}}-\overline{U'}$, $G$ maps the marked disk
$(U_i,p)$ to  the marked  disk $(V_k,q)$, where $V_k$ is the
component of $\mathbb{\overline{C}}-\overline{V'}$ whose boundary is
$g(\partial U_i)$. We require that $G(p)=q$ and  the local degree of
$G$ at $p$ is equal to $\deg{(g|_{\partial U_i})}$.

2). We further require that $G$ is holomorphic in
$\mathbb{\overline{C}}-\overline{V'}$.

%The same argument as above,
By pulling back the standard complex structure on
$\mathbb{\overline{C}}-{V'}$, we can get a $G$-invariant complex
structure whose Beltrami coefficient $\mu$ satisfies $\mu|_{K(g)}=0$
and $\|\mu\|_\infty<1$. Let  $\phi$ solve the Beltrami equation
$\overline{\partial} \phi=\mu \partial \phi$. Then $R=\phi\circ
G\circ\phi^{-1}$ is a rational map, postcritically finite outside
$\phi(K(g))$, as required.

To prove the uniqueness, we need investigate
 some mapping properties of $R$,  a
rational map of degree $d$, to which $g|_{U'}$ is hybrid equivalent
via $\phi$, and
 postcritically finite outside $\phi(K(g))$.  We assume $V'$ is sufficiently close to $K(g)$
 so that
$\phi$ is defined on $V'$. Then $g|_{U'}$ induces a suitable
restriction $R|_{\phi(U')}$. Let $\mathcal{X}_1$ be the collection
of all components of $\mathbb{\overline{C}}-\phi(K(g))$ which
intersect with the boundary curves of $\phi(V')$ and $\mathcal{X}_2$
be the collection of all components of
$\mathbb{\overline{C}}-\phi(K(g))$ which intersect with the boundary
curves of $\phi(U')$. It's obvious that
$\mathcal{X}_1\subset\mathcal{X}_2$. Since the degree of $R$ is
equal to $d$ (this is very important), we have that
%when $X$ ranges
%over all elements of $\mathcal{X}_1$, the components of
$$\{U {\text \ is\ a\ component\ of\ } R^{-1}(X); X\in \mathcal{X}_1\}=\mathcal{X}_2.$$
%$$\bigcup_{X\in \mathcal{X}_1}R^{-1}(X)=\bigcup_{Y\in \mathcal{X}_2}Y.$$
%for any $X\in \mathcal{X}_1$, each component $Y$ of $R^{-1}(X)$ is
%an element of $\mathcal{X}_2$.
Thus for each $X\in \mathcal{X}_2$, $R(X)\in \mathcal{X}_2$. This
implies that each $X\in \mathcal{X}_2$ is eventually periodic under
the map $R$. Suppose $X\in \mathcal{X}_2$ is $R$-periodic, with
period $p$. Since $R$ is poscritically finite outside $\phi(K(g))$,
$R^p|_X:X\rightarrow X$ is proper and each critical point in $X$ has
finite orbit. Thus  $R^p|_X$ is conformally conjugate to $z\mapsto
z^d$, where $d={\deg}(R^p|_X)\geq2$ (see
\cite{DH1} Lemma 4.1 for a proof of this fact). It follows that for all $X\in \mathcal{X}_2$,
the proper map $R|_X:X\rightarrow R(X)$ has only one possible
critical point, which is eventually mapped to a superattracting
cycle. Base on these  observations, we are now ready to  prove the
uniqueness part of the theorem.

Suppose that $R_1$ and $R_2$ are two rational maps of degree $d$,
both are hybrid equivalent to $g|_{U'}$  and poscritically finite
outside $\phi_1(K(g))$ and $\phi_2(K(g))$,  respectively. Here,
$\phi_i$ is a hybrid conjugacy between $g|_{U'}$ and $R_i$, $i=1,2$.
We assume that $V'$ is sufficiently close to $K(g)$ such that
$\phi_i$ is defined on $U'$. Then $g|_{U'}$ induces two restrictions
$R_i|_{\phi_i(U')}, i=1,2$ and a hybrid conjugacy
$\phi=\phi_2\circ\phi_1^{-1}$ between them. One can
  construct a pair of quasi-conformal maps $\varphi_0,\varphi_1:
  \mathbb{\overline{C}}\rightarrow \mathbb{\overline{C}}$ such that

a). $\varphi_0\circ R_1= R_2\circ\varphi_1$ on
$\mathbb{\overline{C}}$.

b). $\varphi_0,\varphi_1$ are isotopic rel $\phi_1(K(g))\cup
P_{R_1}$ and
$\varphi_0|_{\phi_1(U')}=\varphi_1|_{\phi_1(U')}=\phi|_{\phi_1(U')}$.

c). $\varphi_0,\varphi_1$ are holomorphic and identical in a
neighborhood $N$ of all superattracting cycles of $R_1$ in
$\mathbb{\overline{C}}-\phi_1(K(g))$.

Then  there is a sequence of quasi-conformal maps
$\{\varphi_n,n\geq0\}$ such that $\varphi_n\circ R_1=
R_2\circ\varphi_{n+1}$ and $\varphi_n$ is isotopic to
$\varphi_{n+1}$ rel $R_1^{-n}(\phi_1(U')\cup  P_{R_1}\cup N)$. The
quasi-conformal map $\varphi_n$ satisfies $\overline{\partial}
\varphi_n=0$  on $\phi_1(K(g))\cup R_1^{-n}(N)$. The sequence
$\{\varphi_n\}$ has a limit quasi-conformal map $\varphi=\lim
\varphi_n$. Since the Lebesgue measure of
$\mathbb{\overline{C}}-\phi_1(K(g))\cup R_1^{-n}(N)$ tends to zero
as $n\rightarrow\infty$, the map $\varphi$ satisfies
$\overline{\partial} \varphi=0$ outside a zero measure set.  It is
 in fact a holomorphic conjugacy between $R_1$ and $R_2$.
\end{proof}

\subsection{Herman-Siegel renormalization} \label{3-4-2-1}

In this section, we will discuss the renormalizations of Herman
rational maps.

Let $(f,P)$ be  a Herman rational map. The decomposition procedure
in Section \ref{3-1} yields  finitely many $f_*$-cycles of
$\mathcal{S}$-pieces:
$$S_\nu\mapsto f_*(S_\nu)\mapsto\cdots\mapsto f_*^{p_\nu-1}(S_\nu)\mapsto f_*^{p_\nu}(S_\nu)=S_\nu,\ \   1\leq \nu\leq n,$$
where $S_\nu$ is a representative of the $\nu$-th  cycle and $p_\nu$
is the period of  $S_\nu$.

For $i\in[1,n]$, let $V_i=S_i$ and $U_i$ be the unique component of
$f^{-p_i}(S_i)$ %that is contained in $S_i$ and
 parallel to  $S_i$.
The triple $(f^{p_i},U_i,V_i)$ can be considered as a
renormalization of $(f,P)$. In general, $U_i$ is not   contained in
the interior of $V_i$ (for example, if  some boundary curve $\gamma$
of $V_i$ is also a curve in $\Gamma_0$, then $\gamma$ is necessarily
a boundary curve of $U_i$). For this reason, we call
$(f^{p_i},U_i,V_i)$ a Herman-Siegel (`HS' for short) renormalization
of $(f,P)$.

We should show that $\deg(f^{p_i}|_{U_i})\geq2$ for all $i\in[1,n]$.
In fact,  if $\Gamma=\emptyset$, then all boundary curves of  $U_i$
are contained in $\Gamma_0$ and $U_i$ contains at least one critical point of
$f^{p_i}|_{U_i}$. In this case, $\deg(f^{p_i}|_{U_i})\geq2$. If
$\Gamma\neq\emptyset$, then it follows from Theorem \ref{b4b} that
$\lambda(\Gamma,f)<1$. We conclude from Lemma \ref{b3aa} that
$\deg(f^{p_i}|_{U_i})\geq2$.

 The filled
  Julia set $K_i$ and Julia set $J_i$ of the  HS renormalization $(f^{p_i},U_i,V_i)$ are  defined as follows:
\bess K_i=\cap_{k\geq 0}(f^{p_i}|_{U_i})^{-k}(U_i), \ J_i= K_i\cap
J(f).\eess
Note that $\partial K_i$ is not
   a reasonable definition of  the Julia set because $\partial K_i$  may contain a curve  in $\Gamma_0$.
One may check that $K_i$ is connected. Moreover, $J_i=\partial K_i$
if and only if $\partial(V_i)\cap \Gamma_0=\emptyset$ (in this case,
$K_i$ is contained in the interior of $V_i$).  We say that
$(f^{p_i},U_i,V_i)$ is hybrid equivalent to a rational map $R$, if
there is a quasi-conformal map $\phi$ defined in an open set  $N_i$
such that

a). $K_i\setminus \partial U_i\subset N_i\subset U_i$;

 b).
$\overline{\partial} \phi=0$ on $K_i\setminus\partial U_i$;

 c).
$\phi\circ f^{p_i}=R\circ \phi$ in $(f^{p_i}|_{U_i})^{-1}(N_i)\cap N_i$.

\begin{thm} [Herman-Siegel renormalization]\label{b5a3}   Let $(f,P)$ be a Herman rational map
and $(f^{p_i},U_i,V_i), i\in [1,n]$ be all the  HS renormalizations
defined above. Then

1.   For each $i\in[1,n]$, the HS renormalization
$(f^{p_i},U_i,V_i)$ is hybrid equivalent to a rational map $R_i$ of
degree $\deg(f^{p_i}|_{U_i})$ which is postcritically finite outside
$\phi(K_i)$. Here $\phi$ is the hybrid conjugacy. Such $R_i$ is
unique up to M\"obius conjugation.

2.  The Julia set $J(f)$ has zero Lebesgue measure (resp. carries no
invariant line fields) if and only if for each $i\in [1,n]$, the
Julia set $J_i$ has zero Lebesgue measure(resp. carries no invariant
line fields).
\end{thm}

Here, we say a rational map  $f$  carries
an invariant line field if there is a measurable Beltrami differential
$\mu=\mu(z){d\bar{z}}/{dz}$ supported on a measurable subset $E\subset J(f)$
such that $E$ has  positive measure and $f^*\mu=\mu$ a.e. (here, $f^*\mu:=\mu(f(w))\frac{\overline{f'(w)}d\bar{w}}{f(w)dw}$).
The definition can be generalized to $f^{p_i}|_{U_i}$ similarly.

The proof of the first statement  is essentially the same  as that of
Theorem \ref{b5a0}, and the straightening map $R_i$ is either a
Siegel rational map or a Thurston rational map. We omit the details
here.
The proof of the second statement is based on the following (\cite{McM2}, Theorem 3.9):

\begin{thm}[Ergodic or attracting] \label{b5c} Let $f$ be a rational map of degree at least
two, then either

$\bullet$ $J(f)=\mathbb{\overline{C}}$ and the action of $f$ on
$\mathbb{\overline{C}}$ is ergodic, or

$\bullet$ the spherical distance $d(f^n(z),P_f)\rightarrow 0$ for
almost every $z\in J(f)$ as $n\rightarrow\infty$.
\end{thm}

\noindent\textit{Proof of 2 of  Theorem \ref{b5a3}.} Let
$\mathcal{E}_{ess}\subset\mathcal{E}$ be the collection of all
$\mathcal{E}$-pieces (defined in Section \ref{3-1}) that are parallel to  the $\mathcal{S}$-pieces.
 Each element $E\in \mathcal{E}\setminus\mathcal{E}_{ess}$ contains at most one point in
the postcritical set $P_f$.  Moreover, the boundary of $E$ is
contained in the Fatou set $F(f)$.

We can define an itinerary map by:

\begin{equation*}
iter :\begin{cases}
 J(f)\rightarrow \mathcal{E}^{\mathbb{N}},\\
z\mapsto (E_0(z), E_1(z), E_2(z),\cdots).
 \end{cases}
 \end{equation*}
where $E_k(z)$ is the unique $\mathcal{E}$-piece  containing
$f^k(z)$.

Given a point $z\in J(f)$ with itinerary $iter(z)=(E_0(z), E_1(z),
E_2(z),\cdots)$, one can verify that $z\in
\cup_{k\geq0}f^{-k}(J_1\cup\cdots\cup J_n)$ if and only if there is
an integer $N$ (depending on $z$) such that for all $k\geq N$,
$E_k(z)\in \mathcal{E}_{ess}$. Moreover, the set
$\cup_{k\geq0}f^{-k}(J_1\cup\cdots\cup J_n)$ contains all
  the boundaries of
rotation domains, together with their preimages.

 This implies that if $z\in
J(f)-\cup_{k\geq0}f^{-k}(J_1\cup\cdots\cup J_n)$, then there exists
a sequence of integers $\{n_j; j\geq1\}$ such that $E_{n_j}(z)\in
\mathcal{E}\setminus\mathcal{E}_{ess}$ for all $j\geq1$.  By passing
to a subsequence, we assume  $\{E_{n_j}(z);j\geq1\}$ satisfies
either of the following three properties:

1. $E_{n_j}(z)\cap P_f=\emptyset$ for all $j\geq1$.

2. For all $j\geq1$, $E_{n_j}(z)\cap P_f\neq\emptyset$ and
$E_{n_j}(z)$ contains a  point in $P_f$. This point is  contained
either in the Fatou set or in the grand orbit of a repelling cycle.

3.  For all $j\geq1$, $E_{n_j}(z)\cap P_f\neq\emptyset$ and
$E_{n_j}(z)$ contains a point in $P_f$ and the forward orbit of this
point accumulates on
  $P'_f\cap J(f)$, where $P'_f$ is the accumulation set of the postcritical set $P_f$.

In the first two cases, one may easily check that $\limsup
d(f^n(z),P_f)>0$.

In the last case, the set  $\{E_{n_j}(z); j\geq1\}$ can be rewritten
as $\{E_1,\cdots, E_m\}$, which is a finite subset of
$\mathcal{E}\setminus\mathcal{E}_{ess}$. Note that each $E_k$ is
contained in a disk component or an annular component of $\cup
\mathcal{S}-\cup \mathcal{E}_{ess}$, there is an integer $M>0$ such
that
 $f^{-M}(E_1\cup\cdots\cup E_m)\cap P_f=\emptyset$. If $\limsup d(f^n(z),P_f)=0$, then there
exists a sequence of integers $\{\ell_j\}$ such that
$d(f^{\ell_j}(z),(E_1\cup\cdots\cup E_m)\cap P_f)\rightarrow0$ as
$j\rightarrow\infty$. It follows that $f^{\ell_j-M}(z)\in
f^{-M}(E_1\cup\cdots\cup E_m)$ for all large $j$. Since the boundary
of each component of $f^{-M}(E_1\cup\cdots\cup E_m)$ is contained in
the grand orbits of the Herman rings of $f$,  there is a number
$\epsilon(z)>0$ such that $d(f^{\ell_j-M}(z), P_f)\geq \epsilon(z)$
for all large $j$. But this contradicts the assumption that $\limsup
d(f^n(z),P_f)=0$. So in this case, we also have $\limsup
d(f^n(z),P_f)>0$.

% and $z$ is not a Lebesgue density point by Theorem \ref{5c}.

Thus, for  any $z\in J(f)-\cup_{k\geq0}f^{-k}(J_1\cup\cdots\cup
J_n)$, we have $$\limsup d(f^n(z),P_f)>0.$$  It follows from Theorem
\ref{b5c} that the Lebesgue measure of
$J(f)-\cup_{k\geq0}f^{-k}(J_1\cup\cdots\cup J_n)$ is zero.
% contains
%no Lebesgue density point, or equivalently,
%$$Leb(J(f)-\cup_{k\geq0}f^{-k}(J_1\cup\cdots\cup
%J_n))=0.$$
This means $Leb(J(f))=0$ if and only if for each $k\in [1,n]$,
$Leb(J_k)=0$ (here, we use $Leb$ to denote the Lebesgue measure).

 Suppose that $J(f)$ carries an invariant line field. That is,
there is a measurable Beltrami differential $\mu$ supported on a
positive measure subset $E$ of $J(f)$ such that $f^*\mu=\mu$ a.e,
and $|\mu|=1$ on $E$. Let $\mu_k=\mu|_{J_k}$ for $k\in [1,n]$. It
follows from the above argument that there  exists $\ell\in [1,n]$
with $Leb(J_\ell\cap E)>0$.  Then the relation
$(f^{p_\ell}|_{U_\ell})^*\mu_\ell=\mu_\ell$ implies that $\mu_\ell$
is an invariant line field of $f^{p_\ell}|_{U_\ell}$. Conversely, suppose that $\mu_\ell$ is an invariant line field of
$f^{p_\ell}|_{U_\ell}$, then  the Beltrami differential defined by
 $$\mu=\mu_\ell+\sum_{k\geq0} ((f^{k+1})^*\mu_\ell-(f^{k})^*\mu_\ell)$$  is an invariant
line field  of $f$.    \hfill $\Box$

\subsection{Proof of Theorem \ref{b3-11}} \label{3-4-3}

 First, we need some lemmas.

\begin{lem}[Q.c-equivalence implies q.c-conjugacy]\label{b5d}   Let $(f,P)$ and $(g,Q)$ be two
 HST rational maps. If  $(f,P)$ and
 $(g,Q)$ are q.c-equivalent via a
pair of quasi-conformal maps $(\phi_0,\phi_1)$, then there is a  quasi-conformal map $\phi$, holomorphic in the
Fatou set $F(f)$ (probably empty), such that  $\phi f=g \phi$.
\end{lem}
\begin{proof}
We first deal with the case $J(f)=\mathbb{\overline{C}}$. In that case, $(f,P)$ is postcritically finite. If $(f,P)$ is not a  Latt\`es map, then $(f,P)$ and $(g,Q)$ are
M\"obius conjugate  by Thurston's theorem. If   $(f,P)$ is a Latt\`es map, then there is a sequence of quasiconformal  maps $\phi_k$ such that  $\phi_k f=g \phi_{k+1}$ and $\phi_k$ is isotopic to
$\phi_{k+1}$ rel $f^{-k}(P)$. Since all $\phi_k$ have bounded dilatations and $\overline{\cup_{k\geq0}f^{-k}(P)}=J(f)=\mathbb{\overline{C}}$, the sequence $\phi_k$ converges to
a quasi-conformal map which is in fact a  conjugacy between  $(f,P)$ and $(g,Q)$.

In the following, we assume $J(f)\neq\mathbb{\overline{C}}$. By the definition of
q.c-equivalence,  $\phi_0$ and $\phi_1$ are holomorphic and
identical
 in the union of all rotation domains $R_f$ of $f$ (if any).
If $f$ has a superattracting cycle $z_0\mapsto z_1\mapsto \cdots
\mapsto z_{p-1}\mapsto z_p=z_0$, then we can modify $\phi_0$ and
$\phi_1$ such that they are holomorphic and identical near the
cycle. The modification is as follows:

First, note that for any $0\leq i < p$,
$\phi_0(z_i)(=\phi_1(z_i))$ is a superattracting point of $g$. We
can choose a neighborhood $U_i$ of $z_i$ (resp. $V_i$ of
$\phi_0(z_i)$), a B\"ottcher coordinate $B_i^f: U_i\rightarrow
\mathbb{D}$ (resp. $B_{i}^g: V_{i}\rightarrow \mathbb{D}$), such
that the following diagram commutes: $$ \xymatrix{&
 U_i \ar[r]^{B_i^f}
\ar[d]_{f} &\mathbb{D} \ar[d]^{z\mapsto z^{d_i}}
 &  V_{i}\ar[l]_{B_{i}^g}\ar[d]^{g}\\
&  U_{i+1}\ar[r]_{B_{i+1}^f}
 &\mathbb{D}  &
 V_{i+1} \ar[l]^{B_{i+1}^g}
 }
$$
where $d_i$ is the local degree of $f$ at $z_i$. By suitable choices
of the neighborhoods $U_i$ and  the B\"ottcher coordinates, we may
assume that  $\phi_0$ and $\phi_1$ satisfy
$\phi_0|_{U_i}=\phi_1|_{U_i}=(B_i^g)^{-1}\circ B_i^f $. A suitable
modification elsewhere guarantees  $\phi_0 f=g \phi_1$.

In this way,  $\phi_0$ and $\phi_1$ can be made holomorphic in a
neighborhood $N_{SA}$ of all superattracting cycles of $f$ (if any).
 Then we construct a sequence of q.c maps $\{\phi_k;k\geq0\}$
by $\phi_k f=g \phi_{k+1}$ so that $\phi_k$ is isotopic to
$\phi_{k+1}$ rel $f^{-k}(P\cup N_{SA})$. Since $\overline{\cup_{k\geq0}f^{-k}(P\cup N_{SA})}=\overline{\mathbb{C}}$, the sequence $\phi_k$ has a
unique limit $\phi$, holomorphic in
  $\cup_{k\geq0}f^{-k}(R_f\cup N_{SA})=F(f)$, as required.
\end{proof}

Now let $M_1(J(f),f)$ be the space of invariant line fields carried by $(f,P)$ (we define $M_1(J(f),f)$ to be  $\{0d\overline{z}/dz\}$ if $(f,P)$ carries no
invariant line field). It's known
 from McMullen and Sullivan \cite{McS} that $M_1(J(f),f)$  is either a single point or  a finite-dimensional polydisk.
From Lemma \ref{b5d}, we have immediately:

\begin{lem} \label{b5e} $M_{qc}(f,P)\cong M_1(J(f),f)$.
\end{lem}
\begin{proof}  By Lemma \ref{b5d},  $M_{qc}(f,P)$ is the space of all
 rational maps (up to M\"obius conjugation) q.c-conjugate to $(f,P)$.  Moreover, each element of $M_{qc}(f,P)$ corresponds to
  a unique quasiconformal map $\phi$ up to post-composition a M\"obius map so that $\phi$ is holomorphic in the Fatou set $F(f)$. This induces a
   unique Beltrami differential   $\mu_\phi\in M_1(J(f),f)$.
  The converse is immediate.
% and $(g,Q)$ are q.c-equivalent, then they are q.c-conjugace via some q.c-map $\phi$. The
% Beltrami of $\phi$ induce an element $\mu_\phi\in  M_1(J,f)$. The
% It's obvious that if  $J(R)$ carries an invariant line
%field, then the rational realization $(R,Q)$ is not unique up to
%M\"obius conjugation. Conversely, let $(R_1,Q_1)$ be another
%rational realization of $(f,P)$. Then $(R_1,Q_1)$ and $(R,Q)$ are
%q.c-equivalent. It follows from Lemma \ref{b5d} that $(R,Q)$ and
%$(R_1,Q_1)$ are conjugate via some q.c map $\phi$, which is
%holomorphic in $F(R)$. This implies $R^*\mu_\phi=\mu_\phi$ on
%$J(R)$. Since $J(R)$ carries no invariant line field, $\mu_\phi=0$
%almost everywhere on $\mathbb{\overline{C}}$. Thus $\phi$ is a
%%M\"obius transformation.
\end{proof}

\noindent{\it Proof of Theorem \ref{b3-11}.}
 Let
$(f,P)$ be a Herman rational map and $(f^{p_i},U_i,V_i), i\in [1,n]$ be  all its  HS renormalizations
defined in Theorem \ref{b5a3}, whose straightening maps are denoted by $(h_i,P_i)$ respectively. We may renumber them so that
$(h_i,P_i)_{1\leq i \leq m}$ are Siegel rational maps and the rest are Thurston rational maps.
Let $M_1(J_i, f^{p_i}|_{U_i})$ be the space of invariant line fields carried by $(f^{p_i},U_i,V_i)$. By Theorem \ref{b5a3},
 we have $M_1(J_i, f^{p_i}|_{U_i})\cong M_1(J(h_i),h_i)$. By Lemma \ref{b3aa}, none of $(h_i,P_i)_{m<i \leq n}$ is a Latt\`es map, thus
 $M_1(J(h_i),h_i)$ is a singleton.  To prove Theorem \ref{b3-11}, it suffices to show
  $$M_{1}(J(f),f)\cong M_1(J_1, f^{p_1}|_{U_1})\times\cdots\times M_1(J_m, f^{p_m}|_{U_m}).$$

  We define $\Phi:M_{1}(J(f),f)\rightarrow M_1(J_1, f^{p_1}|_{U_1})\times\cdots\times M_1(J_m, f^{p_m}|_{U_m})$ by
  $\Phi(\mu)=(\mu|_{J_1},\cdots,\mu|_{J_m})$. Its inverse is given by
  $$\Phi^{-1}(\mu_{1},\cdots,\mu_{m})=\sum_{1\leq \ell\leq m}\Big(\mu_\ell+\sum_{k\geq0} ((f^{k+1})^*\mu_\ell-(f^{k})^*\mu_\ell)\Big).
  $$

Thus $\Phi$ is a isomorphism.
\hfill $\Box$

\section{Corollaries}\label{3-x}

In this section, we shall prove Theorems \ref{b3-111} and \ref{b3-kk}.

\noindent{\it Proof of Theorem \ref{b3-111}.}
 Let $(f,P)$ be  an  unobstructed  Herman map. By Theorem \ref{b3-1}, there are finitely many  unobstructed  Siegel maps and Thurston maps,
 say $(h_k,P_k)_{1 \leq k\leq n}$,  such that the rational realization of  $(f,P)$ depends on that of  $(h_k,P_k)_{1 \leq k\leq n}$.
 We may renumber them so that the first $m$ maps are Siegel maps and the rest are Thurston maps. The decomposition procedure implies that
  $m\leq n_{RD}(f)+2n_{RA}(f)$. By Lemma \ref{b3aa}, none of $(h_k,P_k)_{m<k \leq n}$ has an
orbifold with signature $(2,2,2,2)$. Thus by Thurston's theorem, all $(h_k,P_k)_{m<k\leq n}$ have rational realizations.
The rational realization of  $(f,P)$ actually  depends on that of  $(h_k,P_k)_{1 \leq k\leq m}$.
\hfill $\Box$

  To prove Theorem \ref{b3-kk}, we need the following

\begin{thm}[Characterization of Siegel disk, \cite{Z2}]\label{b3-kkk}
 Let $(S,Z)$ be a Siegel map.
    Suppose that  $(S,Z)$ has only one fixed rotation disk $D$ of bounded type rotation number and $Z\setminus \overline{D}$ is a finite set.
  Then $(S,Z)$ is
c-equivalent to a rational function $(R,Q)$ if and only if
$(S,Z)$ has no Thurston obstructions.
The rational function $(R,Q)$    is unique up to M\"obius conjugation.
\end{thm}

\noindent{\it Proof of Theorem \ref{b3-kk}.}   We may assume that $(f,P)$ is unobstructed.
  By Theorem \ref{b3-111}, there are two Siegel maps  $(h_1,P_1), (h_2,P_2)$  such that the rational realization of  $(f,P)$ depends on that of   $(h_1,P_1), (h_2,P_2)$. Each of $(h_k,P_k)$ has only one fixed rotation disk $D_k$ of bounded type rotation number and $P_k\setminus \overline{D_k}$ is a finite set. By Theorem \ref{b3-kkk} and Theorem \ref{b3-111}, $(f,P)$ has an rational realization. The converse follows from Theorem \ref{b4b}.

   The rigidity part: Note  that  any two rational realizations of $(f,P)$  are c-equivalent. It's known in \cite{Z1} that for any rational map,
    the boundary of a Siegel disk with bounded type rotation number is  a quasi-circle. This implies any two rational realizations of $(f,P)$  can be made q.c-equivalent.
    So $M_{top}(f,P)\cong M_{qc}(f,P)$.
    It follows from Theorems  \ref{b3-11} and  \ref{b3-kkk} that % \ref{b5a3} and \ref{b3-kkk} that any $(R,Q)\in M_{qc}(f,P)$  has an Julia set of Lebesgue measure zero, thus
    $M_{qc}(f,P)$ is  a singleton. So dose $M_{top}(f,P)$.
\hfill $\Box$

% This is the end of the last section.

% Finally we create the bibliography or list of references.

% Every LaTeX document must end with \end{document}.

\end{document}